\documentclass[3p,11pt]{elsarticle}
\usepackage{amssymb}
\usepackage{amsmath}
\usepackage{amsthm}
\usepackage{tikz}
\usetikzlibrary{math}
\usepackage{pgfplots}

\DeclareMathAlphabet{\mathcal}{OMS}{cmsy}{m}{n}

\makeatletter
\def\ps@pprintTitle{%
 \let\@oddhead\@empty
 \let\@evenhead\@empty
 \def\@oddfoot{\centerline{\thepage}}%
 \let\@evenfoot\@oddfoot}
\makeatother

\newcommand{\bbC}{\mathbb{C}}
\newcommand{\bbF}{\mathbb{F}}
\newcommand{\bbH}{\mathbb{H}}

\newcommand{\bbR}{\mathbb{R}}
\newcommand{\bbT}{\mathbb{T}}
\newcommand{\bbZ}{\mathbb{Z}}

\newcommand{\bfA}{\mathbf{A}}
\newcommand{\bfB}{\mathbf{B}}

\newcommand{\bfI}{\mathbf{I}}
\newcommand{\bfJ}{\mathbf{J}}

\newcommand{\bfP}{\mathbf{P}}

\newcommand{\bfU}{\mathbf{U}}
\newcommand{\bfV}{\mathbf{V}}
\newcommand{\bfW}{\mathbf{W}}
\newcommand{\bfx}{\mathbf{x}}
\newcommand{\bfX}{\mathbf{X}}
\newcommand{\bfy}{\mathbf{y}}

\newcommand{\bfzero}{\boldsymbol{0}}
\newcommand{\bfdelta}{\boldsymbol{\delta}}

\newcommand{\bfGamma}{\boldsymbol{\Gamma}}
\newcommand{\bfchi}{\boldsymbol{\chi}}
\newcommand{\bfphi}{\boldsymbol{\varphi}}
\newcommand{\bfPhi}{\boldsymbol{\Phi}}
\newcommand{\bfpsi}{\boldsymbol{\psi}}
\newcommand{\bfPsi}{\boldsymbol{\Psi}}

\newcommand{\bfxi}{\boldsymbol{\xi}}
\newcommand{\bfXi}{\boldsymbol{\Xi}}

\newcommand{\calB}{\mathcal{B}}
\newcommand{\calD}{\mathcal{D}}

\newcommand{\calG}{\mathcal{G}}
\newcommand{\calH}{\mathcal{H}}

\newcommand{\calK}{\mathcal{K}}

\newcommand{\calN}{\mathcal{N}}

\newcommand{\calR}{\mathcal{R}}
\newcommand{\calS}{\mathcal{S}}

\newcommand{\calU}{\mathcal{U}}
\newcommand{\calV}{\mathcal{V}}

\newcommand{\rmc}{\mathrm{c}}

\newcommand{\rmi}{\mathrm{i}}
\newcommand{\rmj}{\mathrm{j}}
\newcommand{\rmk}{\mathrm{k}}
\newcommand{\rmN}{\mathrm{N}}

\newcommand{\rms}{\mathrm{s}}
\newcommand{\rmT}{\mathrm{T}}

\newcommand{\Tr}{\operatorname{Tr}}

\newcommand{\Orb}{{\operatorname{Orb}}}
\newcommand{\dist}{\operatorname{dist}}
\newcommand{\rank}{\operatorname{rank}}
\newcommand{\BIBD}{{\operatorname{BIBD}}}
\newcommand{\Span}{\operatorname{span}}

\newcommand{\real}{\operatorname{Re}}

\newcommand{\DF}{{\operatorname{DF}}}
\newcommand{\DS}{{\operatorname{DS}}}
\newcommand{\DDS}{{\operatorname{DDS}}}

\newcommand{\TFF}{{\operatorname{TFF}}}
\newcommand{\ETF}{{\operatorname{ETF}}}

\newcommand{\ECTFF}{{\operatorname{ECTFF}}}
\newcommand{\EITFF}{{\operatorname{EITFF}}}

\newcommand{\Fro}{\mathrm{Fro}}

\newcommand{\abs}[1]{|{#1}|}

\newcommand{\biggparen}[1]{\biggl({#1}\biggr)}

\newcommand{\bigbracket}[1]{\bigl[{#1}\bigr]}

\newcommand{\biggbracket}[1]{\biggl[{#1}\biggr]}

\newcommand{\set}[1]{\{{#1}\}}

\newcommand{\norm}[1]{\|{#1}\|}

\newcommand{\ip}[2]{\langle{#1},{#2}\rangle}

\newtheorem{theorem}{Theorem}[section]

\newtheorem{corollary}[theorem]{Corollary}

\theoremstyle{definition}
\newtheorem{definition}[theorem]{Definition}

\begin{document}
\begin{frontmatter}
\title{Certifying the novelty of equichordal tight fusion frames}

\author{Matthew Fickus}
\ead{Matthew.Fickus@gmail.com}
\author{Benjamin R.~Mayo}
\author{Cody E.~Watson}

\address{Department of Mathematics and Statistics, Air Force Institute of Technology, Wright-Patterson AFB, OH 45433}

\begin{abstract}
An equichordal tight fusion frame (ECTFF) is a finite sequence of equi-dimensional subspaces of a finite-dimensional Hilbert space that achieves equality in Conway, Hardin and Sloane's simplex bound.
Every ECTFF is a type of optimal Grassmannian code,
being a way to arrange a given number of members of a Grassmannian so that the minimal chordal distance between any pair of them is as large as possible.
Any nontrivial ECTFF has both a Naimark complement and spatial complement which themselves are ECTFFs.
It turns out that whenever the number of subspaces is at least five,
taking iterated alternating Naimark and spatial complements of one ECTFF yields an infinite family of them with distinct parameters.
This makes it challenging to certify the novelty of any recently discovered ECTFF:
how can one guarantee that it does not arise from any previously known construction in such a Naimark-spatial way?
In this paper, we propose a solution to this problem,
showing that any ECTFF is a member of a Naimark-spatial family originating from either a trivial ECTFF or one with unique "minimal" parameters.
In the latter case, if its minimal parameters do not match those of any previously known ECTFF, it is certifiably new.
As a proof of concept,
we then use these ideas to certify the novelty of some ECTFFs arising from a new method for constructing them from difference families for finite abelian groups.
This method properly generalizes King's construction of ECTFFs from semiregular divisible difference sets.
\end{abstract}

\begin{keyword}
simplex bound \sep tight fusion frame \sep equichordal \sep equi-isoclinic
\MSC[2020] 42C15
\end{keyword}
\end{frontmatter}

\section{Introduction}

Let $\calH$ be a $D$-dimensional Hilbert space over the field $\bbF$ which is either $\bbR$ or $\bbC$.
The \textit{chordal distance}~\cite{ConwayHS96} and \textit{spectral distance}~\cite{DhillonHST08} between two $R$-dimensional subspaces $\calU_1$ and $\calU_2$ of $\calH$ are defined, in terms of their corresponding rank-$R$ projections $\bfP_1$ and $\bfP_2$, as
\begin{equation}
\label{eq.chordal and spectral distance}
\dist_{\rmc}(\calU_1,\calU_2)
:=\tfrac1{\sqrt{2}}\norm{\bfP_1-\bfP_2}_\Fro,
\quad
\dist_{\rms}(\calU_1,\calU_2):=\sqrt{1-\norm{\bfP_1\bfP_2}_2^2},
\end{equation}
respectively.
For any subspaces $\set{\calU_n}_{n=1}^N$ of $\calH$,
each of dimension $R$, it is known~\cite{ConwayHS96,DhillonHST08} that
\begin{equation}
\label{eq.chordal and simplex packing bounds}
\smash{\min_{n_1\neq n_2}\dist_{\rms}(\calU_{n_1},\calU_{n_2})
\leq\tfrac1{\sqrt{R}}\min_{n_1\neq n_2}\dist_{\rmc}(\calU_{n_1},\calU_{n_2})
\leq\sqrt{\tfrac{N}{N-1}\tfrac{D-R}{D}}.}
\end{equation}
The right-hand inequality in~\eqref{eq.chordal and simplex packing bounds}, dubbed the \textit{simplex bound}, holds with equality if and only if $\set{\calU_n}_{n=1}^N$ is an \textit{equichordal tight fusion frame} (ECTFF) for $\calH$, namely a \textit{tight fusion frame} (TFF)---having $\sum_{n=1}^N\bfP_n=A\bfI$ for some $A>0$---for which $\dist_{\rmc}(\calU_{n_1},\calU_{n_2})$ is constant over all $n_1\neq n_2$~\cite{ConwayHS96}.
Every ECTFF is a type of optimal \textit{Grassmannian code} with respect to the chordal distance, being an optimal packing of $N$ members of the \textit{Grassmannian (space)} that consists of all $R$-dimensional subspaces of $\calH$.
Meanwhile, equality holds throughout~\eqref{eq.chordal and simplex packing bounds} if and only if $\set{\calU_n}_{n=1}^N$ is an \textit{equi-isoclinic tight fusion frame} (EITFF) for $\calH$,
namely a TFF for $\calH$ whose subspaces are \textit{equi-isoclinic}~\cite{LemmensS73b}, meaning there exists $\lambda\geq0$ such that $\bfP_{n_1}\bfP_{n_2}\bfP_{n_1}=\lambda\bfP_{n_1}$ for all $n_1\neq n_2$.
Every EITFF is an ECTFF,
and is an optimal Grassmannian packing with respect to the spectral distance.

When $R=1$, every ECTFF for $\calH$ is an EITFF for $\calH$ and equates to an \textit{equiangular tight frame} (ETF) for $\calH$, namely a sequence $\set{\bfphi_n}_{n=1}^N$ of unit vectors in $\calH$ that achieves equality in the \textit{Welch bound}
\smash{$\max_{n_1\neq n_2}\abs{\ip{\bfphi_{n_1}}{\bfphi_{n_2}}}\geq[\frac{N-D}{D(N-1)}]^{\frac12}$},
and so has minimal \textit{coherence}~\cite{Welch74,StrohmerH03}.
More generally, any EITFF yields a dictionary with minimal \textit{block coherence}~\cite{DhillonHST08,EldarKB10}.
ETFs and EITFFs arise in compressed sensing~\cite{BajwaCM12,BandeiraFMW13,CalderbankTX15}
and quantum information theory~\cite{Zauner99,RenesBSC04,FuchsHS17}.
The broader study of ECTFFs in general has also attracted notably diverse interest~\cite{ConwayHS96,Zauner99,KutyniokPCL09,CohnKM16}.

We refer to a(n EI/EC)TFF for a complex $D$-dimensional Hilbert space that consists of $N$ subspaces, each of dimension $R$, as a(n $\operatorname{EI/EC)TFF}(D,N,R)$.
Every $\ECTFF(D,N,R)$ $\set{\calU_n}_{n=1}^N$ with $D\neq NR$ has a \textit{Naimark complement} which is an $\ECTFF(NR-D,N,R)$,
and which is equi-isoclinic if $\set{\calU_n}_{n=1}^N$ is.
Every $\ECTFF(D,N,R)$ $\set{\calU_n}_{n=1}^N$ with $D\neq R$ has a \textit{spatial complement} which is an $\ECTFF(D,N,D-R)$.
As we explain,
taking iterated alternating Naimark and spatial complements (\`{a} la~\cite{CasazzaFMWZ11}) of any $\ECTFF(D,N,R)$ with $N\geq 5$ generates an infinite sequence of ECTFFs with distinct parameters.
This is great news for anyone who wants Grassmannian codes:
the discovery of even a single ECTFF with new parameters implies the existence of a new infinite family of them.
That said, this same fact makes it nontrivial for the community to determine whether a newly discovered ECTFF is genuinely novel:
how does one certify, for example, that its $(D,N,R)$ parameters do no match those of any one obtained from a previously-known construction in this Naimark-spatial way?
In this paper, we offer a solution to this problem.
The key idea is to study a simple function of $(D,N,R)$ which is invariant with respect to both complements,
namely
\begin{equation}
\label{eq.invariant}
f:\bbR^3\rightarrow\bbR,
\quad
f(D,N,R)
:=DNR-D^2-NR^2.
\end{equation}
(See (6.3) of~\cite{CohnKM16} for a function of $(D,N,R)$ that previously arose in the study of the ECTFFs and only differs from $f$ by a function of $N$.)
Since $f(D,N,R)=D(NR-D)-NR^2=NR(D-R)-D^2$,
it is invariant with respect to both the \textit{Naimark} and \textit{spatial involutions}, namely the functions
$\nu,\sigma:\bbR^3\rightarrow\bbR^3$ defined by $\nu(D,N,R):=(NR-D,N,R)$  and $\sigma(D,N,R):=(D,N,D-R)$, respectively.
As we explain, this means that taking iterated alternating Naimark and spatial complements of any $\ECTFF(D,N,R)$ with $f(D,N,R)>0$ yields an infinite \textit{orbit} of $(D,N,R)$ triples whose $(D,R)$ pairs all lie on a common connected component of a hyperbola.
It turns out that such an orbit possesses a unique point $(D_0,N,R_0)$ that is \textit{minimal} in the sense that
\begin{equation}
\label{eq.minimal}
0<D_0\leq NR_0-D_0
\ \text{ and }\
0<R_0\leq D_0-R_0,
\ \text{ i.e., }\
0<\tfrac{2D_0}N\leq R_0\leq\tfrac{D_0}2.
\end{equation}
It also turns out that the only ones of these ECTFFs which might be equi-isoclinic are those with parameters $(D_0,N,R_0)$ and their Naimark complements.
We shall also fully settle the existence of (real and/or equichordal and/or equi-isoclinic) $\TFF(D,N,R)$ with $f(D,N,R)\leq 0$.
In the end, an $\ECTFF(D,N,R)$ will be certifiably novel if $f(D,N,R)>0$ and the minimal point of its orbit is not equal to those of previously discovered constructions.

In the next section we establish notation and review previously known results and concepts that we will need here.
In Section~3, we rigorously formulate and prove the claims we made above about~\eqref{eq.invariant} and~\eqref{eq.minimal};
see Theorems~\ref{thm.N=2,3}--\ref{thm.summary}.
In Section~4, we introduce a new method of constructing ECTFFs from \textit{difference families} for finite abelian groups;
see Theorem~\ref{thm.ECTFF from DF} and Corollary~\ref{cor.ECTFFs from known DFs}.
This method properly generalizes King's construction of ECTFFs from \textit{semiregular divisible difference sets}~\cite{King16},
which itself generalizes a construction of \textit{harmonic mutually unbiased bases} from \textit{semiregular relative difference sets}~\cite{GodsilR09} and that of \textit{harmonic ETFs} from \textit{difference sets}~\cite{Konig99,StrohmerH03,XiaZG05,DingF07}.
In Section~5, we then certify the novelty of some of these ``harmonic ECTFFs" using the theory of Section~3.
To do so, we compare the minimal points of their orbits against those of previously known ECTFFs; see Theorems~\ref{thm.EITFF existence} and~\ref{thm.ECTFF existence}.
(A thorough review of the ECTFF construction literature is given in that section.)
Along the way, we revisit the TFF \textit{existence test} of~\cite{CasazzaFMWZ11},
showing in Theorem~\ref{thm.TFF existence} that a real $\TFF(D,N,R)$ exists whenever a $\TFF(D,N,R)$ does.

\section{Preliminaries}

As above,
let $\calH$ be a $D$-dimensional Hilbert space over $\bbF$.
Let $\calN$ be a finite indexing set of cardinality $N>1$.
A sequence $\set{\calU_n}_{n\in\calN}$ of $R$-dimensional subspaces of $\calH$ is a TFF for $\calH$ if the corresponding rank-$R$ projections $\set{\bfP_n}_{n\in\calN}$ satisfy $\sum_{n\in\calN}\bfP_n=A\bfI$ for some $A>0$.
In this case,
$NR
\geq\rank(\sum_{n\in\calN}\bfP_n)
=\rank(A\bfI)
=D$ and $NR=\Tr(\sum_{n\in\calN}\bfP_n)=\Tr(A\bfI)=AD$,
and so we necessarily have that $A=\frac{NR}{D}\geq 1$.
Since tightness is preserved by unitary transformations,
one can without loss of generality assume $\calH$ is $\bbF^D$
(though it is sometimes convenient to do otherwise~\cite{FickusIJK20}).
As such, the existence of a given TFF only depends on $\bbF$ and the values of $D$, $N$ and $R$, and moreover if a real TFF exists then so does a complex TFF with the same parameters.
Because of this, we use ``$\TFF(D,N,R)$" to denote a complex TFF with such parameters,
and say it is \textit{real} if $\calH$ can be chosen to be $\bbR^D$
(that is, if either $\calH$ is natively real or $\calH$ is complex but there exists a unitary transformation $\bfU:\bbC^\calD\rightarrow\calH$ such that every matrix $\bfU^*\bfP_n\bfU$ has all real entries).

For any sequence $\set{\calU_n}_{n\in\calN}$ of $R$-dimensional subspaces of $\calH$,
a direct calculation gives
\begin{equation}
\label{eq.fusion frame potential}
0
\leq\Tr\biggbracket{\biggparen{\,\sum_{n\in\calN}\bfP_n-\tfrac{NR}{D}\bfI}^2}
=\sum_{n_1\in\calN}\sum_{n_2\neq n_1}\Tr(\bfP_{n_1}\bfP_{n_2})
-\tfrac{NR(NR-D)}{D}.
\end{equation}
Equality holds above if and only if $\set{\calU_n}_{n\in\calN}$ is a TFF for $\calH$.
Here, the chordal distance~\eqref{eq.chordal and spectral distance} satisfies
\begin{equation}
\label{eq.chordal distance}
[\dist_{\rmc}(\calU_{n_1},\calU_{n_2})]^2
=\tfrac12\Tr[(\bfP_{n_1}-\bfP_{n_2})^2]
=R-\Tr(\bfP_{n_1}\bfP_{n_2})
\end{equation}
for any $n_1,n_2\in\calN$.
In particular, $\Tr(\bfP_{n_1}\bfP_{n_2})$ is real,
and so we can rearrange and continue~\eqref{eq.fusion frame potential} as
\begin{equation}
\label{eq.chordal Welch}
\tfrac{R(NR-D)}{D(N-1)}
\leq\tfrac1{N(N-1)}
\sum_{n_1\in\calN}\sum_{n_2\neq n_1}\Tr(\bfP_{n_1}\bfP_{n_2})
\leq\max_{n_1\neq n_2}\Tr(\bfP_{n_1}\bfP_{n_2})
=R-\min_{n_1\neq n_2}[\dist_{\rmc}(\calU_{n_1},\calU_{n_2})]^2.
\end{equation}
Here, equality holds throughout if and only if $\set{\calU_n}_{n\in\calN}$ is a TFF for $\calH$ that is also \textit{equichordal} in the sense that $\dist_{\rmc}(\calU_{n_1},\calU_{n_2})$ is constant over all $n_1\neq n_2$,
namely when $\set{\calU_n}_{n\in\calN}$ is an ECTFF for $\calH$.
Rearranging~\eqref{eq.chordal Welch} gives the right-hand inequality of~\eqref{eq.chordal and simplex packing bounds}, which is called the \textit{simplex bound} in~\cite{ConwayHS96} since $\set{\calU_n}_{n\in\calN}$ is an ECTFF for $\calH$ if and only if $\set{\bfP_n-\tfrac{R}{D}\bfI}_{n\in\calN}$ is a regular simplex in the real Hilbert space of self-adjoint operators on $\calH$ (equipped with the Frobenius inner product).

A similar argument applies if the chordal distance on the Grassmannian is replaced with the spectral distance~\eqref{eq.chordal and spectral distance}.
(We caution that the spectral distance is not a metric on the Grassmannian,
since $\dist_{\rms}(\calU_1,\calU_2)=0$ if and only if $\calU_1$ and $\calU_2$ intersect nontrivially, which can occur even when they are distinct.)
Here, for any $n_1\neq n_2$, the fact that $\bfP_{n_1}\bfP_{n_2}$ has rank at most $R$ implies
\begin{equation}
\label{eq.chordal spectral relation}
\Tr(\bfP_{n_1}\bfP_{n_2})
=\Tr(\bfP_{n_2}^*\bfP_{n_1}^*\bfP_{n_1}^{}\bfP_{n_2}^{})
=\norm{\bfP_{n_1}\bfP_{n_2}}_\Fro^2
\leq R\norm{\bfP_{n_1}\bfP_{n_2}}_2^2
\end{equation}
where equality holds if and only if the $R$ largest singular values of $\bfP_{n_1}\bfP_{n_2}$ are equal,
namely if and only if $\calU_{n_1}$ and $\calU_{n_2}$ are \textit{isoclinic} in the sense that $\bfP_{n_1}\bfP_{n_2}\bfP_{n_1}=\lambda_{n_1,n_2}\bfP_{n_1}$ for some $\lambda_{n_1,n_2}\geq0$.
Here, note that combining~\eqref{eq.chordal and spectral distance},
\eqref{eq.chordal distance} and~\eqref{eq.chordal spectral relation} gives
$\dist_{\rms}(\calU_{n_1},\calU_{n_2})
\leq R^{-\frac12}\dist_{\rmc}(\calU_{n_1},\calU_{n_2})$, yielding the left-hand inequality of~\eqref{eq.chordal and simplex packing bounds}.
Moreover, \eqref{eq.chordal spectral relation} provides an alternative way of continuing~\eqref{eq.fusion frame potential}
\begin{equation}
\label{eq.spectral Welch}
\tfrac{NR-D}{D(N-1)}
\leq\tfrac1{N(N-1)}\sum_{n_1\in\calN}\sum_{n_2\neq n_1}\norm{\bfP_{n_1}\bfP_{n_2}}_2^2
\leq\max_{n_1\neq n_2}\norm{\bfP_{n_1}\bfP_{n_2}}_2^2
=1-\min_{n_1\neq n_2}[\dist_{\rms}(\calU_{n_1},\calU_{n_2})]^2.
\end{equation}
Equality holds throughout here (or equivalently, throughout~\eqref{eq.chordal and simplex packing bounds}) if and only if $\set{\calU_n}_{n\in\calN}$ is an EITFF for $\calH$, namely a TFF for $\calH$ which is equi-isoclinic
(meaning any pair of these subspaces are isoclinic and $\lambda_{n_1,n_2}$ is independent of $n_1\neq n_2$).
Note that in this case~\eqref{eq.chordal spectral relation} becomes $\Tr(\bfP_{n_1}\bfP_{n_2})=R\lambda$ for all $n_1\neq n_2$, implying every EITFF for $\calH$ is necessarily an ECTFF for $\calH$.

The \textit{spatial complement}~\cite{CasazzaFMWZ11} of a $\TFF(D,N,R)$ $\set{\calU_n}_{n\in\calN}$ for $\calH$ with $D\neq R$ is the sequence $\set{\calU_n^\perp}_{n\in\calN}$ of its orthogonal complements.
It is a $\TFF(D,N,D-R)$ for $\calH$ (and is real if $\set{\calU_n}_{n\in\calN}$ is real) since its projections $\set{\bfI-\bfP_n}_{n\in\calN}$ satisfy
$\sum_{n\in\calN}(\bfI-\bfP_n)
=(N-\frac{NR}{D})\bfI$.
The spatial complement of an $\ECTFF(D,N,R)$ for $\calH$ with $D\neq R$ is an $\ECTFF(D,N,D-R)$ for $\calH$ since
\begin{equation*}
[\dist_{\rmc}(\calU_{n_1}^\perp,\calU_{n_2}^\perp)]^2
=\tfrac12\norm{(\bfI-\bfP_{n_1})-(\bfI-\bfP_{n_2})}_\Fro^2
=\tfrac12\norm{\bfP_{n_1}-\bfP_{n_2}}_\Fro^2
=[\dist_{\rmc}(\calU_{n_1},\calU_{n_2})]^2.
\end{equation*}
In contrast, the spatial complement of an $\EITFF(D,N,R)$ is only an EITFF in the special case where $D=2R$;
to see this fact, and also understand the Naimark complement of a TFF,
it helps to first discuss some traditional finite frame theory.

Equip $\bbF^\calN:=\set{\bfx:\calN\rightarrow\bbF}$ with the inner product $\ip{\bfx_1}{\bfx_2}:=\sum_{n\in\calN}\overline{\bfx_1(n)}\bfx_2(n)$.
Throughout, our complex inner products are conjugate-linear in their first arguments.
In the special case where $\calN=[N]:=\set{n\in\bbZ: 1\leq n\leq N}$,
we denote $\bbF^\calN$ as simply $\bbF^N$, as usual.
Let $\set{\bfdelta_n}_{n\in\calN}$ denote the standard basis for $\bbF^\calN$,
having $\bfdelta_n(n):=1$ and $\bfdelta_n(n'):=0$ when $n\neq n'$.
When $\calD$ and $\calN$ are both finite sets,
we identify a linear map $\bfA:\bbF^\calN\rightarrow\bbF^\calD$ with a ``$\calD\times\calN$ matrix" $\bfA\in\bbF^{\calD\times\calN}$ (and vice versa) in the usual way, having $\bfA(d,n)=\ip{\bfdelta_d}{\bfA\bfdelta_n}$ for all $d\in\calD$, $n\in\calN$.

We sometimes regard a vector $\bfphi\in\calH$ as the operator $\bfphi:\bbF\rightarrow\calH$, $\bfphi(x)=x\bfphi$ whose adjoint is the linear functional $\bfphi^*:\calH\rightarrow\bbF$, $\bfphi^*\bfy=\ip{\bfphi}{\bfy}$.
The \textit{synthesis operator} of a sequence $\set{\bfphi_n}_{n\in\calN}$ of vectors in $\calH$ is $\bfPhi:\bbF^\calN\rightarrow\calH$,
$\bfPhi:=\sum_{n\in\calN}\bfphi_n^{}\bfdelta_n^*$.
Its adjoint is the \textit{analysis operator}
$\bfPhi^*:\calH\rightarrow\bbF^\calN$,
$\bfPhi^*=\sum_{n\in\calN}\bfdelta_n^{}\bfphi_n^{*}$.
Composing them gives the \textit{frame operator} $\bfPhi\bfPhi^*:\calH\rightarrow\calH$, $\bfPhi\bfPhi^*=\sum_{n\in\calN}\bfphi_n^{}\bfphi_n^*$ and the \textit{Gram matrix}
$\bfPhi^*\bfPhi
=\sum_{n_1\in\calN}\sum_{n_2\in\calN}\bfdelta_{n_1}^{}\ip{\bfphi_{n_1}^{}}{\bfphi_{n_2}^{}}\bfdelta_{n_2}^*$,
namely the $\calN\times\calN$ matrix with $(n_1,n_2)$th entry
$(\bfPhi^*\bfPhi)(n_1,n_2)=\ip{\bfphi_{n_1}^{}}{\bfphi_{n_2}^{}}$.
In the special case where $\calH=\bbF^\calD$ for some finite set $\calD$,
$\bfPhi$ is the $\calD\times\calN$ matrix which has $\bfphi_n$ as its $n$th column,
$\bfPhi^*$ is its $\calN\times\calD$ conjugate (Hermitian) transpose,
and $\bfPhi\bfPhi^*$ and $\bfPhi^*\bfPhi$ are their $\calD\times\calD$ and $\calN\times\calN$ products, respectively.
In general, every nontrivial positive semidefinite $\calN\times\calN$ matrix is the Gram matrix of some sequence $\set{\bfphi_n}_{n\in\calN}$ of vectors which is unique up to unitary transformations on its span.

A sequence $\set{\bfphi_n}_{n\in\calN}$ of vectors in $\calH$ is a \textit{tight frame} for $\calH$ if $\bfPhi\bfPhi^*=A\bfI$ for some $A>0$.
When this occurs, $\frac1A\bfPhi^*\bfPhi$ is an $\calN\times\calN$ rank-$D$ projection.
In this case, a sequence $\set{\bfpsi_n}_{n\in\calN}$ of vectors in some Hilbert space $\calK$ is a \textit{Naimark complement} of $\set{\bfphi_n}_{n\in\calN}$ if its Gram matrix $\bfPsi^*\bfPsi$ is a positive scalar multiple of the complementary projection $\bfI-\frac1A\bfPhi^*\bfPhi$.
When $N\neq D$, such a sequence $\set{\bfpsi_n}_{n\in\calN}$ exists,
is unique up to nonzero scalar multiples and unitary transformations,
and is a tight frame for its $(N-D)$-dimensional span.
In the special case where $\calH=\bbF^\calD$,
$\set{\bfphi_n}_{n\in\calN}$ is a tight frame for $\calH$ if and only if the rows of the $\calD\times\calN$ matrix $\bfPhi$ are nonzero, equal-norm and mutually orthogonal.
In this case, we can complete these rows to an equal-norm orthogonal basis for $\bbF^\calN$,
that is,
find a matrix $\bfPsi$ such that
$[\begin{smallmatrix}\bfPhi\\\bfPsi\end{smallmatrix}]$
is a scalar multiple of a unitary.
Here,
the columns $\set{\bfpsi_n}_{n\in\calN}$ of $\bfPsi$ form a Naimark complement of $\set{\bfphi_n}_{n\in\calN}$.

Note that a sequence $\set{\bfphi_n}_{n\in\calN}$ of unit-norm vectors in $\calH$ is a tight frame for $\calH$ if and only if the corresponding rank-$1$ projections $\set{\bfphi_n^{}\bfphi_n^*}_{n\in\calN}$ sum to a multiple of the identity.
From this perspective, tight fusion frames are generalizations of unit-norm tight frames.
But they can also be regarded as special cases of them,
and this leads to the notion of the Naimark complement of a TFF.
To elaborate,
let $\set{\calU_n}_{n\in\calN}$ be a $\TFF(D,N,R)$ for $\calH$ with corresponding projections $\set{\bfP_n}_{n\in\calN}$.
Fix some $R$-element index set $\calR$.
For each $n\in\calN$, let $\set{\bfphi_{n,r}}_{r\in\calR}$ be any orthonormal basis (ONB) for $\calU_n$,
and let $\bfPhi_n:\bbF^\calR\rightarrow\calH$ be its synthesis operator,
implying $\bfP_n=\bfPhi_n^{}\bfPhi_n^*=\sum_{r\in\calR}\bfphi_{n,r}^{}\bfphi_{n,r}^*$.
Then the concatenation $\set{\bfphi_{n,r}}_{(n,r)\in\calN\times\calR}$ of these $N$ subspace-ONBs is a (traditional) tight frame for $\calH$:
\begin{equation*}
\bfPhi\bfPhi^*
=\sum_{(n,r)\in\calN\times\calR}
\bfphi_{n,r}^{}\bfphi_{n,r}^*
=\sum_{n\in\calN}\sum_{r\in\calR}\bfphi_{n,r}^{}\bfphi_{n,r}^*
=\sum_{n\in\calN}\bfPhi_n^{}\bfPhi_n^*
=\sum_{n\in\calN}\bfP_n
=\tfrac{NR}{D}\bfI.
\end{equation*}
As such,
the $(\calN\times\calR)\times(\calN\times\calR)$ Gram matrix $\bfPhi^*\bfPhi$ of $\set{\bfphi_{n,r}}_{(n,r)\in\calN\times\calR}$ is a rank-$D$ projection scaled by a factor of $\frac{NR}{D}\geq 1$.
At the same time, $\bfPhi^*\bfPhi$ is naturally regarded as an $\calN\times\calN$ array whose $(n_1,n_2)$th block is the $\calR\times\calR$ \textit{cross-Gram} matrix $\bfPhi_{n_1}^*\bfPhi_{n_2}^{}$:
\begin{equation*}
(\bfPhi^*\bfPhi)((n_1,r_1),(n_2,r_2))
=\ip{\bfphi_{n_1,r_1}}{\bfphi_{n_2,r_2}}
=\ip{\bfPhi_{n_1}\bfdelta_{r_1}}{\bfPhi_{n_2}\bfdelta_{r_2}}
=(\bfPhi_{n_1}^*\bfPhi_{n_2}^{})(r_1,r_2).
\end{equation*}
In particular, since each sequence $\set{\bfphi_{n,r}}_{r\in\calR}$ is orthonormal,
every diagonal block of this ``fusion Gram" matrix is an $\calR\times\calR$ identity.
(We caution that since $\set{\bfphi_{n,r}}_{r\in\calR}$ may be any ONB for $\calU_n$, its synthesis operator $\bfPhi_n$ is only unique up to right unitaries.
Thus, while the \textit{fusion frame operator} $\bfPhi\bfPhi^*=\sum_{n\in\calN}\bfPhi_n^{}\bfPhi_n^*=\sum_{n\in\calN}\bfP_n$ is invariant with respect to one's choices of bases,
the ``fusion Gram" matrix $\bfPhi^*\bfPhi$ is not.)
Since $\set{\bfphi_{n,r}}_{(n,r)\in\calN\times\calR}$ is an $NR$-vector tight frame for the $D$-dimensional space $\calH$,
it has a Naimark complement $\set{\bfpsi_{n,r}}_{(n,r)\in\calN\times\calR}$ which is a tight frame for some $(NR-D)$-dimensional Hilbert space $\calK$, provided $NR\neq D$.
We elect to scale this Naimark complement so that the diagonal blocks of its Gram matrix $\bfPsi^*\bfPsi$ are $\calR\times\calR$ identities, letting
\begin{equation}
\label{eq.Naimark complement}
\bfPsi^*\bfPsi
=\tfrac{NR}{NR-D}(\bfI-\tfrac{D}{NR}\bfPhi^*\bfPhi),
\ \text{ i.e., }\
\bfPsi_{n_1}^*\bfPsi_{n_2}^{}
=\left\{\begin{array}{cl}
\bfI,&\ n_1=n_2,\\
-\tfrac{D}{NR-D}\bfPhi_{n_1}^*\bfPhi_{n_2}^{},&\ n_1\neq n_2.
\end{array}\right.
\end{equation}
Because of this, $\set{\bfpsi_{n,r}}_{(n,r)\in\calN\times\calR}$ is a tight frame for $\calK$ with the property that, for every $n$, $\set{\bfpsi_{n,r}}_{r\in\calR}$ is orthonormal.
Defining $\set{\calV_n}_{n\in\calN}$ by $\calV_n:=\Span\set{\bfpsi_{n,r}}_{r\in\calR}$ for each $n$ thus yields a $\TFF(NR-D,N,R)$ for $\calK$.
Here, if $\set{\calU_n}_{n\in\calN}$ is real,
$\set{\calV_n}_{n\in\calN}$ can be chosen to be real.

In this context, for any $n_1,n_2\in\calN$, the fact that $\norm{\bfPhi_{n_1}^*\bfPhi_{n_2}^{}}_2\leq\norm{\bfPhi_{n_1}}_2\norm{\bfPhi_{n_2}}_2=1$ implies that the singular values of $\bfPhi_{n_1}^*\bfPhi_{n_2}$ can be expressed as $\set{\cos(\theta_{n_1,n_2,r})}_{r=1}^{R}$ for some increasing sequence $\set{\theta_{n_1,n_2,r}}_{r=1}^{R}$ of \textit{principal angles} in $[0,\frac \pi2]$.
Here since
$\Tr(\bfP_{n_1}\bfP_{n_2})
=\Tr(\bfPhi_{n_1}^{}\bfPhi_{n_1}^*\bfPhi_{n_2}^{}\bfPhi_{n_2}^*)
=\norm{\bfPhi_{n_1}^*\bfPhi_{n_2}^{}}_\Fro^2
=\sum_{r=1}^{R}\cos^2(\theta_{n_1,n_2,r})$,
$\set{\calU_n}_{n\in\calN}$ is equichordal if and only if its off-diagonal cross-Gram matrices all have the same Frobenius norm.
In light of~\eqref{eq.Naimark complement}, this implies that the Naimark complement of any $\ECTFF(D,N,R)$ with $D\neq NR$ is an $\ECTFF(NR-D,N,R)$.
Meanwhile, $\set{\calU_n}_{n\in\calN}$ is equi-isoclinic if and only if there exists $\lambda\geq0$ such that
\begin{equation*}
\bfPhi_{n_1}^{}\bfPhi_{n_1}^*\bfPhi_{n_2}^{}\bfPhi_{n_2}^*\bfPhi_{n_1}^{}\bfPhi_{n_1}^*
=\bfP_{n_1}\bfP_{n_2}\bfP_{n_1}
=\lambda\bfP_{n_1}
=\lambda\bfPhi_{n_1}^{}\bfPhi_{n_1}^*
\end{equation*}
for all $n_1\neq n_2$.
Since $\bfPhi_n^*\bfPhi_n^{}=\bfI$ for all $n$,
this equates to having
$\bfPhi_{n_1}^*\bfPhi_{n_2}^{}\bfPhi_{n_2}^*\bfPhi_{n_1}^{}=\lambda\bfI$ for all $n_1\neq n_2$,
namely to $\theta_{n_1,n_2,r}$ being some constant $\theta$ over all $n_1\neq n_2$ and $r$.
That is, $\set{\calU_n}_{n\in\calN}$ is equi-isoclinic if and only if each of its off-diagonal cross-Gram matrices is a common scalar multiple of some unitary.
(In particular, when $\set{\calU_n}_{n\in\calN}$ is an $\EITFF(D,N,R)$,
we have equality throughout~\eqref{eq.spectral Welch} and so this scalar is necessarily $\cos(\theta)=[\frac{NR-D}{D(N-1)}]^{\frac12}$.)
Together with~\eqref{eq.Naimark complement} this implies that the Naimark complement of an $\EITFF(D,N,R)$ with $D\neq NR$ is an $\EITFF(NR-D,N,R)$.

Notably, the spatial complement of an $\EITFF(D,N,R)$ with $D\neq R$ is only an EITFF when $D=2R$.
Indeed, any $\TFF(D,N,R)$ $\set{\calU_n}_{n\in\calN}$ with $R<D<2R$ cannot be equi-isoclinic,
since in this case any two subspaces $\calU_{n_1}$ and $\calU_{n_2}$ intersect nontrivially, implying their smallest principal angle $\theta_{n_1,n_2,1}$ satisfies $\cos^2(\theta_{n_1,n_2,1})=\norm{\bfP_{n_1}\bfP_{n_2}}_2^2=1>\tfrac{NR-D}{D(N-1)}$.
To precisely determine how spatial complements affect principal angles,
for each $n$ let $\bfXi_n$ be the synthesis operator for an ONB $\set{\bfxi_s}_{s\in\calS}$ for $\calU_n^\perp$,
and so $\bfPhi_n^{}\bfPhi_n^*+\bfXi_n^{}\bfXi_n^*=\bfI$, $\bfPhi_n^*\bfPhi_n^{}=\bfI$, $\bfXi_n^*\bfXi_n^{}=\bfI$ and $\bfPhi_n^*\bfXi_n^{}=\bfzero$.
Note that
\begin{align*}
\bfI-(\bfPhi_{n_1}^*\bfPhi_{n_2}^{})(\bfPhi_{n_1}^*\bfPhi_{n_2}^{})^*
&=\bfI-\bfPhi_{n_1}^*(\bfI-\bfXi_{n_2}^{}\bfXi_{n_2}^*)\bfPhi_{n_1}^{}
=(\bfPhi_{n_1}^*\bfXi_{n_2}^{})(\bfPhi_{n_1}^*\bfXi_{n_2}^{})^*,\\
\bfI-(\bfXi_{n_1}^*\bfXi_{n_2}^{})^*(\bfXi_{n_1}^*\bfXi_{n_2}^{})
&=\bfI-\bfXi_{n_2}^*(\bfI-\bfPhi_{n_1}^{}\bfPhi_{n_1}^*)\bfXi_{n_2}^{}
=(\bfPhi_{n_1}^*\bfXi_{n_2}^{})^*(\bfPhi_{n_1}^*\bfXi_{n_2}^{}).
\end{align*}
As such, the spectra of $\bfI-(\bfPhi_{n_1}^*\bfPhi_{n_2}^{})(\bfPhi_{n_1}^*\bfPhi_{n_2}^{})^*$
and
$\bfI-(\bfXi_{n_1}^*\bfXi_{n_2}^{})^*(\bfXi_{n_1}^*\bfXi_{n_2}^{})$
are zero-padded versions of each other,
which in turn implies that the sequences of singular values of $\bfPhi_{n_1}^*\bfPhi_{n_2}^{}$ and $\bfXi_{n_1}^*\bfXi_{n_2}^{}$ are $1$-padded versions of each other,
and so the sequences of principal angles between $\calU_{n_1}$ and $\calU_{n_2}$ and those between $\calU_{n_1}^\perp$ and $\calU_{n_2}^\perp$ are $0$-padded versions of each other.
In the special case where $D=2R$, this implies that if $\set{\calU_n}_{n\in\calN}$ is equi-isoclinic then $\set{\calU_n^\perp}_{n\in\calN}$ is as well.

In general, the rank-$R$ projections $\set{\bfP_n}_{n\in\calN}$ of any sequence $\set{\calU_n}_{n\in\calN}$ of $R$-dimensional subspaces of $\calH$ lie in the real Hilbert space of self-adjoint operators over $\calH$.
The Gram matrix of these projections (not to be confused with the aforementioned ``fusion Gram" matrix of a concatenation of orthonormal bases for these spaces) has $(n_1,n_2)$th entry $\ip{\bfP_{n_1}}{\bfP_{n_2}}_\Fro=\Tr(\bfP_{n_1}\bfP_{n_2})$.
When $\set{\calU_n}_{n\in\calN}$ is equichordal and nonidentical,
this Gram matrix is of the form $R\bfI+C(\bfJ-\bfI)$ where $0<C<R$ and $\bfJ$ is an all-ones $\calN\times\calN$ matrix.
Such a Gram matrix is invertible,
implying $\set{\bfP_n}_{n\in\calN}$ is linearly independent.
This yields \textit{Gerzon's bound} on the number $N$ of equichordal nonidentical subspaces of a Hilbert space of dimension $D$:
\begin{equation}
\label{eq.Gerzon's bound}
N
\leq
\dim\set{\bfA:\calH\rightarrow\calH\ |\ \bfA^*=\bfA}
=\left\{\begin{array}{cl}
\tfrac12D(D+1),&\ \bbF=\bbR,\\
D^2,&\ \bbF=\bbC.
\end{array}\right.
\end{equation}

In general, for any sequence $\set{\calU_n}_{n\in\calN}$ of $R$-dimensional subspaces of $\calH$ the fact that $\bfPhi_n^*\bfPhi_n^{}=\bfI$ for all $n$ implies that
$\norm{\bfP_{n_1}\bfP_{n_2}}_2
=\norm{\bfPhi_{n_1}^{}\bfPhi_{n_1}^*\bfPhi_{n_2}^{}\bfPhi_{n_2}^*}_2
=\norm{\bfPhi_{n_1}^*\bfPhi_{n_2}^{}}_2$ for all $n_1,n_2$,
converting~\eqref{eq.spectral Welch} into the following lower bound on the \textit{block coherence}~\cite{EldarKB10,CalderbankTX15} of $\set{\bfphi_{n,r}}_{(n,r)\in\calN\times\calR}$:
\begin{equation}
\label{eq.block Welch}
\max_{n_1\neq n_2}\norm{\bfPhi_{n_1}^*\bfPhi_{n_2}^{}}_2
\geq\bigbracket{\tfrac{NR-D}{D(N-1)}}^{\frac12}.
\end{equation}
When $R=1$,
each $\bfPhi_n$ is the synthesis operator of a single arbitrary unit vector $\bfphi_n\in\calU_n$.
In this special case, each cross-Gram matrix $\bfPhi_{n_1}^*\bfPhi_{n_2}^{}$ is a $1\times 1$ matrix with entry $\bfphi_{n_1}^*\bfphi_{n_2}^{}=\ip{\bfphi_{n_1}}{\bfphi_{n_2}}$,
and so
$\norm{\bfP_{n_1}\bfP_{n_2}}_\Fro^2
=\Tr(\bfP_{n_1}\bfP_{n_2})
=\norm{\bfPhi_{n_1}^*\bfPhi_{n_2}}_\Fro^2
=\abs{\ip{\bfphi_{n_1}}{\bfphi_{n_2}}}^2
=\norm{\bfPhi_{n_1}^*\bfPhi_{n_2}}_2^2
=\norm{\bfP_{n_1}\bfP_{n_2}}_2^2$.
In this context, $\set{\calU_n}_{n\in\calN}$ is equichordal if and only if it is equi-isoclinic, and this occurs if and only if $\set{\bfphi_n}_{n\in\calN}$ is \textit{equiangular}, that is, $\abs{\ip{\bfphi_{n_1}}{\bfphi_{n_2}}}$ is constant over all $n_1\neq n_2$.
Here, \eqref{eq.block Welch} reduces to the Welch bound,
and $\set{\bfphi_n}_{n\in\calN}$ achieves equality in it if and only if it is an ETF for $\calH$.
For this reason, we sometimes refer to an $\ECTFF(D,N,1)$ as an ``$\ETF(D,N)$".
In particular,
every $\ETF(D,N)$ with $D<N$ has a Naimark complement which is an $\ETF(N-D,N)$,
but the spatial complement of an $\ETF(D,N)$ with $D>1$ is an $\ECTFF(D,N,D-1)$.
(In light of~\eqref{eq.Gerzon's bound}, the latter can only be an ETF when $D=2$ and $N\in\set{2,3,4}$.)

We need two trivial TFF constructions for our work in the next section.
Any $\TFF(D,N,R)$ for $\calH$ with $D=R$ necessarily consists of $N$ copies of the entire space $\calH$.
Meanwhile, the synthesis operator $\bfPhi:\bbF^{\calN\times\calR}\rightarrow\calH$ of a (concatenation of orthonormal bases for the subspaces of a) $\TFF(D,N,R)$ for $\calH$ with $D=NR$ necessarily satisfies $\bfPhi\bfPhi^*=\tfrac{NR}{D}\bfI=\bfI$ where $\dim(\bbF^{\calN\times\calR})=NR=D=\dim(\calH)$,
implying that such a transformation $\bfPhi$ is necessarily unitary.
Any such $\TFF(D,N,R)$ thus consists of the $N$ respective spans of any partition of an ONB for $\calH$ into $R$-element subsequences.
In either of these two trivial cases,
these constructions in fact yield an $\EITFF(D,N,R)$,
and can moreover always be chosen to be real by, for example, letting $\calH=\bbR^D$.

\section{Naimark-spatial orbits of equichordal tight fusion frames}

In the previous section we reviewed a number of previously known facts about TFFs,
including: if a $\TFF(D,N,R)$ exists then $R\leq D\leq NR$;
any $\TFF(D,N,R)$ with $D\neq NR$ has a Naimark complement that is a $\TFF(NR-D,N,R)$
and that is real and/or equichordal and/or equi-isoclinic if the original TFF is as well;
any $\TFF(D,N,R)$ with $D\neq R$ has a spatial complement that is a $\TFF(D-R,N,R)$
and that is real and/or equichordal if the original TFF is too.
Remarkably, it turns out that taking alternating Naimark and spatial complements of any $\TFF(D,N,R)$ with either $N\geq 5$ or $N=4$, $D\neq 2R$ yields an infinite sequence of TFFs with distinct parameters.

For example, four copies of the scalar $1$ form a trivial $\ETF(1,4)$ for $\calH=\bbR$, and so a real $\EITFF(1,4,1)$.
Though this ECTFF does not have a spatial complement (since $D=1=R$), its Naimark complement is a real $\EITFF(3,4,1)$.
Geometrically, each of the four lines that comprise this $\ECTFF(3,4,1)$ contains one of the four vertices of a regular tetrahedron centered at the origin.
Though the Naimark complement of this $\EITFF(3,4,1)$ is just our original $\EITFF(1,4,1)$ (up to unitary transformations and nonzero scalar multiples),
its spatial complement is a real $\ECTFF(3,4,2)$.
Taking alternating Naimark and spatial complements in this fashion yields an infinite sequence of real ECTFFs with the following $(D,N,R)$ parameters:
\begin{equation*}
\mathbf{(1,4,1)}
\underset{\rmN}{\leftrightarrow}
(3,4,1)
\underset{\rms}{\leftrightarrow}
(3,4,2)
\underset{\rmN}{\leftrightarrow}
(5,4,2)
\underset{\rms}{\leftrightarrow}
(5,4,3)
\underset{\rmN}{\leftrightarrow}
(7,4,3)
\underset{\rms}{\leftrightarrow}
(7,4,4)
\underset{\rmN}{\leftrightarrow}
\dotsb.
\end{equation*}
For another example, consider the real $\EITFF(4,4,1)$ comprised of the four canonical axes of $\bbR^4$.
Though this EITFF has no Naimark complement (since $D=4=NR$),
its spatial complement is a real $\ECTFF(4,4,3)$,
and moreover taking alternating Naimark and spatial complements yields an infinite sequence of real ECTFFs with the following $(D,N,R)$ parameters:
\begin{equation*}
\mathbf{(4,4,1)}
\underset{\rms}{\leftrightarrow}
(4,4,3)
\underset{\rmN}{\leftrightarrow}
(8,4,3)
\underset{\rms}{\leftrightarrow}
(8,4,5)
\underset{\rmN}{\leftrightarrow}
(12,4,5)
\underset{\rms}{\leftrightarrow}
(12,4,7)
\underset{\rmN}{\leftrightarrow}
(16,4,7)
\underset{\rms}{\leftrightarrow}
\dotsb.
\end{equation*}
Such sequences can also be bi-infinite:
an $\ETF(3,7)$ exists~\cite{FickusM16},
and taking alternating Naimark and spatial complements of this $\EITFF(3,7,1)$ yields distinct parameters depending on which complement is applied first:
\begin{equation}
\label{eq.(3,7,1) orbit}
\dotsb
\underset{\rmN}{\leftrightarrow}
(11,7,9)
\underset{\rms}{\leftrightarrow}
(11,7,2)
\underset{\rmN}{\leftrightarrow}
(3,7,2)
\underset{\rms}{\leftrightarrow}
\mathbf{(3,7,1)}
\underset{\rmN}{\leftrightarrow}
(4,7,1)
\underset{\rms}{\leftrightarrow}
(4,7,3)
\underset{\rmN}{\leftrightarrow}
(17,7,3)
\underset{\rms}{\leftrightarrow}
\dotsb.
\end{equation}
To discuss such sequences in general, it helps to formally define some related concepts in terms of the Naimark and spatial involutions mentioned in the introduction:
\begin{definition}
\label{def.orbit}
Let $\nu,\sigma:\bbR^3\rightarrow\bbR^3$, $\nu(D,N,R):=(NR-D,N,R)$, $\sigma(D,N,R):=(D,N,D-R)$.
The \textit{Naimark-spatial sequence} of $(D,N,R)\in\bbZ^3$ with $N>1$ is the doubly infinite sequence $\set{(D^{(K)},N,R^{(K)})}_{K=-\infty}^{\infty}$ with
$(D^{(0)},N,R^{(0)})=(D,N,R)$ and
\begin{equation*}
\nu(D^{(2J+1)},N,R^{(2J+1)})=(D^{(2J)},N,R^{(2J)})
=\sigma(D^{(2J-1)},N,R^{(2J-1)}),\quad\forall J\in\bbZ.
\end{equation*}
The \textit{(Naimark-spatial) orbit} of $(D,N,R)$ is its orbit of $(D,N,R)$ under the action of the group generated by $\nu$ and $\sigma$,
namely the set $\Orb(D,N,R):=\set{(D^{(K)},N,R^{(K)}): K\in\bbZ}$.
\end{definition}

Here, since $\nu$ and $\sigma$ are involutions (that is, are their own inverses),
the above definition of the Naimark-spatial sequence of $(D,N,R)$ is simply a succinct way of stating that it is the sequence of all triples obtained by iteratively applying the Naimark and spatial involutions to it in an alternating fashion, beginning with either one:
\begin{align}
\nonumber\dotsb\ (D^{(-3)},N,R^{(-3)})&=\sigma(\nu(\sigma(D,N,R))),\\
\nonumber(D^{(-2)},N,R^{(-2)})&=\nu(\sigma(D,N,R)),\\
\nonumber(D^{(-1)},N,R^{(-1)})&=\sigma(D,N,R),\\
\label{eq.Naimark spatial sequence}
(D^{(0)},N,R^{(0)})&=(D,N,R),\\
\nonumber(D^{(1)},N,R^{(1)})&=\nu(D,N,R),\\
\nonumber(D^{(2)},N,R^{(2)})&=\sigma(\nu(D,N,R)),\\
\nonumber(D^{(3)},N,R^{(3)})&=\nu(\sigma(\nu(D,N,R))),\ \dotsb .
\end{align}
(Though the middle parameter of this sequence remains constant, we do not discard it since it is used to evaluate $\nu$.)
From this, it immediately follows that $\Orb(D,N,R)$ is indeed the orbit of $(D,N,R)$ under the action of the group generated by $\nu$ and $\sigma$.

When a $\TFF(D,N,R)$ exists, a Naimark or spatial complement of it has parameters $\nu(D,N,R)$ or $\sigma(D,N,R)$, respectively.
As such, in this case $\Orb(D,N,R)$ contains the parameter triples of any TFF that can be obtained from it in the \textit{Naimark-spatial way}, that is, via an arbitrary finite number of iterated alternating Naimark and spatial complements, beginning with either type.
We caution however that since a $\TFF(D,N,R)$ only has a Naimark or spatial complement when $D\neq NR$ or $D\neq R$, respectively, $\Orb(D,N,R)$ might contain triples that are not the parameters of any TFF.
For example,
$(1,4,0)\in\Orb(1,4,1)$ but no $\TFF(1,4,0)$ exists.
Rather, a $\TFF(D',N,R')$ exists for every $(D',N,R')\in\Orb(D,N,R)$ only if $\Orb(D,N,R)$ is contained in $\set{(D',N,R')\in\bbZ^3: 0<R'<D'<NR'}$.
Since $\Orb(D,N,R)$ is invariant under $\nu$ and $\sigma$,
this occurs if and only if it is contained in $\set{(D',N,R')\in\bbZ^3: D'>0, R'>0}$.

As mentioned in the introduction,
the key to understanding Naimark-spatial orbits and their implications for (real and/or equichordal and/or equi-isoclinic) TFFs is the function~\eqref{eq.invariant}, which is invariant with respect to both the Naimark and spatial involutions.
For any positive integer $N$, this function is a quadratic form of $(D,R)$:
\begin{equation}
\label{eq.quadratic form}
f_N:\bbR^2\rightarrow\bbR,
\quad
f_N(D,R)
:=f(D,N,R)
=DNR-D^2-NR^2
=\left[\begin{array}{cc}D&R\end{array}\right]
\left[\begin{array}{rr}-1&\tfrac N2\\\tfrac N2&-N\end{array}\right]
\left[\begin{array}{c}D\\R\end{array}\right].
\end{equation}
This $2\times 2$ matrix has characteristic polynomial
\begin{equation*}
\left|\begin{array}{rr}
 \lambda+1&-\tfrac N2\\
-\tfrac N2& \lambda+N
\end{array}\right|
=(\lambda+1)(\lambda+N)-\tfrac{N^2}{4}
=\lambda^2+(N+1)\lambda-\tfrac{N(N-4)}4,
\end{equation*}
and so has eigenvalues $\lambda=\tfrac12\{-(N+1)\pm[(N+1)^2+N(N-4)]^{\frac12}\}$.
In particular, its lesser eigenvalue is negative,
while its greater eigenvalue is negative, zero or positive depending on whether $N<4$, $N=4$ or $N>4$, respectively.
In these three cases,
a level set $\set{(D,R)\in\bbR^2: f_N(D,R)=C}$ is an ellipse, either one line or two parallel lines, or a hyperbola, respectively.
Taking a Naimark or spatial involution corresponds to moving horizontally or vertically, respectively, from one point on such a level set to another
(which might be the same point if $D=NR-D$ or $R=D-R$).
Taking alternating Naimark and spatial involutions thus corresponds to moving along such a level set via an alternating sequence of horizontal and vertical steps.
When $N<4$ such paths only contain a finite number of distinct points.
Meanwhile, when $N\geq 4$, the nature of such a path depends on whether $C=0$, $C<0$ or $C>0$.
In particular,
when $C<0$, it infinitely bounces back and forth between either two parallel lines (when $N=4$) or the two connected components of a hyperbola (when $N>4$).
When $C>0$ (and so necessarily $N>4$), it instead infinitely weaves itself along a single connected component of a hyperbola that lies in the first quadrant.
As such, we now rigorously analyze $\Orb(D,N,R)$ for $(D,N,R)\in\bbZ^3$, $D>0$, $N>1$, $R>0$ in four distinct cases: (i) $N\in\set{2,3}$, (ii) $N\geq 4$, $f(D,N,R)=0$, (iii) $N\geq 4$, $f(D,N,R)<0$ and (iv) $f(D,N,R)>0$.

\subsection{Naimark-spatial orbits when $N\in\set{2,3}$}

When $N\in\set{2,3}$, the orbit of any $(D,N,R)$ under $\nu$ and $\sigma$ is relatively simple.
For instance, when $N=2$, such an orbit consists of at most eight distinct points:
\begin{align}
\nonumber
\dotsb
\underset{\sigma}{\leftrightarrow}&
(D,2,R)
\underset{\nu}{\leftrightarrow}
(2R-D,2,R)
\underset{\sigma}{\leftrightarrow}
(2R-D,2,R-D)
\underset{\nu}{\leftrightarrow}
(-D,2,R-D)\\
\nonumber
\underset{\sigma}{\leftrightarrow}&
(-D,2,-R)
\underset{\nu}{\leftrightarrow}
(D-2R,2,-R)
\underset{\sigma}{\leftrightarrow}
(D-2R,2,D-R)
\underset{\nu}{\leftrightarrow}
(D,2,D-R)\\
\label{eq.orbit when N=2}
\underset{\sigma}{\leftrightarrow}&
(D,2,R)
\underset{\nu}{\leftrightarrow}
\dotsb.
\end{align}
Here note that when a $\TFF(D,2,R)$ exists, a $\TFF(-D,2,-R)$ does not.
In this case, $\Orb(D,2,R)$ thus properly contains the set of triples that arise in the Naimark-spatial way from a $\TFF(D,2,R)$.
A similar phenomenon occurs when $N=3$:
for any $(D,R)\in\bbR^2$, $\Orb(D,3,R)$ consists of at most twelve distinct points,
one of which is $(-D,3,-R)$:
\begin{align}
\nonumber
\dotsb
\underset{\sigma}{\leftrightarrow}&
(D,3,R)
\underset{\nu}{\leftrightarrow}
(3R-D,3,R)
\underset{\sigma}{\leftrightarrow}
(3R-D,3,2R-D)
\underset{\nu}{\leftrightarrow}
(3R-2D,3,2R-D)\\
\nonumber
\underset{\sigma}{\leftrightarrow}&
(3R-2D,3,R-D)
\underset{\nu}{\leftrightarrow}
(-D,3,R-D)
\underset{\sigma}{\leftrightarrow}
(-D,3,-R)
\underset{\nu}{\leftrightarrow}
(D-3R,3,-R)\\
\nonumber
\underset{\sigma}{\leftrightarrow}&
(D-3R,3,D-2R)
\underset{\nu}{\leftrightarrow}
(2D-3R,3,D-2R)
\underset{\sigma}{\leftrightarrow}
(2D-3R,3,D-R)
\underset{\nu}{\leftrightarrow}
(D,3,D-R)\\
\label{eq.orbit when N=3}
\underset{\sigma}{\leftrightarrow}&
(D,3,R)
\underset{\nu}{\leftrightarrow}
\dotsb.
\end{align}
It is therefore not surprising that there is little freedom in the parameters of such TFFs:

\begin{theorem}
\label{thm.N=2,3}
Let $D$ and $R$ be positive integers.
\begin{enumerate}
\renewcommand{\labelenumi}{(\alph{enumi})}
\item
A $\TFF(D,2,R)$ exists if and only if $R\in\set{\frac D2,D}$.
Such TFFs are necessarily equi-isoclinic, and can be chosen to be real.\smallskip
\item
A $\TFF(D,3,R)$ exists if and only if $R\in\set{\frac D3,\frac D2,\frac{2D}3,D}$.
Such TFFs are necessarily equi-isoclinic when $R\in\set{\frac D3,\frac D2,D}$,
equichordal when $R=\frac{2D}3$,
and can be chosen to be real.
\end{enumerate}
\end{theorem}

\begin{proof}
Throughout, recall from Section~2 that any $\TFF(D,N,R)$ with $R\in\set{\frac DN, D}$ is necessarily equi-isoclinic,
and conversely that a trivial real $\EITFF(D,N,R)$ exists whenever $(D,N,R)\in\bbZ^3$ satisfies $D>0$, $N>1$ and $R\in\set{\frac DN, D}$.

In light of this fact, for (a) it suffices to show that $R\in\set{\tfrac D2,D}$ for any $\TFF(D,2,R)$.
Here recall from Section~2 that such a TFF equates to two rank-$R$ projections $\bfP_1$ and $\bfP_2$ on a $D$-dimensional Hilbert space that satisfy $\bfP_1+\bfP_2=\tfrac{2R}{D}\bfI$.
Here, $\bfP_1=\tfrac{2R}{D}\bfI-\bfP_2$,
and so $\tfrac{2R}{D}\bfI-\bfP_2$ has $1$ as an eigenvalue,
implying $1\in\set{\tfrac{2R}{D}-1,\tfrac{2R}{D}}$.
Solving for $R$ in both cases indeed gives $R\in\set{\tfrac D2,D}$.

For (b), we begin by proving the following claim:
if a $\TFF(D,3,R)$ exists then either $R=\frac D3$ or $R\geq\frac D2$.
Here, the rank-$R$ projections $\set{\bfP_1,\bfP_2,\bfP_3}$ of any such subspaces $\set{\calU_1,\calU_2,\calU_3}$ satisfy $\bfP_1+\bfP_2+\bfP_3=\tfrac{3R}{D}\bfI$.
Since $\bfP_1+\bfP_2=\frac{3R}{D}\bfI-\bfP_3$,
the operator $\bfP_1+\bfP_2$ has eigenvalues $\frac{3R}{D}-1$ and $\frac{3R}{D}$ with multiplicities $R$ and $D-R$, respectively.
If $\calU_1\subseteq\calU_2^\perp$ then $\bfP_1+\bfP_2$ is the projection onto $\calU_1+\calU_2$,
implying its greatest eigenvalue is $\frac{3R}{D}=1$ and so $R=\tfrac D3$.
If there instead exists $\bfy_1\in\calU_1$, $\bfy_2\in\calU_2$
such that $\ip{\bfy_1}{\bfy_2}\neq0$ then the greatest eigenvalue of $\bfP_1+\bfP_2$ is strictly greater than $1$.
We now exploit a well-known fact from matrix analysis:
the eigenvalues $\set{a_d}_{d=1}^D$ of any self-adjoint operator $\bfA:\calH\rightarrow\calH$ obtained by adding a rank-$1$ projection to any other self-adjoint operator $\bfB:\calH\rightarrow\calH$ \textit{interlace} with the eigenvalues $\set{b_d}_{d=1}^D$ of $\bfB$,
that is,
\begin{equation*}
a_1\geq b_1\geq a_2\geq b_2\geq\dotsc\geq a_{d-1}\geq b_{d-1}\geq a_d\geq b_d.
\end{equation*}
In particular, in our current situation where the greatest eigenvalue of $\bfP_1+\bfP_2$ is strictly greater than that of $\bfP_1$,
its multiplicity $D-R$ is at most the rank of $\bfP_2$, namely $R$.
Thus, $D-R\leq R$ and so $R\geq \frac D2$, as claimed.

Having the claim, note that if a $\TFF(D,3,R)$ exists then either $R\in\set{\frac D3,\frac D2,D}$ or $\frac D2<R<D$.
In the latter case this TFF's Naimark complement is a $\TFF(D,3,D-R)$ where $D-R<D-\frac D2=\tfrac D2$ and so, by the claim, $D-R=\frac D3$, that is, $R=\tfrac{2D}3$.
Moreover,
as mentioned at the beginning of this proof, any $\TFF(D,3,R)$ with $R\set{\frac D3,D}$ is necessarily equi-isoclinic.
This also holds when $R=\frac D2$: any $\TFF(D,3,\frac D2)$ is the Naimark complement of a $\TFF(\frac D2,3,\frac D2)$, which has equal ``$D$" and ``$R$" parameters and so is necessarily equi-isoclinic.
Meanwhile, any $\TFF(D,3,R)$ with $R=\frac{2D}3$ cannot be equi-isoclinic since $\frac D2<R<D$ but, as noted above, is the spatial complement of a (necessarily equi-isoclinic and so also equichordal) $\TFF(D,3,\frac D3)$.
Conversely, a trivial real $\TFF(D,3,R)$ exists whenever $R\in\set{\frac D3,D}$,
and taking the Naimark complement of a trivial real $\TFF(\frac D2,3,\frac D2)$ or the spatial complement of a trivial real $\TFF(D,3,\frac D3)$ yields real $\TFF(D,3,R)$ with $R=\frac D2$ and $R=\frac{2D}3$, respectively.
\end{proof}

From this result, we see that only ever at most a small portion of an orbit of the form~\eqref{eq.orbit when N=2} or~\eqref{eq.orbit when N=3} will correspond to actual TFFs.
In particular, the parameters of any $\TFF(D,2,R)$ appear in one of two types of
``Naimark-spatial paths" of length two:
\begin{equation}
\label{eq.NS path N=2}
(R,2,R)
\underset{\rmN}{\leftrightarrow}
(R,2,R),
\qquad
(2R,2,R)
\underset{\rms}{\leftrightarrow}
(2R,2,R).
\end{equation}
(A $\TFF(R,2,R)$ has no spatial complement, while a $\TFF(2R,2,R)$ has no Naimark complement.)
Similarly, the parameters of any $\TFF(D,3,R)$ appear in a path of length four, either of the form
\begin{equation}
\label{eq.NS path N=3 case 1}
(3R,3,R)
\underset{\rms}{\leftrightarrow}
(3R,3,2R)
\underset{\rmN}{\leftrightarrow}
(3R,3,2R)
\underset{\rms}{\leftrightarrow}
(3R,3,R)
\end{equation}
when $R\in\set{\tfrac D3,\tfrac{2D}3}$
(where a $\TFF(3R,3,R)$ has no Naimark complement),
or of the form
\begin{equation}
\label{eq.NS path N=3 case 2}
(R,3,R)
\underset{\rmN}{\leftrightarrow}
(2R,3,R)
\underset{\rms}{\leftrightarrow}
(2R,3,R)
\underset{\rmN}{\leftrightarrow}
(R,3,R)
\end{equation}
when $R\in\set{\tfrac D2,D}$
(where a $\TFF(R,3,R)$ has no spatial complement).
All in all, we see that the orbit formalism of Definition~\ref{def.orbit} is not particularly helpful in the study of $\TFF(D,N,R)$ with $N<4$.
The same will be true for $\TFF(D,N,R)$ with $f(D,N,R)=0$.
However, orbits will prove invaluable in the study of those with either $f(D,N,R)>0$ or $N\geq 4$ and $f(D,N,R)<0$.

\subsection{Naimark-spatial orbits when $f(D,N,R)=0$}

As we now explain, $\TFF(D,N,R)$ with $f(D,N,R)=0$ are rare:

\begin{theorem}
\label{thm.f=0}
If $f(D,N,R):=DNR-D^2-NR^2=0$ for some $(D,N,R)\in\bbZ^3$ with $D>0$, $N>1$ and $R>0$ then $N=4$ and $D=2R$.
In this case, $\Orb(D,N,R)$ is a singleton set.
\end{theorem}

\begin{proof}
Recall that for any positive integer $N$,
quadratic form $f_N$ of~\eqref{eq.quadratic form} arises from a negative definite matrix whenever $N<4$.
Here since $(D,N,R)\in\bbZ^3$ satisfies $f_N(D,R)=f(D,N,R)=0$ where $D>0$, $N>1$ and $R>0$ we thus have $N\geq 4$.
By the quadratic formula, the zero set of $f_N$ is the union of the two lines in the $(D,R)$-plane with slopes \smash{$\tfrac12[1\pm(\tfrac{N-4}{N})^{\frac12}]$}.
Here if $N>4$ then $(\tfrac{N-4}{N})^{\frac12}=\tfrac1{N^2}[N(N-4)]^{\frac12}$ is irrational since $N(N-4)=(N-2)^2-4$ whereas $4$ and $0$ are the only perfect squares whose difference is $4$.
This contradicts the fact that both $D$ and $R$ are positive integers.
Thus, $N=4$, and the two aforementioned lines are equal, both having slope $\frac12$.
Altogether, we see that $(D,N,R)=(2R,4,R)$.
Moreover, any such point is fixed by both $\nu$ and $\sigma$:
\begin{equation*}
\nu(2R,4,R)=(4R-2R,4,R)=(2R,4,R),
\quad
\sigma(2R,4,R)=(2R,4,2R-R)=(2R,4,R).
\end{equation*}
Thus, $\Orb(D,N,R)=\Orb(2R,4,R)=\set{(2R,4,R)}$ is a singleton set.
\end{proof}

From this result, we see that if a (real and/or equichordal and/or equi-isoclinic) $\TFF(D,N,R)$ with $f(D,N,R)=0$ exists,
then both its Naimark and spatial complements exist,
and are themselves (real and/or equichordal and/or equi-isoclinic) $\TFF(D,N,R)$.
Later on in Theorem~\ref{thm.(2R,4,R)}, we combine results from the existing literature with some new analysis to fully characterize when a (real and/or equichordal and/or equi-isoclinic) $\TFF(2R,4,R)$ exists:
a (complex) $\EITFF(2R,4,R)$ and real $\ECTFF(2R,4,R)$ exist for any positive integer $R$, but a real $\EITFF(2R,4,R)$ exists if and only if $R$ is even.

\subsection{Naimark-spatial orbits when $N\geq 4$ and $f(D,N,R)<0$}

When $N\geq 4$ and $C<0$, a level set of the form $\set{(D,R)\in\bbR^2: f_N(D,R)=C}$ consists of the two connected components of either a hyperbola or a union of two parallel lines, depending on whether $N>4$ or $N=4$, respectively.
Applying $\nu$ or $\sigma$ to any $(D,N,R)$ for which $f_N(D,R)=C$ corresponds to jumping horizontally or vertically, respectively, from one of these two connected components to the other.
Applying $\nu$ and $\sigma$ in an iterative, alternating fashion yields an infinite ``staircase" of $(D,R)$ pairs of arbitrarily large positive and negative values.
Clearly, those pairs in this staircase with a nonpositive entry do not correspond to the parameters of any $\TFF(D,N,R)$.
Remarkably, as we now explain,
the remaining ``positive pairs" in this staircase do correspond to TFFs if and only if the minimal such pair $(D_0,R_0)$ equates to a trivial $\TFF(D_0,N,R_0)$, namely if and only if either $D_0=R_0$ or $D_0=NR_0$.

\begin{theorem}
\label{thm.f<0}
If $f(D,N,R):=DNR-D^2-NR^2<0$ for some $(D,N,R)\in\bbZ^3$ with $D>0$, $N\geq 4$ and $R>0$ then $\Orb^{+}(D,N,R):=\set{(D',N,R')\in\Orb(D,N,R): D'>0, R'>0}$ is an infinite proper subset of $\Orb(D,N,R)$,
and contains no $(D_0,N,R_0)$ that satisfies~\eqref{eq.minimal},
but does contain a unique point $(D_0,N,R_0)$ such that $D_0\leq D'$ and $R_0\leq R'$ for all $(D',N,R')\in\Orb^{+}(D,N,R)$.
\smallskip

If $R_0\notin\set{\frac{D_0}N,D_0}$ then no $\TFF(D',N,R')$ exists for any $(D',N,R')\in\Orb(D,N,R)$.
If instead $R_0\in\set{\frac{D_0}N,D_0}$ then a $\TFF(D_0,N,R_0)$ exists,
and $\Orb^+(D,N,R)$ is the set of $(D',N,R')$ for which there exists a $\TFF(D',N,R')$ that can be obtained from it via iterated alternating Naimark and spatial complements.
In this case, any such $\TFF(D',N,R')$ is necessarily equichordal, but can only be  equi-isoclinic if $R'=R_0$ and $D'\in\set{D_0,NR_0-D_0}$.
Moreover, in this case $(D_0,N,R_0)$ is the only point $(D',N,R')\in\Orb^{+}(D,N,R)$ with $R'\in\set{\frac{D'}N,D'}$.
\end{theorem}

\begin{proof}
We begin by claiming that if $(D,N,R)\in\bbR^3$ satisfies $N\geq 4$ and $f(D,N,R)<0$ then either
\begin{equation}
\label{eq.no minimal}
(D<NR-D\text{ and }R>D-R)
\text{ or }
(D>NR-D\text{ and }R<D-R).
\end{equation}
Indeed, since $N\geq 4$ and $f(D,N,R)<0$,
\begin{align*}
(NR-2D)(D-2R)
&=DNR-2D^2-2NR^2+4DR\\
&=(DNR-D^2-NR^2)-(D^2-4DR+NR^2)\\
&=f(D,N,R)-(D-2R)^2-(N-4)R^2\\
&<0,
\end{align*}
and so one of the two quantities $NR-2D$ and $D-2R$ is positive while the other is negative.
Now fix any $(D,N,R)\in\bbZ^3$ with $N\geq 4$ and $f(D,N,R)<0$.
By applying $\nu$ if necessary,
we assume without loss of generality that $D\leq NR-D$,
and so by~\eqref{eq.no minimal} that $D<NR-D$.
That is, we assume that the Naimark-spatial sequence $\set{(D^{(K)},N,R^{(K)})}_{K=-\infty}^{\infty}$ of $(D,N,R)$ (see Definition~\ref{def.orbit}) satisfies $D^{(0)}<D^{(1)}$.
(Replacing $(D,N,R)$ with $\nu(D,N,R)$ shifts and reverses $\set{(D^{(K)},N,R^{(K)})}_{K=-\infty}^{\infty}$,
but leaves $\Orb(D,N,R)=\set{(D^{(K)},N,R^{(K)}): K\in\bbZ}$ unchanged.)
Since $f(D^{(K)},N,R^{(K)})=f(D,N,R)<0$ for all $K$,
combining this assumption with~\eqref{eq.no minimal} gives that
\begin{equation*}
D^{(2J-1)}=D^{(2J)}<D^{(2J+1)},
\quad
R^{(2J-1)}<R^{(2J)}=R^{(2J+1)},
\quad
\forall J\in\bbZ.
\end{equation*}
(We leave a formal inductive proof to the interested reader.)
In particular, both \smash{$\set{D^{(K)}}_{K=-\infty}^{\infty}$} and \smash{$\set{R^{(K)}}_{K=-\infty}^{\infty}$} are nondecreasing and unbounded both above and below.
Now define
\begin{equation*}
K_0:=\min\set{K: D^{(K)}>0\text{ and }R^{(K)}>0},
\quad
D_0:=D^{(K_0)},
\quad
R_0:=R^{(K_0)}.
\end{equation*}
Clearly,
$\Orb^{+}(D,N,R)
:=\set{(D',N,R')\in\Orb(D,N,R): D'>0, R'>0}$ can be reexpressed as
$\Orb^{+}(D,N,R)=\set{(D^{(K)},N,R^{(K)}): K\geq K_0}$,
and so is an infinite proper subset of $\Orb(D,N,R)$.
It is also clear that $(D_0,N,R_0)\in\Orb^{+}(D,N,R)$ satisfies $D_0\leq D'$ and $R_0\leq R'$ for all $(D',N,R')\in\Orb^{+}(D,N,R)$,
and that such a point is necessarily unique.
We caution however that $(D_0,N,R_0)$ is not ``minimal" in the sense of~\eqref{eq.minimal}:
indeed, for any $(D_0,N,R_0)\in\Orb(D,N,R)$ we have $N\geq 4$ and $f(D_0,N,R_0)=f(D,N,R)<0$, implying via~\eqref{eq.no minimal} a contradiction of~\eqref{eq.minimal}.

Continuing, note that since
$(D_0,N,R_0)$ lies in $\Orb^{+}(D,N,R)$ while
$(D^{(K_0-1)},N,R^{(K_0-1)})$ does not,
we have both $D_0>0$ and $R_0>0$
while either $D^{(K_0-1)}\leq 0$ or $R^{(K_0-1)}\leq 0$.
The ramifications of this fact depend on whether $K_0$ is even or odd.
If $K_0$ is even,
\begin{equation*}
(D^{(K_0-1)},N,R^{(K_0-1)})
=\sigma(D^{(K_0)},N,R^{(K_0)})
=\sigma(D_0,N,R_0)
=(D_0,N,D_0-R_0),
\end{equation*}
implying $D^{(K_0-1)}=D_0>0$ and so
$D_0-R_0
=R^{(K_0-1)}
\leq 0$.
If instead $K_0$ is odd,
\begin{equation*}
(D^{(K_0-1)},N,R^{(K_0-1)})
=\nu(D^{(K_0)},N,R^{(K_0)})
=\nu(D_0,N,R_0)
=(NR_0-D_0,N,R_0),
\end{equation*}
implying $R^{(K_0-1)}=R_0>0$ and so
$NR_0-D_0
=D^{(K_0-1)}
\leq 0$.
Thus, in general, either $D_0\leq R_0$ or $NR_0\leq D_0$.
Note that both of these inequalities cannot hold simultaneously:
since $R_0>0$, having $NR_0\leq D_0\leq R_0$ implies $N\leq 1$, a contradiction.
In particular, any $\TFF(D_0,N,R_0)$ has either a Naimark or spatial complement but not the other.

Note that if a $\TFF(D_0,N,R_0)$ exists then the parameters of any $\TFF(D',N,R')$ constructed from it via iterated alternating Naimark and spatial complements satisfy $D'>0$, $R'>0$ and $(D',N,R')\in\Orb(D,N,R)$, that is, $(D',N,R')\in\Orb^+(D,N,R)$.
Conversely, if a $\TFF(D_0,N,R_0)$ exists then for any $(D',N,R')\in\Orb^+(D,N,R)$,
a $\TFF(D',N,R')$ can be constructed from it in this way.
To be precise, writing $(D',N,R')=(D^{(K_1)},N,R^{(K_1)})$ where $K_1\geq K_0$,
a $\TFF(D',N,R')$ arises from a $\TFF(D_0,N,R_0)$ via $K_1-K_0$ such complements in total,
beginning with a Naimark complement when $K_0$ is even and with a spatial complement when $K_0$ is odd.
Note that by the same logic,
if a $\TFF(D',N,R')$ exists for some $(D',N,R')\in\Orb^+(D,N,R)$ then a $\TFF(D_0,N,R_0)$ can be constructed from it by applying these same complements in the reverse order.

Now note that if $R_0\notin\set{\frac{D_0}N,D_0}$ then since $D_0\leq R_0$ or $NR_0\leq D_0$ in general, either $D_0<R_0$ or $NR_0<D_0$.
In either case, recall from Section~2 that a $\TFF(D_0,N,R_0)$ does not exist.
This in turn implies that no $\TFF(D',N,R')$ exists for any $(D',N,R')\in\Orb(D,N,R)$:
such a TFF is nonsensical if either $D'\leq 0$ or $R'\leq 0$,
while if $D'>0$ and $R'>0$ then $(D',N,R')\in\Orb^+(D,N,R)$, and so the existence of a $\TFF(D',N,R')$ would imply that a $\TFF(D_0,R,N_0)$ exists, a contradiction.

If instead $R_0\in\set{\frac{D_0}N,D_0}$, as we assume to be the case for the remainder of this proof, recall from Section~2 that a $\TFF(D_0,N,R_0)$ exists, and that any such TFF is necessarily equi-isoclinic.
In this case, as already noted above, $\Orb^+(D,N,R)$ is the set of $(D',N,R')$ for which there exists a $\TFF(D',N,R')$ that can be obtained from it via iterated alternating Naimark and spatial complements,
and so any such $\TFF(D',N,R')$ is necessarily equichordal.

Next, we show that if $(D',N,R')\in\Orb^+(D,N,R)$ satisfies $R'\in\set{\frac{D'}N,D'}$ then necessarily $(D',N,R')=(D_0,N,R_0)$.
If $D'=NR'$ then both $\sigma(D',N,R')=(NR',N,(N-1)R')$ and $(D',N,R')$ have all positive entries while $\nu(D',N,R')=(0,N,R')$ does not.
In this case,
$(D',N,R')$ is necessarily the minimally-indexed member of  $\set{(D^{(K)},N,R^{(K)})}_{K=-\infty}^{\infty}$ that has all-positive entries,
that is, $(D',N,R')=(D^{(K_0)},N,R^{(K_0)})=(D_0,N,R_0)$.
Similarly, if $D'=R'$ then both $\nu(D',N,R')=((N-1)R',N,R')$ and $(D',N,R')$ have all positive entries while $\sigma(D',N,R')=(R',N,0)$ does not,
and so $(D',N,R')=(D^{(K_0)},N,R^{(K_0)})=(D_0,N,R_0)$.

To conclude, fix any $(D',N,R')\in\Orb^{+}(D,N,R)$ for which an $\EITFF(D',N,R')$ exists.
We will show that $R'=R_0$ and $D'\in\set{D_0,NR_0-D_0}$.
If $D'=NR'$, then as noted above, $(D',N,R')=(D_0,N,R_0)$ and so $R'=R_0$ and $D'=D_0\in\set{D_0,NR_0-D_0}$.
If instead $D'<NR'$ then this $\EITFF(D',N,R')$ has a Naimark complement,
and at least one of these two Naimark complementary TFFs is an $\EITFF(D'',N,R')$ where $D''\in\set{D',NR'-D'}$ and $D''\leq NR'-D''$.
In fact, since $(D'',N,R')\in\Orb(D,N,R)$ we have $f(D'',N,R')=f(D,N,R)<0$ and so~\eqref{eq.no minimal} gives $D''<NR'-D''$ and $R'>D''-R'$.
Now recall from Section~2 that any such $\EITFF(D'',N,R')$ with $D''<2R'$ necessarily has $D''=R'$.
Since $(D'',N,R')\in\Orb^{+}(D,N,R)$ satisfies $D''=R'$ recall from above that $(D'',N,R')=(D_0,N,R_0)$.
Thus, $R'=R_0$ and $D_0=D''\in\set{D',NR'-D'}=\set{D',NR_0-D'}$,
that is, $D'\in\set{D_0,NR_0-D_0}$.
\end{proof}

Taken together, Theorems~\ref{thm.N=2,3} and~\ref{thm.f<0} fully characterize the existence of all (real and/or equichordal and/or equi-isoclinic) $\TFF(D,N,R)$ with $f(D,N,R)<0$.
We caution that some of the nonexistence implications of Theorem~\ref{thm.f<0} are subtle, ruling out the existence of certain $\TFF(D,N,R)$ that satisfy the basic requirement that $R\leq D\leq NR$.
For example,
no $\TFF(7,4,2)$ exists despite the fact that $2\leq 7\leq(4)(2)$,
since if one did then its Naimark complement would be a $\TFF(1,4,2)$ where $1<2$.
Similarly, no $\TFF(5,4,4)$ exists despite the fact that $4\leq 5\leq(4)(4)$ since if one did then its spatial complement would be a $\TFF(5,4,1)$ where $5>(4)(1)$.
Such concerns lie at the heart of the TFF existence test of~\cite{CasazzaFMWZ11}.
Here, we have bypassed this issue:
having $f(D,N,R)<0$ neither implies that a $\TFF(D,N,R)$ exists nor that it does not,
but in either case, it does imply that existence problem is already settled.
In Section~5, we briefly revisit the TFF existence test of~\cite{CasazzaFMWZ11} from this new perspective.

\subsection{Naimark-spatial orbits when $f(D,N,R)>0$}

Due to the importance of this case in the study of still-open problems regarding the existence of (real and/or equichordal and/or equi-isoclinic) TFFs, we begin with an example.
For the aforementioned $\EITFF(3,7,1)$ considered in~\eqref{eq.(3,7,1) orbit},
\begin{equation*}
f_7(3,1)
=f(3,7,1)
=(3)(7)(1)-(3)^2-(7)(1)^2
=5
>0.
\end{equation*}
In this case the corresponding level set $\set{(D,R)\in\bbR^2: f_7(D,R)=7DR-D^2-7R^2=5}$ of the quadratic form $f_7$ is a hyperbola that contains $(3,1)$.
As depicted in Figure~\ref{fig.(7,3,1) orbit},
one of the two connected components of this hyperbola is contained in the subset $\set{(D,R)\in\bbR^2: 0<R<D<7R}$ of the first quadrant of the $(D,R)$-plane.
(Its other connected component is the negation of this one.)
As such, any $\ECTFF(D,7,R)$ whose $(D,R)$ parameters lie on this hyperbola has both a Naimark and spatial complement.
Their $(D',R')$ parameters are obtained by moving horizontally or vertically, respectively, from $(D,R)$ to another point on this hyperbola.
Interestingly, the path from $(3,1)$ to $(4,1)$ to $(4,3)$ to $(17,3)$, etc.,  (alternating complements, beginning with Naimark)
``interweaves" with that from $(3,1)$ to $(3,2)$ to $(11,2)$ to $(11,9)$, etc.,
(alternating complements, beginning with spatial).
From this graph, it is clear that $\Orb(3,7,1)$ is infinite and moreover contains a unique point $(D_0,7,R_0)$ such that $D_0\leq D$ and $R_0\leq R$ for all $(D,7,R)\in\Orb(3,7,1)$, namely $(D_0,7,R_0)=(3,7,1)$.
It is also clear that $(3,7,1)$ is also the only point of $\Orb(3,7,1)$ that lies both below and to the left of its spatial and Naimark-complementary cousins,
namely that satisfies~\eqref{eq.minimal}.
In the next result, we formally state and prove these and other claims in general.

\begin{figure}[htb]

\pgfplotsset{every axis/.append style={
                    axis x line=left,    
                    axis y line=left,    
                    axis line style={->}, 
                    xlabel={$D$},          
                    ylabel={$R$},          
                    }}

\tikzset{>=stealth}

\tikzmath{
    \DMax = 20;
    \RMax = 20;
    \D    = 3;
    \N    = 7;
    \R    = 1;
    \C    = \D*\N*\R-\D^2-\N*\R^2;
    \DMin = sqrt((4*\C)/(\N-4));
    \Samples = 2^10;
    }

\begin{center}
\begin{tikzpicture}
    \begin{axis}[
        xmin=0,
        xmax=\DMax,
        ymin=0,
        ymax=\RMax,
        clip=false]
        \addplot [only marks,mark=*] coordinates {
            (17,14)
            (17,3)
            (4,3)
            (4,1)
            (3,1)
            (3,2)
            (11,2)
            (11,9)};
        \addplot [dotted] coordinates {
            (\DMax,14)
            (17,14)
            (17,3)
            (4,3)
            (4,1)
            (3,1)
            (3,2)
            (11,2)
            (11,9)
            (\DMax,9)};
        \addplot[dashed, domain=0:\DMax] {x} node[pos=1, right] {$R=D$};
        \addplot[dashed, domain=0:\DMax] {x*(1/\N)} node[pos=1, below right] {$R=\frac DN$};
        \addplot[dashed, domain=0:\DMax] {x*(1/2+sqrt(1/4-1/\N))} node[pos=1, right] {$R=[\tfrac12+(\tfrac14-\tfrac1N)^{\frac12}]D$};
        \addplot[dashed, domain=0:\DMax] {x*(1/2-sqrt(1/4-1/\N))} node[pos=1, right] {$R=[\tfrac12-(\tfrac14-\tfrac1N)^{\frac12}]D$};
        \addplot[solid, domain=\DMin:\DMax,samples=\Samples] {x/2+sqrt( (x/2)^2-((x^2+\C)/(\N)) )};
        \addplot[solid, domain=\DMin:\DMax,samples=\Samples] {x/2-sqrt( (x/2)^2-((x^2+\C)/(\N)) )};
    \end{axis}
\end{tikzpicture}
\end{center}

\caption{\label{fig.(7,3,1) orbit}
The Naimark-spatial orbit of an $\ETF(3,7)$, regarded as an $\EITFF(3,7,1)$.
Each node indicates the $(D,R)$ parameters of an $\ECTFF(D,7,R)$ that can be obtained from the $\EITFF(3,7,1)$ via iterated alternating Naimark and spatial complements; see~\eqref{eq.(3,7,1) orbit}.
The hyperbola on which they lie is a level set of the quadratic form
$f_7(D,R)=f(D,7,R)=7DR-D^2-7R^2$ of~\eqref{eq.quadratic form} at ``elevation" $f_7(3,1)=5>0$.
Here, since the function $f$ of~\eqref{eq.invariant} is invariant under the Naimark and spatial involutions of Definition~\ref{def.orbit},
taking a Naimark or spatial complement of one of these ECTFFs corresponds to traversing horizontally or vertically, respectively, from one node on this hyperbola to another.
Since $f_7(3,1)=5>0$, this component of this hyperbola lies between the lines with slopes \smash{$\tfrac12[1\pm(\tfrac{N-4}{N})^{\frac12}]$} that form the zero set of this quadratic form.
These two slopes are themselves between $\frac1N$ and $1$,
meaning any $\ECTFF(D,7,R)$ whose $(D,R)$ parameters lie on this hyperbola has both a Naimark and spatial complement.
From this graph, it is intuitively obvious that this orbit is infinite and contains exactly one point $(D_0,R_0)=(3,1)$ that lies both below and to the left of its nearest neighbors, that is, is minimal in the sense of~\eqref{eq.minimal}.
This point is also minimal in the sense that both $D_0\leq D$ and $R_0\leq R$ for all other nodes $(D,R)$.
To be clear, however, there are points on this component of this hyperbola with lesser values of $D$ and $R$.
In Theorem~\ref{thm.f>0}, we verify that these phenomena occur in general whenever $f(D,N,R)>0$.
In Section~5, we exploit this theory of minimal points to certify the novelty of certain ECTFFs constructed in Section~4.
}
\end{figure}
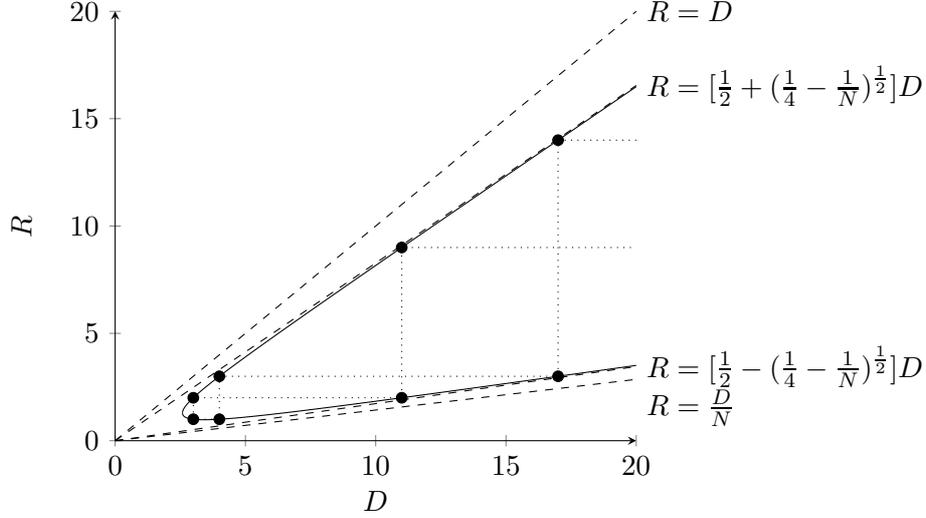

\begin{theorem}
\label{thm.f>0}
If $f(D,N,R):=DNR-D^2-NR^2>0$ for some $(D,N,R)\in\bbZ^3$ with $D>0$, $N>1$ and $R>0$ then $N>4$ and $\Orb(D,N,R)$ is an infinite subset of $\set{(D',N,R')\in\bbZ^3: D'>0, R'>0}$ that contains a unique point $(D_0,N,R_0)$ that satisfies~\eqref{eq.minimal}.
Moreover, $D_0\leq D'$ and $R_0\leq R'$ for all $(D',N,R')\in\Orb(D,N,R)$.\smallskip

Here, if a $\TFF(D,N,R)$ exists then $\Orb(D,N,R)$ is the set of $(D',N,R')$ for which there exists a $\TFF(D',N,R')$ that can be obtained from it via iterated alternating Naimark and spatial complements.
Any such $\TFF(D',N,R')$ can only be equi-isoclinic if $R'=R_0$ and $D'\in\set{D_0,NR_0-D_0}$.
\end{theorem}

\begin{proof}
Recall that the matrix that gives rise to the quadratic form $f_N$ of~\eqref{eq.quadratic form} is negative semidefinite when $N\in\set{2,3,4}$.
As such, if $f(D,N,R)>0$ for some $(D,N,R)\in\bbZ^3$ with $N>1$ then $N>4$.
Accordingly, for the remainder of the proof let $N>4$ be a fixed integer.
As noted in the proof of Theorem~\ref{thm.f=0},
for such $N$,
the zero set of $f_N$ is the union of the two (distinct) lines in the $(D,R)$-plane with slopes \smash{$\tfrac12[1\pm(\tfrac{N-4}{N})^{\frac12}]$}.
These slopes are bounded above by $1$ and below by $\tfrac 1N$:
since $N(N-4)<(N-2)^2$ we have
\smash{$(\tfrac{N-4}{N})^{\frac12}<\tfrac{N-2}{N}=1-\tfrac 2N$},
and so
\smash{$\tfrac1N<\tfrac12[1-(\tfrac{N-4}{N})^{\frac12}]$}.
Note $f_N$ is negative at nonzero points of these bounding lines:
\begin{align}
\nonumber
f_N(D,\tfrac DN)
&=DN\tfrac DN-D^2-N(\tfrac DN)^2
=-\tfrac{D^2}N
<0,\\
\label{eq.f is negative when D=R}
f_N(D,D)
&=DND-D^2-ND^2
=-D^2
<0,
\end{align}
for all $D\neq 0$.
Meanwhile it is positive at the nonzero points on the line of slope $\frac12$:
\begin{equation*}
f_N(D,\tfrac D2)
=DN\tfrac D2-D^2-N(\tfrac D2)^2
=(\tfrac N4-1)D^2
>0,
\end{equation*}
for all $D\neq 0$.
In particular, $\set{(D,R)\in\bbR^2: f_N(D,R)>0}$ is a subset of the union of the first and third quadrants, with its first-quadrant component lying within the cone
$0<\frac DN<R<D$, that is,
\begin{equation}
\label{eq.positive component}
\set{(D,R)\in\bbR^2: D>0,\, R>0,\, f_N(D,R)>0}\\
\subseteq\set{(D,R)\in\bbR^2: 0<R<D<NR}.
\end{equation}

At this stage, fix positive integers $D$ and $R$ such that $f(D,N,R)>0$,
let $C=f(D,N,R)$,
and define $\Orb(D,N,R)$ as in Definition~\ref{def.orbit}.
Note that $(D,N,R)$ is a point in the set
\begin{equation}
\label{eq.orbit superset}
\set{(D',N,R')\in\bbZ^3: D'>0,\, R'>0,\, f(D',N,R')=C}.
\end{equation}
To show that $\Orb(D,N,R)$ is contained in
$\set{(D',N,R')\in\bbZ^3: D>0,\, R>0}$ as claimed it thus suffices to show that~\eqref{eq.orbit superset} is invariant under both $\nu$ and $\sigma$.
To see this, note that if $(D',N,R')\in\bbZ^3$ then
$\nu(D',N,R')=(NR'-D',N,R')$ and $\sigma(D',N,R')=(D',N,D'-R')$ lie in $\bbZ^3$
and satisfy $f(\nu(D',N,R'))=f(\sigma(D',N,R'))=f(D',N,R')=C$.
Moreover, by~\eqref{eq.positive component}, $0<R'<D'<NR'$ and so all entries of both $\nu(D',N,R')$ and $\sigma(D',N,R')$ are positive.

We next show that $\Orb(D,N,R)$ contains a point $(D_0,N,R_0)$ that is minimal in the sense of~\eqref{eq.minimal}.
Let $R_0:=\min\set{R': (D',N,R')\in\Orb(D,N,R)}$.
(This is well-defined since
$\Orb(D,N,R)$ is contained in $\set{(D',N,R')\in\bbZ^3: D>0,\, R>0}$.)
Take $D_0$ such that $(D_0,N,R_0)\in\Orb(D,N,R)$.
By applying $\nu$ if necessary,
we may assume without loss of generality that $0<D_0\leq NR_0-D_0$,
namely that $(D_0,N,R_0)$ satisfies the first condition of~\eqref{eq.minimal}.
At the same time, note that since $(D_0,N,D_0-R_0)=\sigma(D_0,N,R_0)\in\Orb(D,N,R)$,
the definition of $R_0$ gives $0<R_0\leq D_0-R_0$,
namely that $(D_0,N,R_0)$ also satisfies the second condition of~\eqref{eq.minimal}.

As we now explain, $(D_0,N,R_0)$ is also a minimal point of $\Orb(D,N,R)$ in the traditional sense, that is, satisfies $D_0\leq D'$ and $R_0\leq R'$ for all $(D',N,R')\in\Orb(D,N,R)$.
To see this, let $\set{(D^{(K)},N,R^{(K)})}_{K=-\infty}^{\infty}$ be the Naimark spatial sequence of $(D_0,N,R_0)$ as defined by~\eqref{eq.Naimark spatial sequence} when $(D^{(0)},N,R^{(0)}):=(D_0,N,R_0)$.
(We caution that since $(D_0,N,R_0)$ is not necessarily equal to $(D,N,R)$,
their respective Naimark-spatial sequences might not be equal.
That said, since $(D_0,N,R_0)\in\Orb(D,N,R)$,
each of these two sequences can be obtained by shifting and/or reversing the other,
and $\Orb(D_0,N,R_0)=\Orb(D,N,R)$ regardless.)

To proceed, we formally prove something evidenced in graphs of such orbits,
such as Figure~\ref{fig.(7,3,1) orbit}: when $N>4$,
any rightwards move in $\Orb(D,N,R)$ is followed by one upwards,
while any upwards move is followed by one rightwards.
To be precise, if $(D',N,R')\in\Orb(D,N,R)$ and $D'\leq NR'-D'$
(that is, if the first entry of $(D',N,R')$ is no more than that of $\nu(D',N,R')$),
then since $N>4$ (and so $\tfrac{N}{2}<N-2$),
\begin{equation*}
D'
\leq\tfrac{N}{2}R'
<(N-2)R'
=(N-1)R'-R',
\quad
\text{i.e.,}
\quad
R'<(N-1)R'-D',
\end{equation*}
namely that the third entry of $\nu(D',N,R')$ is less than that of
\begin{equation}
\label{eq.Naimark then spatial}
\sigma(\nu(D',N,R'))
=\sigma(NR'-D',N,R')
=(NR'-D',(N-1)R'-D').
\end{equation}
Similarly, if $(D',N,R')\in\Orb(D,N,R)$ and $R'\leq D'-R'$ (that is, if the third entry of $(D',N,R')$ is no more than that of $\sigma(D',N,R')$) then
\begin{equation*}
NR'
\leq\tfrac{N}{2}D'
<(N-2)D'
=(N-1)D'-D',
\quad
\text{i.e.,}
\quad
D'<(N-1)D'-NR',
\end{equation*}
namely that the first entry of $\sigma(D',N,R')$ is less than that of
\begin{equation}
\label{eq.spatial then Naimark}
\nu(\sigma(D',N,R'))
=\nu(D',N,D'-R')
=((N-1)D'-NR',D'-R').
\end{equation}
In particular,
since $D_0\leq NR_0-D_0$, iteratively applying these facts to the positively-indexed portion of the Naimark-spatial sequence~\eqref{eq.Naimark spatial sequence} of $(D_0,N,R_0)$ gives
\begin{align}
\label{eq.right then up}
\begin{split}
D_0 & = D^{(0)} \leq D^{(1)} = D^{(2)} < D^{(3)} = D^{(4)} < D^{(5)} =\dotsb,\\
R_0 & = R^{(0)}    = R^{(1)} < R^{(2)} = R^{(3)} < R^{(4)} = R^{(5)} <\dotsb,
\end{split}
\end{align}
where $D^{(2J-1)}=D^{(2J)}<D^{(2J+1)}$ and $R^{(2J-1)}<D^{(2J)}=D^{(2J+1)}$ for all $J\geq 1$.
Since $R_0\leq D-R_0$, iteratively applying these same facts to the negatively-indexed portion of this sequence also gives
\begin{align}
\label{eq.up then right}
\begin{split}
D_0 & = D^{(0)}    = D^{(-1)} < D^{(-2)} = D^{(-3)} < D^{(-4)} =D^{(-5)} <\dotsb,\\
R_0 & = R^{(0)} \leq R^{(-1)} = R^{(-2)} < R^{(-3)} = R^{(-4)} <R^{(-5)} =\dotsb,
\end{split}
\end{align}
where $D^{(-2J+1)}<D^{(-2J)}=D^{(-2J-1)}$ and $R^{(-2J+1)}=D^{(-2J)}<D^{(-2J-1)}$ for all $J\geq 1$.
(We leave formal inductive proofs of these facts to the interested reader.)
In particular, $\Orb(D,N,R)$ has infinite cardinality,
and both $D_0\leq D'$ and $R_0\leq R'$ for all $(D',N,R')\in\Orb(D,N,R)$.
In fact, our argument implies that any $(D_0,N,R_0)\in\Orb(D,N,R)$ that satisfies~\eqref{eq.minimal} necessarily has $D_0=\min\set{D': (D',N,R')\in\Orb(D,N,R)}$ and $R_0=\min\set{R': (D',N,R')\in\Orb(D,N,R)}$,
and so such a point $(D_0,N,R_0)$ is necessarily unique.

To prove the final parts of this result,
recall that since $f(D,N,R)>0$, $\Orb(D,N,R)$ is contained in $\set{(D',N,R')\in\bbZ^3: D'>0, R'>0}$, and so any $\TFF(D',N,R')$ with $(D',N,R')\in\Orb(D,N,R)$ has both a Naimark and spatial complement.
Moreover, since any member of $\Orb(D,N,R)$ can be obtained from any other via iterated alternating Naimark and spatial involutions,
if a (real and/or equichordal) $\TFF(D',N,R')$ exists for any $(D',N,R')\in\Orb(D,N,R)$ then one exists for all $(D',N,R')\in\Orb(D,N,R)$.
In this case, $\Orb(D,N,R)$ is the set of $(D',N,R')$ for which there exists a $\TFF(D',N,R')$ that can be obtained from a $\TFF(D,N,R)$ in this manner.
Only a scant few of these might be equi-isoclinic.
To see this, assume an $\EITFF(D',N,R')$ exists for some $(D',N,R')\in\Orb(D,N,R)$.
By taking its Naimark complement if necessary,
we then have that an $\EITFF(D'',N,R')$ exists for some $(D'',N,R')\in\Orb(D,N,R)$ that satisfies $D''\leq NR'-D''$ where $D''\in\set{D',NR'-D'}$.
From Section~2, recall that an $\EITFF(D'',N,R')$ can only exist if either $2R'\leq D''$ or $D''=R'$.
The latter case cannot happen here since~\eqref{eq.f is negative when D=R} gives $f(D'',N,D'')<0$ whereas the invariance of $f$ over $\Orb(D,N,R)$ guarantees $f(D'',N,R')=f(D,N,R)>0$.
Thus, $2R'\leq D''$.
Since $(D'',N,R')\in\Orb(D,N,R)$ satisfies both $D''\leq NR'-D''$ and $2R'\leq D''$,
it is minimal~\eqref{eq.minimal},
implying $R'=R_0$ and $D''=D_0$ (and so $D'\in\set{D_0,NR_0-D_0}$), as claimed.
\end{proof}

As we have just seen, when $f(D,N,R)>0$, there are only ever at most two choices of triples in the infinite orbit of $(D,N,R)$ for which a corresponding EITFF might exist (and at most one such choice when $D_0=NR_0-D_0$),
namely those corresponding to the lowest point(s) $\set{(D_0,R_0),(NR_0-D_0,R_0)}$ on the ``Naimark-spatial path" on its associated hyperbola.
We remark that more generally, by combining~\eqref{eq.right then up} and~\eqref{eq.up then right} with the facts from Section~2 about how the singular values of cross-Gram matrices evolve with respect to Naimark and spatial complements,
one finds that for any $J>1$, there are at least $J$ distinct principal angles between any two subspaces of any $\TFF(D^{(K)},N,R^{(K)})$ when either $K\geq 2J$ or $K\leq -2J-1$ (under the assumptions and notation of Theorem~\ref{thm.f>0} and its above proof).

We also remark that when $f(D,N,R)>0$, \eqref{eq.right then up} and~\eqref{eq.up then right} imply that $(D_0,N,R_0)$ is the only point $(D',N,R')$ of $\Orb(D,N,R)$ that might be a fixed point of either $\nu$ or $\sigma$,
that is, have either $D'=NR'-D'$ or $R'=D-R'$.
Here, unlike in Theorem~\ref{thm.f=0}, it cannot be a fixed point of both:
if $D_0=NR_0-D_0$ and $R_0=D_0-R_0$ then $4R_0=2D_0=NR_0$ implying $N_0=4$ and $D_0=2R_0$,
and so $f(D,N,R)=f(D_0,N,R_0)=f(2R_0,4,R_0)=0$, a contradiction.
In such cases, the graph of the orbit is not ``braided" like that of Figure~\ref{fig.(7,3,1) orbit}, but rather appears as two copies of the same path emanating from $(D_0,R_0)$.
For context, note that in the proof of Theorem~\ref{thm.f<0} we showed that no point $(D,N,R)$ with $N\geq4$ and $f(D,N,R)<0$ can be a fixed point of either $\nu$ or $\sigma$.

We further note that the mapping $(D,R)\mapsto(NR-D,(N-1)R-D)$ essentially seen in~\eqref{eq.Naimark then spatial} is a linear transformation whose inverse is the mapping $(D,R)\mapsto((N-1)D-NR,D-R)$, cf.~\eqref{eq.spatial then Naimark}.
It turns out that its $2\times 2$ matrix representation is diagonalizable provided $N>4$.
This permits one to derive explicit closed-form expressions for all members of a Naimark-spatial sequence in terms of $(D,N,R)$ and $K$.
We do not do so here, since in our opinion, their technical nature only serves to obscure the delicate yet elementary arguments we have used up to this point.

\subsection{A Naimark-spatial necessary condition on the existence of tight fusion frames}

We conclude this section with a result that summarizes the ramifications of the previous ones in regards to the existence of (real and/or equichordal and/or equi-isoclinic) TFFs:

\begin{theorem}
\label{thm.summary}
If a $\TFF(D,N,R)$ exists then it can be obtained via iterated alternating Naimark and spatial complements from either:
\begin{enumerate}
\renewcommand{\labelenumi}{(\roman{enumi})}
\item
a $\TFF(D_0,N,R_0)$ with $R_0\in\set{\frac{D_0}N,D_0}$,
which is necessarily equi-isoclinic, or
\item
a $\TFF(D_0,N,R_0)$ where $(D_0,N,R_0)$ satisfies~\eqref{eq.minimal},
\end{enumerate}
depending on whether $f(D,N,R):=DNR-D^2-NR^2$ is negative or nonnegative, respectively.
In either case, $(D_0,N,R_0)$ is unique,
and this $\TFF(D,N,R)$ can only be equi-isoclinic if $R=R_0$ and $D\in\set{D_0,NR_0-D_0}$.
Any such $\TFF(D_0,N,R_0)$ is real and/or equichordal if and only if this $\TFF(D,N,R)$ is as well.
\end{theorem}

\begin{proof}
Fix any $\TFF(D,N,R)$ where $(D,N,R)\in\bbZ^3$ satisfies $D>0$, $N>1$ and $R>0$.
We usually assume $N>1$ by convention to avoid dividing by $0$ in~\eqref{eq.chordal and simplex packing bounds}.
That said, if one wishes to more generally consider TFFs for $\calH$ that consist of a single subspace $\calU$,
note that necessarily $\calU=\calH$.
This $\TFF(D,1,D)$ has neither a Naimark or spatial complement,
but is vacuously equi-isoclinic,
and $(D_0,N,R_0)=(D,1,D)$ satisfies $R_0=D\in\set{D}=\set{\frac{D_0}N,D_0}$.
That is, (i) applies even in this degenerate case.

Now for the moment assume that this $\TFF(D,N,R)$ arises from a $\TFF(D_0,N,R_0)$ via iterated alternating Naimark and spatial complements.
Clearly, if one of these two TFFs is real and/or equichordal then the other is as well.
Further recall that since $f$ is invariant with respect to $\nu$ and $\sigma$, $f(D,N,R)=f(D_0,N,R_0)$.
In this case, if $R_0\in\set{\frac{D_0}N,D_0}$ then $f(D,N,R)<0$ since either $f(D,N,R)=f(NR_0,N,R_0)=-NR_0^2$ or $f(D,N,R)=f(R_0,N,R_0)=-R_0^2$.
If instead $(D_0,N,R_0)$ satisfies~\eqref{eq.minimal} then $0<4R_0\leq 2D_0\leq NR_0$ and so $N\geq 4$,
at which point Theorem~\ref{thm.f<0} implies that $f(D_0,N,R_0)\geq 0$.
We now use Theorems~\ref{thm.N=2,3}--\ref{thm.f>0} to show that this $\TFF(D,N,R)$ indeed arises in exactly one of these two ways.

If $N=2$ then $f(D,N,R)<0$ and recall that Theorem~\ref{thm.N=2,3} gives that $R\in\set{\frac D2,D}$,
meaning that we can choose our desired $\TFF(D_0,N,R_0)$ be this $\TFF(D,N,R)$.
To be clear, in light of~\eqref{eq.NS path N=2}, we are also free to choose our $\TFF(D_0,N,R_0)$ to be the spatial or Naimark complement of this $\TFF(D,N,R)$ when $R=\frac D2$ and $R=D$, respectively.
Regardless, $(D_0,N,R_0)=(D,N,R)$.)
Here, the condition $R=R_0$, $D\in\set{D_0,NR_0-D_0}$ is automatically satisfied.

If $N=3$ then $f(D,N,R)<0$ and Theorem~\ref{thm.N=2,3} gives that
$R$ is $\frac D3$, $\frac D2$, $\frac{2D}3$, or $D$.
When either $R=\frac D3$ or $R=D$ we can just let our desired $\TFF(D_0,N,R_0)$ be this $\TFF(D,N,R)$.
If instead $R=\frac D2$ then~\eqref{eq.NS path N=3 case 2} suggests we let our $\TFF(D_0,N,R_0)$ be the Naimark complement of this $\TFF(D,N,R)$,
implying $(D_0,N,R_0)=\nu(D,3,\tfrac D2)=(\tfrac D2,3,\tfrac D2)$ and so $R_0=D_0$.
Similarly, if $R=\frac{2D}3$ then~\eqref{eq.NS path N=3 case 1} suggests we let our $\TFF(D_0,N,R_0)$ be the spatial complement of this $\TFF(D,N,R)$,
implying $(D_0,N,R_0)=\sigma(D,3,\tfrac{2D}3)=(D,3,\tfrac D3)$ and so $R_0=\frac{D_0}N$.
In fact, by~\eqref{eq.NS path N=3 case 1} and~\eqref{eq.NS path N=3 case 2},
when $R$ is $\frac D3$, $\frac D2$, $\frac{2D}3$ or $D$ we must choose
$(D_0,N,R_0)$ to be
\begin{equation*}
(D,3,\tfrac D3),
\qquad
(\tfrac D2,3,\tfrac D2),
\qquad
(D,3,\tfrac D3),
\qquad
(D,3,D),
\end{equation*}
respectively.
To be clear, though there is a unique choice of $(D_0,N,R_0)$ in each case,
the corresponding $\TFF(D_0,N,R_0)$ is not necessarily unique.
When $R=\frac D2$ for example,
both the Naimark complement of our given $\TFF(D,N,R)$ and the Naimark complement of its spatial complement are suitable $\TFF(D_0,N,R_0)$.
In all but the third of these four cases,
$(D,N,R)$ satisfies $R=R_0$ and $D\in\set{D_0,NR_0-D_0}$.
Meanwhile, in the third case, $R=\frac{2D}3$, implying $R<D<2R$ and so no $\EITFF(D,N,R)$ exists.
Thus, when $N=3$, if an $\EITFF(D,N,R)$ exists then indeed $R=R_0$ and $D\in\set{D_0,NR_0-D_0}$, as needed.

Next consider the case where $N\geq 4$ and $f(D,N,R)<0$.
Here, Theorem~\ref{thm.f<0} gives that this $\TFF(D,N,R)$ arises via iterated alternating Naimark complements from a $\TFF(D_0,N,R_0)$ with $R_0\in\set{\frac{D_0}N,D_0}$.
Such a $\TFF(D_0,N,R_0)$ is necessarily equi-isoclinic.
Theorem~\ref{thm.f<0} moreover gives that this $(D_0,N,R_0)$ is unique:
if $(D',N,R')\in\Orb^+(D,N,R)$ satisfies $R'\in\set{\frac{D'}N,D'}$ then $(D',N,R')=(D_0,N,R_0)$ where
\begin{equation*}
D_0=\min\set{D': (D',N,R')\in\Orb^+(D,N,R)},
\quad
R_0=\min\set{R': (D',N,R')\in\Orb^+(D,N,R)}.
\end{equation*}
In fact, recalling from the proof of Theorem~\ref{thm.f<0} that no member of the Naimark-spatial sequence of $(D,N,R)$ is a fixed point of either $\nu$ or $\sigma$,
our $\TFF(D_0,N,R_0)$ is itself unique.
Theorem~\ref{thm.f<0} further gives that this $\TFF(D,N,R)$ is equi-isoclinic only if $R=R_0$ and $D\in\set{D_0,NR_0-D_0}$.

Next, in the case where $f(D,N,R)=0$, Theorem~\ref{thm.f=0} gives that $N=4$ and $D=2R$.
In this case, letting $\TFF(D_0,N,R_0)$ be this $\TFF(D,N,R)$ we have $(D_0,N,R_0)=(D,N,R)=(2R,4,R)$ which satisfies~\eqref{eq.minimal}.
More generally, we could take $\TFF(D_0,N,R_0)$ to be any TFF obtained from this $\TFF(D,N,R)$ via iterated alternating Naimark and spatial complements.
Regardless, $(D_0,N,R_0)=(2R,4,R)$.
Here, $R=R_0$ and the condition $D\in\set{D_0,NR_0-D_0}$ is automatically satisfied since $D=2R$ and $\set{D_0,NR_0-D_0}=\set{2R}$.

Finally, in the case where $f(D,N,R)>0$,
Theorem~\ref{thm.f>0} gives that this $\TFF(D,N,R)$ arises via iterated alternating Naimark and spatial complements from a $\TFF(D_0,N,R_0)$
where $(D_0,N,R_0)$ is the unique member of $\Orb(D,N,R)$ that satisfies~\eqref{eq.minimal}.
(In fact, a careful analysis of the proof of Theorem~\ref{thm.f>0} reveals that this $\TFF(D_0,N,R_0)$ is itself unique when $4R_0<2D_0<NR_0$,
but in only unique up to spatial or Naimark complements when either $2R_0=D_0$ or $2D_0=NR_0$, respectively.)
Theorem~\ref{thm.f>0} moreover gives that this $\TFF(D,N,R)$ can only be equi-isoclinic if $R=R_0$ and $D\in\set{D_0,NR_0-D_0}$.
\end{proof}

\section{Equichordal tight fusion frames from difference families}

In this section we introduce a method for constructing an ECTFF from a difference family for a finite abelian group.
This method properly generalizes King's construction of an ECTFF from a semiregular divisible difference set~\cite{King16}.
Except in the most trivial cases, each $\ECTFF(D,N,R)$ that results from this process has $f(D,N,R)\geq 0$.
In the next section, we compare the minimal points of the orbits of these ECTFFs against those of previously known constructions, and certify that some of them are truly novel.

Let $\calG$ be a finite abelian group whose operation we denoted as addition.
A \textit{character} of $\calG$ is a homomorphism $\gamma:\calG\rightarrow\bbT:=\set{z\in\bbC: \abs{z}=1}$.
The \textit{(Pontryagin) dual} $\hat{\calG}$ of $\calG$ is the set of all characters of $\calG$,
which is itself an abelian group under pointwise multiplication.
In fact it is well known that $\hat{\calG}$ is isomorphic to $\calG$ in this setting.
The \textit{character table} $\bfGamma$ of $\calG$ is the synthesis operator of the sequence \smash{$\set{\gamma}_{\gamma\in\hat{\calG}}$} of all characters of $\calG$
(with each $\gamma$ serving as its own index),
namely that $\calG\times\hat{\calG}$ matrix whose $(g,\gamma)$th entry is $\bfGamma(g,\gamma)=\gamma(g)$.
It is well known that the characters of $\calG$ form an equal-norm orthogonal basis for $\bbC^\calG$, and so $\bfGamma$ is a complex Hadamard matrix, having $\bfGamma^{-1}=\frac1G\bfGamma^*$ where $G$ is the order of $\calG$.
Here, $\bfGamma^*$ is the \textit{discrete Fourier transform} (DFT) on $\calG$.
Since $\bfGamma^*$ is the analysis operator of the characters,
it satisfies $(\bfGamma^*\bfy)(\gamma)=\ip{\gamma}{\bfy}$ for all $\bfy\in\bbC^\calG$.

For any nonempty subset $\calD$ of $\calG$,
the corresponding \textit{harmonic frame} consists of the normalized restrictions of the characters of $\calG$ to it, namely $\set{\bfphi_\gamma}_{\gamma\in\hat{\calG}}$,
$\bfphi_\gamma(d):=D^{-\frac12}\gamma(d)$ where $D:=\#(\calD)$.
Any such frame is unit norm and moreover tight:
for any $d_1,d_2\in\calD$,
\begin{equation*}
(\bfPhi\bfPhi^*)(d_1,d_2)
=\sum_{\gamma\in\hat{\calG}}\bfphi_\gamma(d_1)\overline{\bfphi_\gamma(d_2)}
=\tfrac1D\sum_{\gamma\in\hat{\calG}}\gamma(d_1)\overline{\gamma(d_2)}
=\tfrac1D(\bfGamma\bfGamma^*)(d_1,d_2)
=\tfrac GD\bfI(d_1,d_2).
\end{equation*}
The entries of the Gram matrix of a harmonic frame are expressible in terms of the DFT of the characteristic function $\bfchi_\calD$ of $\calD$:
\begin{equation}
\label{eq.harmonic Gram}
\ip{\bfphi_{\gamma_1}}{\bfphi_{\gamma_2}}
=\sum_{d\in\calD}\overline{\bfphi_{\gamma_1}(d)}\bfphi_{\gamma_2}(d)
=\tfrac1D\sum_{g\in\calG}\overline{(\gamma_1^{}\gamma_2^{-1})(g)}\bfchi_\calD(g)
=\tfrac1D(\bfGamma^*\bfchi_\calD)(\gamma_1^{}\gamma_2^{-1}).
\end{equation}
To compute the modulus of such an entry,
we exploit the well-known ways in which the DFT distributes over the \textit{convolution} $\bfy_1*\bfy_2\in\bbC^\calG$ and \textit{involution} $\tilde{\bfy}\in\bbC^\calG$ of any $\bfy_1,\bfy_2,\bfy\in\bbC^\calG$, defined by $(\bfy_1*\bfy_2)(g):=\sum_{g'\in\calG}\bfy_1(g-g')\bfy_2(g')$ and $\tilde{\bfy}(g):=\overline{\bfy(-g)}$, respectively.
Specifically,
$[\bfGamma^*(\bfy_1*\bfy_2)](\gamma)
=(\bfGamma^*\bfy_1)(\gamma)(\bfGamma^*\bfy_2)(\gamma)$
and
$(\bfGamma^*\tilde{\bfy})(\gamma)=\overline{(\bfGamma^*\bfy)(\gamma)}$ for all $\gamma\in\hat{\calG}$.
In particular,
the DFT of the \textit{autocorrelation} $\tilde{\bfy}*\bfy$ of $\bfy\in\bbC^\calG$ satisfies
$[\bfGamma^*(\tilde{\bfy}*\bfy)](\gamma)=\abs{(\bfGamma^*\bfy)(\gamma)}^2$
for all $\gamma\in\hat{\calG}$,
and so is the pointwise modulus-squared of the DFT of $\bfy$.
Thus, for a harmonic frame $\set{\bfphi_\gamma}_{\gamma\in\hat{G}}$,
\begin{equation}
\label{eq.harmonic Gram modulus}
\abs{\ip{\bfphi_{\gamma_1}}{\bfphi_{\gamma_2}}}^2
=\tfrac1{D^2}\abs{(\bfGamma^*\bfchi_\calD)(\gamma_1^{}\gamma_2^{-1})}^2
=\tfrac1{D^2}\bfGamma^*(\tilde{\bfchi}_\calD*\bfchi_\calD)(\gamma_1^{}\gamma_2^{-1}).
\end{equation}
Here, for any $g\in\calG$, $(\tilde{\bfchi}_\calD*\bfchi_{\calD})(g)$ is both the cardinality of the intersection of $\calD$ with its $g$th shift,
as well as the number of times $g$ can be written as a difference of members of $\calD$:
since the mapping $d\mapsto(d,d-g)$ is a bijection from $\calD\cap(g+\calD)$ onto
$\set{(d_1,d_2)\in\calD\times\calD: g=d_1-d_2}$,
\begin{equation}
\label{eq.autocorrelation of characteristic}
(\tilde{\bfchi}_\calD*\bfchi_{\calD})(g)
=\sum_{g'\in\calG}\bfchi_{g+\calD}(g')\bfchi_\calD(g')
=\#[\calD\cap(g+\calD)]
=\#\set{(d_1,d_2)\in\calD\times\calD: g=d_1-d_2}.
\end{equation}
From this, we can begin to see a connection between harmonic ETFs, ECTFFs, etc., and subsets $\calD$ of finite abelian groups $\calG$ whose differences have various types of nice combinatorial properties.

For example, $\calD$ is a \textit{divisible difference set} (DDS) for $\calG$ with respect to some subgroup $\calH$ of $\calG$ of order $H$ if there exists constants $\Lambda_1$, $\Lambda_2$ such that
\begin{equation*}
\#\set{(d_1,d_2)\in\calD\times\calD: g=d_1-d_2}
=\left\{\begin{array}{cl}
\Lambda_1,&\ g\in\calH, g\neq0,\\
\Lambda_2,&\ g\notin\calH.
\end{array}\right.
\end{equation*}
Throughout, it will be clear from context whether a particular set $\calH$ is a Hilbert space or alternatively a subgroup of a finite abelian group.
We follow~\cite{Pott95} and denote such a set $\calD$ as a $\DDS(\frac GH,H,D,\Lambda_1,\Lambda_2)$.
By~\eqref{eq.autocorrelation of characteristic},
this equates to having
\begin{equation}
\label{eq.DDS autocorrelation}
\tilde{\bfchi}_\calD*\bfchi_{\calD}
=(D-\Lambda_1)\bfdelta_0+(\Lambda_1-\Lambda_2)\bfchi_\calH+\Lambda_2\bfchi_\calG.
\end{equation}
To proceed, we use the \textit{Poisson summation formula} (PSF),
which is a type of \textit{projection slice theorem},
and colloquially known in the signal processing community as the fact that ``the DFT of a comb is a comb."
Specifically,
letting
$\calH^\perp:=\set{\gamma\in\hat{G}:\ \gamma(h)=1,\,\forall\, h\in\calH}$
be the \textit{annihilator} of $\calH$,
the PSF states that $\bfGamma^*\bfchi_\calH=H\bfchi_{\calH^\perp}$,
or equivalently, that $\bfGamma\bfchi_{\calH^\perp}=\frac GH\bfchi_\calH$.
Here, $\calH^\perp$ is a subgroup of $\hat{G}$ of order $\frac GH$.
In fact, $\gamma\mapsto((g+\calH)\mapsto\gamma(g))$ is a well-defined isomorphism from $\calH^\perp$ onto the dual of $\calG/\calH$.
Applying the PSF to~\eqref{eq.DDS autocorrelation} gives that it is equivalent to having
\begin{equation}
\label{eq.DFT of DDS autocorrelation}
\abs{(\bfGamma^*\bfchi_{\calD})(\gamma)}^2
=(D-\Lambda_1)+H(\Lambda_1-\Lambda_2)\bfchi_{\calH^\perp}(\gamma)+G\Lambda_2\bfdelta_1(\gamma),
\quad
\forall\ \gamma\in\hat{\calG}.
\end{equation}
Here, evaluating the above equation at the identity character $1$ gives a necessary relationship between the five parameters of a DDS:
\begin{equation}
\label{eq.DDS parameter relation}
D^2
=(D-\Lambda_1)+H(\Lambda_1-\Lambda_2)+G\Lambda_2
=D+(H-1)\Lambda_1+(G-H)\Lambda_2.
\end{equation}
Combining it with~\eqref{eq.harmonic Gram modulus} and~\eqref{eq.DFT of DDS autocorrelation}, we see that $\calD$ is a DDS for $\calG$ with respect to $\calH$ if and only if
\begin{equation}
\label{eq.harmonic from DDS}
\abs{\ip{\bfphi_{\gamma_1}}{\bfphi_{\gamma_2}}}^2
=\tfrac1{D^2}\abs{(\bfGamma^*\bfchi_{\calD})(\gamma_1^{}\gamma_2^{-1})}^2
=\tfrac1{D^2}\left\{\begin{array}{cl}
D^2-G\Lambda_2,&\ \gamma_1\calH^\perp=\gamma_2\calH^\perp,\,\gamma_1\neq\gamma_2,\\
D-\Lambda_1,&\ \gamma_1\calH^\perp\neq\gamma_2\calH^\perp.
\end{array}\right.
\end{equation}
That is, $\calD$ is a DDS if and only if the magnitude of the inner product of two distinct vectors in its corresponding harmonic frame depends solely on whether their indices lie in a common coset of $\calH^\perp$.
In at least four special cases,
this theory leads to a useful class of harmonic frames.

The first case is when $\calH=\set{0}$,
namely when $\calD$ is a \textit{difference set} for $\calG$.
Here, \eqref{eq.DDS parameter relation} becomes \smash{$\Lambda_2=\frac{D(D-1)}{G-1}$} and $\calH^\perp=\hat{\calG}$, meaning~\eqref{eq.harmonic from DDS} holds if and only if
\smash{$\abs{\ip{\bfphi_{\gamma_1}}{\bfphi_{\gamma_2}}}^2=\frac{G-D}{D(G-1)}$} whenever $\gamma_1\neq\gamma_2$,
namely if and only if $\set{\bfphi_\gamma}_{\gamma\in\hat{\calG}}$ is an ETF for $\bbC^\calD$.
This is the well-known equivalence between difference sets and harmonic ETFs~\cite{Konig99,StrohmerH03,XiaZG05,DingF07}.
In the next two cases, $\calH\neq\set{0}$ and $\calD$ is either \textit{semiregular} or a \textit{relative difference set} (RDS) for $\calG$.

A DDS $\calD$ is semiregular if $D^2=G\Lambda_2$ and $D\neq\Lambda_1$.
By~\eqref{eq.DDS parameter relation} this equates to  \smash{$\Lambda_1=\tfrac{D(DH-G)}{G(H-1)}$} and $\emptyset\neq\calD\neq\calG$.
In this case,
\eqref{eq.harmonic from DDS} becomes
\begin{equation}
\label{eq.harmonic from semiregular DDS}
\abs{\ip{\bfphi_{\gamma_1}}{\bfphi_{\gamma_2}}}^2
=\left\{\begin{array}{cl}
0,&\ \gamma_1\calH^\perp=\gamma_2\calH^\perp,\,\gamma_1\neq\gamma_2,\\
\tfrac{H(G-D)}{DG(H-1)},&\ \gamma_1\calH^\perp\neq\gamma_2\calH^\perp.
\end{array}\right.
\end{equation}
Here every subsequence $\set{\bfphi_{\gamma'}}_{\gamma'\in\gamma\calH^\perp}$ of the harmonic frame that is indexed by a coset of $\calH^\perp$ is orthonormal.
Since every harmonic frame is tight,
this immediately implies that \smash{$\set{\calU_{\gamma\calH^\perp}}_{\gamma\calH^\perp\in\hat{\calG}/\calH}$},
\smash{$\calU_{\gamma\calH^\perp}:=\Span\set{\bfphi_{\gamma'}}_{\gamma'\in\gamma\calH^\perp}$}
defines a $\TFF(D,H,\frac GH)$ for $\bbC^\calD$.
In fact, as noted in~\cite{King16}, it is an $\ECTFF(D,H,\frac GH)$ for $\bbC^\calD$ since whenever $\gamma_1\calH^\perp\neq\gamma_2\calH^\perp$,
\eqref{eq.harmonic from DDS} implies that every entry of the
corresponding cross-Gram matrix $\bfPhi_{\gamma_1\calH^\perp}^*\bfPhi_{\gamma_2\calH^\perp}^{}$ has modulus \smash{$[\tfrac{H(G-D)}{DG(H-1)}]^{\frac12}$}, and so
\begin{equation*}
\norm{\bfPhi_{\gamma_1\calH^\perp}^*\bfPhi_{\gamma_2\calH^\perp}^{}}_\Fro^2
=(\tfrac GH)^2\tfrac{H(G-D)}{DG(H-1)}
=\tfrac{G(G-D)}{DH(H-1)}.
\end{equation*}
As expected, this equals the constant $\tfrac{R(NR-D)}{D(N-1)}$ from~\eqref{eq.chordal Welch} when $(D,N,R)=(D,H,\tfrac GH)$.

Meanwhile, an RDS for $\calG$ is a DDS for $\calG$ with $\Lambda_1=0$,
meaning no two distinct members of $\calD$ differ by a member of $\calH$.
In this case, \eqref{eq.DDS parameter relation} and~\eqref{eq.harmonic from DDS} become \smash{$\Lambda_2=\frac{D(D-1)}{G-H}$} and
\begin{equation}
\label{eq.harmonic from RDS}
\abs{\ip{\bfphi_{\gamma_1}}{\bfphi_{\gamma_2}}}^2
=\left\{\begin{array}{cl}
\frac{G-DH}{D(G-H)},&\ \gamma_1\calH^\perp=\gamma_2\calH^\perp,\,\gamma_1\neq\gamma_2,\\
\frac1D,&\ \gamma_1\calH^\perp\neq\gamma_2\calH^\perp,
\end{array}\right.
\end{equation}
respectively.
In this case, \smash{$[\frac{G-DH}{D(G-H)}]^{\frac12}$} is the Welch bound for $\#(\calH^\perp)=\frac GH$ vectors in $\bbC^\calD$,
meaning every subsequence $\set{\bfphi_{\gamma'}}_{\gamma'\in\gamma\calH^\perp}$ of the harmonic frame that is indexed by a coset of $\calH^\perp$ is an $\ETF(D,\frac GH)$ for $\bbC^\calD$.
Such \textit{mutually unbiased} ETFs were recently used to construct new ETFs~\cite{FickusM21}.

The fourth case is the intersection of the previous two:
comparing~\eqref{eq.harmonic from semiregular DDS} and~\eqref{eq.harmonic from RDS} reveals that a DDS $\calD$ is both semiregular and an RDS if and only if $G=DH$ where $1<D<G$.
As noted in~\cite{GodsilR09}, such a harmonic frame consists of $H$ \textit{mutually unbiased bases} for $\bbC^\calD$.

We now give a new way to construct ECTFFs via harmonic frames.
This method properly generalizes King's aforementioned construction of ECTFFs from semiregular DDSs~\cite{King16}.
In particular, we fully characterize when partitioning a harmonic frame's vectors according to the cosets of $\calH^\perp$ yields orthonormal bases for equichordal subspaces.
Our approach here parallels that of~\cite{FickusS20},
which refines some of the analysis of~\cite{FickusJKM18} concerning harmonic ETFs that partition into regular simplices for their span.
Like~\cite{FickusS20}, our ``harmonic ECTFF" characterization exploits a certain generalization of a difference set known as a \textit{difference family} (DF).
In this characterization, this DF is not for $\calG$ but rather for its subgroup $\calH$,
and so we briefly discuss them using more general notation.

Let $\calV$ be a finite abelian group of order $V$.
A sequence $\set{\calD_r}_{r\in\calR}$ of $R$ subsets of $\calV$, each of cardinality $K$, is a $\DF(V,K,\Lambda)$ for $\calV$ if
\begin{equation}
\label{eq.difference family}
\sum_{r\in\calR}(\tilde{\bfchi}_{\calD_r}*\bfchi_{\calD_r}^{})(v)
=\sum_{r\in\calR}\#\set{(d_1,d_2)\in\calD_r: d_1-d_2=v}
=\Lambda,
\quad\forall\ v\in\calV,\, v\neq0.
\end{equation}
That is, we sum (over $r$) the number of ways that $v\neq0$ can be written as a difference of members of $\calD_r$.
(We caution that this classical definition of a DF,
given in~\cite{Wilson72} for example,
is simpler and more restrictive than a common alternative~\cite{AbelB07} that permits ``short blocks.")
Note that summing~\eqref{eq.difference family} over all $v\neq0$ gives
$\Lambda(V-1)
=\sum_{r\in\calR}\#\set{(d_1,d_2)\in\calD_r: d_1-d_2\neq0}
=RK(K-1)$,
meaning that the number $R$ of subsets of $\calV$ that form a $\DF(V,K,\Lambda)$ is necessarily $R=\frac{\Lambda(V-1)}{K(K-1)}$.

For example,
$\set{\set{1,3,9},\set{2,6,5}}$ is a $\DF(13,3,1)$ for $\bbZ_{13}$ since every nonzero member of $\bbZ_{13}$ appears exactly once in the difference tables of $\set{1,3,9}$ and $\set{2,6,5}$, taken together:
\begin{equation*}
\begin{array}{c|ccc}
 -& 1& 3& 9\\
\hline
 1& 0&11& 5\\
 3& 2& 0& 7\\
 9& 8& 6& 0
\end{array}
\qquad
\begin{array}{c|ccc}
 -& 2& 6& 5\\
\hline
 2& 0& 9&10\\
 6& 4& 0& 1\\
 5& 3&12& 0
\end{array}\ .
\end{equation*}
This concept in hand, we give the following characterization of ``harmonic ECTFFs":

\begin{theorem}
\label{thm.ECTFF from DF}
Let $\hat{\calG}$ be the Pontryagin dual of a finite abelian group $\calG$ of order $G$.
Let $\calD$ be a nonempty $D$-element subset of $\calG$,
and define $\set{\bfphi_\gamma}_{\gamma\in\hat{\calG}}$ by \smash{$\bfphi_\gamma\in\bbC^\calD$}, \smash{$\bfphi_\gamma(d):=D^{-\frac12}\gamma(d)$}.
For any subgroup $\calH$ of $\calG$ of order $H$, let \smash{$\calH^\perp$} be its annihilator and define $\calD_g:=\calH\cap(\calD-g)$ for any $g\in\calG$.

Then \smash{$\set{\bfphi_{\gamma'}}_{\gamma'\in\gamma\calH^\perp}$} is orthonormal for any or all $\gamma\in\hat{\calG}$ if and only if the cardinality of \smash{$\calD_g$} is independent of $g$.
In this case, \smash{$\set{\calU_{\gamma\calH^\perp}}_{\gamma\calH^\perp\in\hat{\calG}/\calH^\perp}$},
\smash{$\calU_{\gamma\calH^\perp}:=\Span\set{\bfphi_{\gamma'}}_{\gamma'\in\gamma\calH^\perp}$} is a $\TFF(D,H,\frac GH)$ for $\bbC^\calD$ that is equichordal if and only if $\set{\calD_g}_{g+\calH\in\calG/\calH}$ is a difference family for $\calH$,
and is moreover equi-isoclinic if and only if each $\calD_g$ is a difference set of $\calH$.

In particular, if $\calD$ is a semiregular DDS for $\calG$,
then $\set{\calD_g}_{g+\calH\in\calG/\calH}$ is a difference family for $\calH$.
\end{theorem}

\begin{proof}
To simplify notation, for any $\gamma\in\hat{\calG}$ we re-index \smash{$\set{\bfphi_{\gamma'}}_{\gamma'\in\gamma\calH^\perp}$} as \smash{$\set{\bfphi_{\gamma\eta}}_{\eta\in\calH^\perp}$},
and let $\bfPhi_\gamma$ be its $\calD\times\calH^\perp$ synthesis operator.
That is, we index the subsequence of the harmonic frame \smash{$\set{\bfphi_\gamma}_{\gamma\in\hat{\calG}}$} corresponding to the $\gamma$th coset $\gamma\calH^\perp$ of $\calH^\perp$ by $\calH^\perp$ rather than by its coset.
Doing so simply permutes its vectors, and so has no effect on its orthonormality, its span \smash{$\calU_{\gamma\calH^\perp}$}, or the singular values of its cross-Gram matrix with another such subsequence.
For any $\gamma_1,\gamma_2\in\hat{\calG}$,
the latter is the $\calH^\perp\times\calH^\perp$ matrix $\bfPhi_{\gamma_1}^*\bfPhi_{\gamma_2}^{}$ which by~\eqref{eq.harmonic Gram}
has an $(\eta_1,\eta_2)$th entry of
\begin{equation*}
(\bfPhi_{\gamma_1}^*\bfPhi_{\gamma_2}^{})(\eta_1,\eta_2)
=\ip{\bfphi_{\gamma_1\eta_1}}{\bfphi_{\gamma_2\eta}}
=\tfrac1D(\bfGamma^*\bfchi_\calD)(\gamma_1^{}\gamma_2^{-1}\eta_1^{}\eta_2^{-1}).
\end{equation*}
This matrix is $\calH^\perp$-circulant,
and so is diagonalized by the DFT of $\calH^\perp$.
In particular, exploiting the well-known fact that $g+\calH\mapsto(\eta\mapsto\eta(g))$ is a well-defined isomorphism from $\calG/\calH$ onto the dual of $\calH^\perp$,
we can write the spectrum of $\bfPhi_{\gamma_1}^*\bfPhi_{\gamma_2}^{}$ as $\set{\lambda_{\gamma_1,\gamma_2}^{(g+\calH)}}_{g+\calH\in\calG/\calH}$ where
\begin{equation*}
\lambda_{\gamma_1,\gamma_2}^{(g+\calH)}
:=\sum_{\eta\in\calH^\perp}\overline{\eta(-g)}(\bfPhi_{\gamma_1}^*\bfPhi_{\gamma_2}^{})(\eta,1)
=\tfrac1D\sum_{\eta\in\calH^\perp}\eta(g)(\bfGamma^*\bfchi_\calD)(\gamma_1^{}\gamma_2^{-1}\eta).
\end{equation*}
Here for the sake of notational convenience later on, we have without loss of generality opted to correlate the $(g+\calH)$th eigenvalue of $\bfPhi_{\gamma_1}^*\bfPhi_{\gamma_2}^{}$ with the $-g$th character of $\calH^\perp$.
To simplify this expression, we now express $(\bfGamma^*\bfchi_\calD)(\gamma_1^{}\gamma_2^{-1}\eta)$ as a summation and then interchange summations:
\begin{equation*}
\lambda_{\gamma_1,\gamma_2}^{(g+\calH)}
=\tfrac1D\sum_{\eta\in\calH^\perp}\eta(g)\sum_{g'\in\calG}\overline{(\gamma_1^{}\gamma_2^{-1}\eta)(g')}\bfchi_\calD(g')
=\tfrac1D\sum_{g'\in\calG}(\gamma_1^{-1}\gamma_2^{})(g')\bfchi_\calD(g')\sum_{\eta\in\calH^\perp}\eta(g-g').
\end{equation*}
We now simplify the inner sum using the PSF and the fact that $\calH$ is closed under inverses:
\begin{equation*}
\sum_{\eta\in\calH^\perp}\eta(g-g')
=\sum_{\gamma\in\hat{\calG}}\bfchi_{\calH^\perp}(\gamma)\gamma(g-g')
=(\bfGamma\bfchi_{\calH^\perp})(g-g')
=\tfrac GH\bfchi_\calH(g-g')
=\tfrac GH\bfchi_\calH(g'-g).
\end{equation*}
Substituting this into the previous sum,
and making the change of variables $h=g'-g$ gives
\begin{equation*}
\lambda_{\gamma_1,\gamma_2}^{(g+\calH)}
=\tfrac{G}{DH}\sum_{g'\in\calG}(\gamma_1^{-1}\gamma_2^{})(g')\bfchi_\calD(g')\bfchi_\calH(g'-g)
=\tfrac{G}{DH}\sum_{h\in\calG}(\gamma_1^{-1}\gamma_2^{})(h+g)\bfchi_\calD(h+g)\bfchi_\calH(h).
\end{equation*}
To proceed,
we recall that $\calD_g:=\calH\cap(\calD-g)\subseteq\calH$ and identify $\calG/\calH^\perp$ with the dual of $\calH$ via the isomorphism $\gamma\calH^\perp\mapsto(h\mapsto\gamma(h))$.
Under this identification, the character table of $\calH$ becomes the $\calH\times(\calG/\calH^\perp)$ matrix $\bfGamma_\calH$ with entries $\bfGamma_\calH(h,\gamma\calH^\perp)=\gamma(h)$,
and the above expression becomes
\begin{equation}
\label{eq.harmonic ECTFF cross Gram eigenvalues}
\lambda_{\gamma_1,\gamma_2}^{(g+\calH)}
=\tfrac{G}{DH}(\gamma_1^{-1}\gamma_2^{})(g)
\sum_{h\in\calH}\overline{(\gamma_1^{}\gamma_2^{-1})(h)}\bfchi_{\calD_g}(h)
=\tfrac{G}{DH}(\gamma_1^{-1}\gamma_2^{})(g)(\bfGamma_\calH^*\bfchi_{\calD_g})(\gamma_1^{}\gamma_2^{-1}\calH^\perp).
\end{equation}
In particular, for any $\gamma\in\hat{\calG}$
we have that \smash{$\set{\bfphi_{\gamma\eta}}_{\eta\in\calH^\perp}$} is orthonormal if and only if $\bfPhi_\gamma^*\bfPhi_\gamma^{}=\bfI$,
which by~\eqref{eq.harmonic ECTFF cross Gram eigenvalues} is equivalent to having
\begin{equation}
\label{eq.harmonic orthonormality}
1
=\lambda_{\gamma,\gamma}^{(g+\calH)}\tfrac{G}{DH}(\bfGamma_\calH^*\bfchi_{\calD_g})(1\calH^\perp)
=\tfrac{G}{DH}\sum_{h\in\calH}\bfchi_{\calD_g}(h)
=\tfrac{G}{DH}\#(\calD_g)
\end{equation}
for all $g+\calH\in\calG/\calH$.
Note that this condition is independent of $\gamma$:
if \smash{$\set{\bfphi_{\gamma\eta}}_{\eta\in\calH^\perp}$} is orthonormal for any $\gamma\in\hat{\calG}$ then it is orthonormal for all such $\gamma$.
We caution however that while $g+\calD_g=g+[\calH\cap(\calD-g)]=(g+\calH)\cap\calD$ is agnostic with respect to one's choice of coset representative,
which in turn implies that $\#(\calD_g)$ is as well,
the set $\calD_g$ itself is not:
if $g_1+\calH=g_2+\calH$,
then writing $g_2=g_1+h$ for some $h\in\calH$ gives
\begin{equation*}
\calD_{g_2}
=\calH\cap(\calD-g_2)
=\calH\cap(\calD-g_1-h)
=[\calH\cap(\calD-g_1)]-h
=\calD_{g_1}-h,
\end{equation*}
meaning $\calD_{g_1}$ and $\calD_{g_2}$ might differ by a shift.
Regardless, partitioning $\calD$ according to the cosets of $\calH$ gives
$\calD
=\sqcup_{g+\calH\in\calG/\calH}[\calD\cap(g+\calH)]
=\sqcup_{g+\calH\in\calG/\calH}(g+\calD_g)$,
and so $D=\sum_{g+\calH\in\calG/\calH}\#(\calD_g)$.
From this perspective, \eqref{eq.harmonic orthonormality} holds if and only if $\#(\calD_g)$ is constant over all $g\in\calG$,
or equivalently, over any choice of coset representatives of $\calG/\calH$.

For the rest of the proof, we assume that~\eqref{eq.harmonic orthonormality} holds,
namely that \smash{$\set{\bfphi_{\gamma\eta}}_{\eta\in\calH^\perp}$} is orthonormal for all $\gamma\in\hat{\calG}$.
Since \smash{$\set{\bfphi_\gamma}_{\gamma\in\hat{\calG}}$} is a tight frame for $\bbC^\calD$, \smash{$\set{\calU_{\gamma\calH^\perp}}_{\gamma\calH^\perp\in\hat{\calG}/\calH^\perp}$},
\smash{$\calU_{\gamma\calH^\perp}:=\Span\set{\bfphi_{\gamma\eta}}_{\eta\in\calH^\perp}$} is a $\TFF(D,H,\frac GH)$ for $\bbC^\calD$.
To characterize when this TFF is equichordal and equi-isoclinic,
we consider the singular values of off-diagonal cross-Gram matrices.
Here since $\bfPhi_{\gamma_1}^*\bfPhi_{\gamma_2}^{}$ is $\calH^\perp$-circulant it is also normal, meaning the squares of its singular values are simply the squared-moduli of its eigenvalues
$\set{\lambda_{\gamma_1,\gamma_2}^{(g+\calH)}}_{g+\calH\in\calG/\calH}$.
Here since~\eqref{eq.harmonic ECTFF cross Gram eigenvalues} gives
\begin{equation*}
\abs{\lambda_{\gamma_1,\gamma_2}^{(g+\calH)}}^2
=\tfrac{G^2}{D^2H^2}\abs{(\bfGamma_\calH^*\bfchi_{\calD_g})(\gamma_1^{}\gamma_2^{-1}\calH^\perp)}^2
=\tfrac{G^2}{D^2H^2}[\bfGamma_\calH^*(\tilde{\bfchi}_{\calD_g}*\bfchi_{\calD_g})](\gamma_1^{}\gamma_2^{-1}\calH^\perp),
\end{equation*}
we see that our TFF is equi-isoclinic if and only if \smash{$[\bfGamma_\calH^*(\tilde{\bfchi}_{\calD_g}*\bfchi_{\calD_g})](\gamma\calH^\perp)$} is constant over all $g\in\calG$ and \smash{$\gamma\in\hat{\calG}$}, \smash{$\gamma\notin\calH^\perp$}.
Since $\bfGamma_\calH^*(a\bfdelta_0+b\bfchi_\calH)
=a\bfchi_{\hat{\calG}/\calH^\perp}+bH\bfdelta_{1\calH^\perp}$ for any $a,b\in\bbC$,
this occurs if and only if $(\tilde{\bfchi}_{\calD_g}*\bfchi_{\calD_g})(h)$ is constant over all $g\in\calG$, $h\in\calH$, $h\neq0$, namely if and only if each $\calD_g$ is a difference set for $\calH$.
Similarly, our TFF is equichordal if and only if
\begin{equation*}
\norm{\bfPhi_{\gamma\gamma'}^*\bfPhi_{\gamma}^{}}_\Fro^2
=\sum_{g+\calH\in\calG/\calH}\abs{\lambda_{\gamma\gamma',\gamma'}^{(g+\calH)}}^2
=\biggbracket{\bfGamma_\calH^*\biggparen{\,\sum_{g+\calH\in\calG/\calH}\tilde{\bfchi}_{\calD_g}*\bfchi_{\calD_g}}}(\gamma\calH^\perp)
\end{equation*}
is constant over all $\gamma\notin\calH^\perp$.
This occurs if and only if
$\sum_{g+\calH\in\calG/\calH}(\tilde{\bfchi}_{\calD_g}*\bfchi_{\calD_g})(h)$ is constant over all $h\neq 0$,
which by~\eqref{eq.difference family} equates to $\set{\calD_g}_{g+\calH\in\calG/\calH}$ being a difference family for $\calH$.
To be clear, this family consists of one set $\calD_g$ for each distinct coset of $\calH$, where $g$ can be any representative from that coset:
if $g_1$ and $g_2$ differ by an element of $\calH$,
the sets $\calD_{g_1}$ and $\calD_{g_2}$ differ by at most a shift,
and so the autocorrelations of $\bfchi_{\calD_{g_1}}$ and $\bfchi_{\calD_{g_2}}$ are identical.

In particular, when $\calD$ is a semiregular DDS for $\calG$ then,
as proven in~\cite{King16} and discussed above,
\smash{$\set{\calU_{\gamma\calH^\perp}}_{\gamma\calH^\perp\in\hat{\calG}/\calH^\perp}$} is an ECTFF for $\bbC^\calD$,
and so $\set{\calD_g}_{g+\calH\in\calG/\calH}$ is necessarily a DF for $\calH$.
\end{proof}

The previous result can be used to construct an ECTFF from any difference family.
To elaborate,
for any $\DF(V,K,\Lambda)$ for some abelian group $\calV$,
let $\calG=\calV\times\calR$ where $\calR$ is an abelian group of order \smash{$R=\frac{\Lambda(V-1)}{K(K-1)}$}.
Index the sets of this DF for $\calV$ as $\set{\tilde{\calD}_r}_{r\in\calR}$,
let $\calH=\calV\times\set{0}$,
and let $\calD=\sqcup_{r\in\calR}(\tilde{\calD}_r\times\set{r})$.
Following the recipe of Theorem~\ref{thm.ECTFF from DF} then recovers an isomorphic copy $\set{\tilde{\calD}_r\times\set{0}}_{r\in\calR}$ of this DF $\set{\tilde{\calD}_r}_{r\in\calR}$ for our isomorphic copy $\calH$ of $\calV$:
choosing $\set{(0,r)}_{r\in\calR}$ to be the coset of representatives of $\calG/\calH=(\calV\times\calR)/(\calV\times\set{0})$, we have
\begin{equation*}
\calD_{(0,r)}
=\calH\cap[\calD-(0,r)]
=(\calV\times\set{0})\cap\bigbracket{\sqcup_{r'\in\calR}(\tilde{\calD}_{r'}\times\set{r'-r})}
=\tilde{\calD}_r\times\set{0}
\end{equation*}
for any $r\in\calR$.
Applying Theorem~\ref{thm.ECTFF from DF} to $\calD$ thus yields an $\ECTFF(D,N,R)$ where \begin{equation}
\label{eq.harmonic ECTFF params from DF params}
D=\#(\calD)=KR=\tfrac{\Lambda(V-1)}{K-1},
\quad
N=\#(\calH)=\#(\calV)=V,
\quad
R=\tfrac{G}{H}=\tfrac{\Lambda(V-1)}{K(K-1)}.
\end{equation}
For example,
since $\set{\set{1,3,9},\set{2,6,5}}$ is a $\DF(13,3,1)$ for $\bbZ_{13}$,
applying Theorem~\ref{thm.ECTFF from DF} to the subset   $\calD=\set{(1,0),(3,0),(9,0),(2,1),(6,1),(5,1)}$ of $\bbZ_{13}\times\bbZ_2$ yields an $\ECTFF(6,13,2)$ for $\bbC^\calD$.
This fact allows us to produce numerous ECTFFs from classical constructions of DFs.
For example,
for any positive integers $Q\geq K$ and $\Lambda$ such that $Q$ is a prime power and $K(K-1)$ divides $\Lambda(Q-1)$,
Theorems~7, 8 and~5 of \cite{Wilson72} give $\DF(Q,K,\Lambda)$ for the additive group of $\bbF_Q$ in the following three cases, respectively:
when $2\Lambda$ is divisible by either $K-1$ or $K$, when $\Lambda\geq K(K-1)$,
or when $Q>[\frac12K(K-1)]^{K(K-1)}$.
Combining these facts with~\eqref{eq.harmonic ECTFF params from DF params} immediately yields the following result:

\begin{corollary}
\label{cor.ECTFFs from known DFs}
Any $\DF(V,K,\Lambda)$ yields an
$\ECTFF(\tfrac{\Lambda(V-1)}{K-1}, V, \tfrac{\Lambda(V-1)}{K(K-1)})$.
In particular, for any prime power $Q$ and positive integers $K$ and $R$ such that $K\leq Q$, an $\ECTFF(KR,Q,R)$ exists whenever both $\frac{RK(K-1)}{Q-1}\in\bbZ$ and any one of the following three additional conditions is met~\cite{Wilson72}:
\begin{enumerate}
\renewcommand{\labelenumi}{(\roman{enumi})}
\item
either $\tfrac{2RK}{Q-1}\in\bbZ$ or $\tfrac{2R(K-1)}{Q-1}\in\bbZ$,\smallskip
\item
$R\geq Q-1$,\smallskip
\item
$Q>[\frac12K(K-1)]^{K(K-1)}$.
\end{enumerate}
\end{corollary}
To be clear, many other $\DF(V,K,\Lambda)$ are known~\cite{AbelB07},
and we leave it to the interested reader to derive the
$(D,N,R)=(\tfrac{\Lambda(V-1)}{K-1}, V, \tfrac{\Lambda(V-1)}{K(K-1)})$ parameters of the resulting ECTFFs.

We conclude this section by noting that Corollary~\ref{cor.ECTFFs from known DFs} properly generalizes King's aforementioned method of constructing an ECTFF from a semiregular DDS.
When $\calD$ is a semiregular $\DDS(\frac GH,H,D,\Lambda_1,\Lambda_2)$ for $\calG$,
Theorem~\ref{thm.ECTFF from DF} produces an \smash{$\ECTFF(D,H,\frac GH)$} for $\bbC^\calD$ in the exact same manner as~\cite{King16}.
As noted in Theorem~\ref{thm.ECTFF from DF}, it more generally produces such an $\ECTFF$ whenever $\set{\calD_g}_{g+\calH\in\calG/\calH}$ is a $\DF(V,K,\Lambda)$ for $\calH$.
Here, since this DF is obtained by ``partitioning" $\calD$ into \smash{$R=\frac GH$} subsets of $\calH$ of equal cardinality, we necessarily have
\begin{equation*}
V=H,
\quad
K=\tfrac{DH}{G},
\quad
\Lambda
=\tfrac{RK(K-1)}{V-1}
=\tfrac1{H-1}\tfrac{G}{H}\tfrac{DH}{G}(\tfrac{DH}{G}-1)
=\tfrac{D(DH-G)}{G(H-1)}.
\end{equation*}
And though, as noted above, the subsets of any $\DF(V,K,\Lambda)$ can be assembled together to form a subset $\calD$ of $\calG$ that ``partitions" in this way,
it is not necessarily a semiregular $\DDS(\frac GH,H,D,\Lambda_1,\Lambda_2)$.
Indeed, as noted above, such a DDS can only exist if both \smash{$\Lambda_1=\tfrac{D(DH-G)}{G(H-1)}$} and \smash{$\Lambda_2=\frac{D^2}{G}$} are integers.
Though $\Lambda_1=\Lambda$ is always an integer here, \smash{$\Lambda_2=\frac{D^2}{G}$} is not.
For example, the aforementioned $\DS(13,3,1)$ yields an $\ECTFF(6,13,2)$ but no semiregular $\DDS(\frac GH,H,D,\Lambda_1,\Lambda_2)$ with $(D,H,\frac GH)=(6,13,2)$ exists since \smash{$\Lambda_2=\frac{D^2}{G}=\frac{6^2}{2(13)}=\frac{18}{13}\notin\bbZ$.}
That is, Corollary~\ref{cor.ECTFFs from known DFs} can produce an $\ECTFF(6,13,2)$ whereas the method of~\cite{King16} cannot.
Intuitively, this is because the method of~\cite{King16} requires the harmonic frame to be nicer than it needs to be,
having cross-Gram matrices with entries of constant modulus,
whereas only ones of constant Frobenius norms are required.

\section{The minimal points of the orbits of known ECTFFs}

In this section we apply the tools of Section~3 to known constructions of (real and/or equichordal and/or equi-isoclinic) $\TFF(D,N,R)$,
including the ``harmonic ECTFFs" of Section~4.
By Theorem~\ref{thm.summary},
the existence of a (real and/or equichordal and/or equi-isoclinic) $\TFF(D,N,R)$ is fully settled whenever $f(D,N,R):=DNR-D^2-NR^2<0$.
This same theorem gives that every $\TFF(D,N,R)$ with $f(D,N,R)\geq0$ arises (via iterated alternating Naimark complements) from a $\TFF(D_0,N,R_0)$ where $(D_0,N,R_0)$ is minimal in the sense of~\eqref{eq.minimal}.
In this section, we develop a catalog of the minimal $(D_0,N,R_0)$ of known $\ECTFF(D,N,R)$ with $f(D,N,R)\geq0$.
This provides a means of certifying the novelty of a newly discovered ECTFF:
if $f(D,N,R)\geq0$ and its corresponding minimal $(D_0,N,R_0)$ does not yet appear in this catalog,
then this ECTFF is new, being unobtainable from any previously known example in a Naimark-spatial way.
(In fact, by Theorem~\ref{thm.f>0}, taking iterated alternating Naimark and spatial complements of any such ECTFF with $N>4$ yields an infinite family of them,
every one of which is new.)
We begin by revisiting the TFF existence test of~\cite{CasazzaFMWZ11} from this new perspective.
We find that a real $\TFF(D,N,R)$ exists whenever $f(D,N,R)\geq0$,
and so also whenever a complex one does.
We then compile our catalog.

\subsection{The existence of real TFFs}

The TFF existence test of~\cite{CasazzaFMWZ11} fully characterizes the existence of
(complex) $\TFF(D,N,R)$.
In short, it turns out that certain $(D,N,R)$ are ``ambiguous" in the sense that they satisfy certain necessary conditions on the existence of a $\TFF(D,N,R)$ yet fail to meet slightly stronger conditions that suffice to guarantee existence.
It turns out by taking iterated alternating Naimark and spatial complements, one can always convert an ambiguous $(D,N,R)$ into an unambiguous one.
Here, the sufficient conditions rely on explicit construction,
specifically one that involves complex modulation of a real matrix produced via \textit{spectral tetris}.
In the following result, we reinterpret this characterization from the perspective of Section~3, yielding a simpler existence test that exploits spectral tetris in a new way that preserves realness:

\begin{theorem}
\label{thm.TFF existence}
For any $(D,N,R)\in\bbZ^3$ with $D>0$, $N>1$, $R>0$,
a $\TFF(D,N,R)$ exists if and only if either $f(D,N,R):=DNR-D^2-NR^2\geq0$ or there exists $(D_0,N,R_0)\in\Orb(D,N,R)$ with $0<R_0\in\set{\frac{D_0}N,D_0}$.
Moreover, if a $\TFF(D,N,R)$ exists then a real $\TFF(D,N,R)$ exists.
\end{theorem}

\begin{proof}
If a $\TFF(D,N,R)$ exists and $f(D,N,R)<0$ then Theorem~\ref{thm.summary} gives that it can be obtained in the Naimark-spatial way from some $\TFF(D_0,N,R_0)$ with $R_0\in\set{\frac{D_0}N,D_0}$,
immediately implying $(D_0,N,R_0)\in\Orb(D,N,R)$ and $R_0>0$.

Having proven the ``only if" part of the result,
note that in order to prove both its ``if" part as well as its final claim,
it suffices to show that a real $\TFF(D,N,R)$ exists whenever either $f(D,N,R)\geq0$ or $\Orb(D,N,R)$ contains some point $(D_0,N,R_0)$ that satisfies $0<R_0\in\set{\frac{D_0}N,D_0}$.
We consider the latter case first.
Here a trivial $\TFF(D_0,N,R_0)$ exists, $(D,N,R)\in\Orb(D_0,N,R_0)$
and $f(D,N,R)=f(D_0,N,R_0)<0$
since either $f(NR_0,N,R_0)=-NR_0^2$ or $f(R_0,N,R_0)=-R_0^2$.
(In particular, this can only happen when either $-f(D,N,R)$ or $-\frac1N f(D,N,R)$ is a positive perfect square.)
If $N=2$ then, depending on whether $R_0=\frac{D_0}2$ or $R_0=D_0$,
\eqref{eq.orbit when N=2} becomes either
\begin{align*}
\Orb(2R_0,2,R_0)
&=\set{
(2R_0,2,R_0),
(0,2,R_0),
(0,2,-R_0),
(-2R_0,2,-R_0)},\\
\Orb(R_0,2,R_0)
&=\set{
(R_0,2,R_0),
(R_0,2,0),
(-R_0,2,0),
(-R_0,2,-R_0)}.
\end{align*}
Since $D>0$, $R>0$ and $(D,N,R)\in\Orb(D_0,N,R_0)$,
we thus have $(D,N,R)$ is either $(2R_0,2,R_0)$ or $(R_0,2,R_0)$.
In either case, a trivial real $\TFF(D,N,R)$ exists.
Similarly if $N=3$ and $R_0=\frac{D_0}3$ then~\eqref{eq.orbit when N=3} gives that $\Orb(3R_0,3,R_0)$ is
\begin{equation*}
\set{(3R_0,3,R_0),
(0,3,R_0),
(0,3,-R_0),
(-3R_0,3,-R_0),
(-3R_0,3,-2R_0),
(3R_0,3,2R_0)},
\end{equation*}
implying $(D,N,R)$ is either $(3R_0,3,R_0)$ or $(3R_0,3,2R_0)$,
and so by Theorem~\ref{thm.N=2,3} that a real $\TFF(D,N,R)$ exists.
When $N=3$ and $R_0=D_0$, \eqref{eq.orbit when N=3} instead gives that $\Orb(R_0,3,R_0)$ is
\begin{equation*}
\set{(2R_0,3,R_0),
(R_0,3,R_0),
(R_0,3,0),
(-R_0,3,0),
(-R_0,3,-R_0),
(-2R_0,3,-R_0)},
\end{equation*}
implying $(D,N,R)$ is either $(2R_0,3,R_0)$ or $(R_0,3,R_0)$,
at which point Theorem~\ref{thm.N=2,3} again implies that a real $\TFF(D,N,R)$ exists.
If $N\geq 4$ then since $f(D_0,N,R_0)<0$,
Theorem~\ref{thm.f<0} gives a unique $(\tilde{D}_0,N,\tilde{R}_0)\in\Orb^+(D_0,N,R_0)$ that satisfies $\tilde{D}_0\leq D'$ and $\tilde{R}_0\leq R'$ for all $(D',N,R')\in\Orb^+(D_0,N,R_0)$.
Here, since $0<R_0\in\set{\frac{D_0}N,D_0}$, a real $\TFF(D_0,N,R_0)$ exists and $(D_0,N,R_0)\in\Orb^+(D_0,N,R_0)$,
at which point Theorem~\ref{thm.f<0} further implies that $(\tilde{D}_0,N,\tilde{R}_0)$ is the only $(D',N,R')\in\Orb^+(D_0,N,R_0)$ with $R'\in\set{\frac{D'}N,D'}$,
and so $(\tilde{D}_0,N,\tilde{R}_0)=(D_0,N,R_0)$.
Here, since $(D,N,R)$ satisfies $D>0$, $R>0$ and $(D,N,R)\in\Orb(D_0,N,R_0)$,
we have $(D,N,R)\in\Orb^+(D_0,N,R_0)$,
and so Theorem~\ref{thm.f<0} implies that a (real) $\TFF(D,N,R)$ can indeed be obtained in the Naimark-spatial way from our real $\TFF(D_0,N,R_0)$.

Next, we show that a real $\TFF(D,N,R)$ exists if $f(D,N,R)=0$.
By Theorem~\ref{thm.f=0} we necessarily have that $N=4$ and $D=2R$.
Here let $\set{\bfdelta_d}_{d=1}^{D}$ be the standard basis for $\calH=\bbR^{2R}$,
and define $\calU_1=\calU_2:=\Span\set{\bfdelta_d}_{d=1,\, d\text{ odd}}^{D}$ and
$\calU_3=\calU_4:=\Span\set{\bfdelta_d}_{d=1,\, d\text{ even}}^{D}$.
A direct calculation then gives that the four corresponding rank-$R$ projections satisfy $\bfP_1+\bfP_2+\bfP_3+\bfP_4
=\bfI+\bfI=2\bfI$,
implying $\set{\calU_n}_{n=1}^4$ is a $\TFF(2R,4,R)$ for $\bbR^{2R}$.

To conclude, we show that a real $\TFF(D,N,R)$ exists if $f(D,N,R)>0$.
In this case, Theorem~\ref{thm.f>0} gives that $\Orb(D,N,R)$ contains a unique point $(D_0,N,R_0)$ that satisfies~\eqref{eq.minimal}.
We now recall Theorem~13 of~\cite{CasazzaFMWZ11},
which for any positive integers $C$ and $M$ such that $2C\leq M$,
gives a unit norm tight frame $\set{\bfpsi_m}_{m=1}^{M}$ for $\bbR^C$ with the property that $\ip{\bfpsi_{m_1}}{\bfpsi_{m_2}}=0$ whenever $m_1-m_2\geq\lfloor\frac MC\rfloor+3$.
(Though that theorem states that this frame is complex,
the remark that immediately follows it in~\cite{CasazzaFMWZ11} notes that its spectral tetris-based proof actually produces a real frame.)
Here, \eqref{eq.minimal} implies that $C:=D_0$ and $M:=NR_0$ are positive integers such that $2C=2D_0\leq NR_0=M$, and so that theorem gives a unit norm tight frame $\set{\bfpsi_m}_{m=1}^{NR_0}$ for $\bbR^{D_0}$ with the property that $\ip{\bfpsi_{m_1}}{\bfpsi_{m_2}}=0$ whenever $m_1-m_2\geq\lfloor\frac{NR_0}{D_0}\rfloor+3$.
Re-index these frame vectors as
\begin{equation}
\label{eq.UNTF}
\set{\bfphi_{n,r}}_{n=1}^N\,_{r=1}^{R_0},
\quad
\bfphi_{n,r}:=\bfpsi_{N(r-1)+n},
\ \forall n=1,\dotsc,N,\ r=1,\dotsc,R_0.
\end{equation}
We now claim that $\set{\bfphi_{n,r}}_{r=1}^{R_0}$ is orthonormal for any $n$.
Here,
note that since $(D_0,N,R_0)$ satisfies~\eqref{eq.minimal},
\smash{$\frac{R_0}{D_0}\leq\frac12$} and \smash{$\frac 2N\leq\tfrac 12$}.
Moreover, at least one of these two inequalities must be strict since otherwise $f(D,N,R)=f(D_0,N,R_0)=f(2R_0,4,R_0)=0$, a contradiction.
Thus,
\begin{equation*}
\tfrac 2N+\tfrac{R_0}{D_0}<1,
\quad
\text{implying}
\quad
N-2>\tfrac{NR_0}{D_0}\geq\lfloor\tfrac{NR_0}{D_0}\rfloor
\quad
\text{and so}
\quad
N\geq\lfloor\tfrac{NR_0}{D_0}\rfloor+3.
\end{equation*}
As such, $\set{\bfphi_{n,r}}_{r=1}^{R_0}$ is indeed orthonormal:
for any $r_1>r_2$,
\begin{equation*}
[N(r_1-1)+n]-[N(r_2-1)+n]
=N(r_1-r_2)
\geq N\geq \lfloor\tfrac{NR_0}{D_0}\rfloor+3
\end{equation*}
and so
$\ip{\bfphi_{n,r_1}}{\bfphi_{n,r_2}}
=\ip{\bfpsi_{N(r_1-1)+n}}{\bfpsi_{N(r_2-1)+n}}
=0$.
Since~\eqref{eq.UNTF} is a tight frame for $\bbR^{D_0}$,
this implies that $\set{\calU_n}_{n=1}^{N}$,
$\calU_n:=\Span\set{\bfphi_{n,r}}_{r=1}^{R_0}$ is a $\TFF(D_0,N,R_0)$ for $\bbR^{D_0}$.
By Theorem~\ref{thm.f>0},
taking iterated alternating Naimark complements of it yields a real $\TFF(D,N,R)$.
\end{proof}

\subsection{Known equi-isoclinic tight fusion frames}

By Theorem~3.6, if an $\EITFF(D,N,R)$ exists and $f(D,N,R)\geq0$ then either it or its Naimark complement has parameters $(D_0,N,R_0)$ that are minimal, that is, satisfy~\eqref{eq.minimal}.
This same result gives that any $\EITFF(D,N,R)$ with $f(D,N,R)<0$ is either trivial, having $R\in\set{\frac DN,D}$, or has $D=(N-1)R$ and so is the Naimark complement of a trivial $\EITFF(R,N,R)$.

As noted in Section~2,
any $\ECTFF(D,N,1)$ is an $\EITFF(D,N,1)$ and corresponds to an $\ETF(D,N)$.
For any $\ETF(D,N)$ that is not an ONB ($D=N$), a sequence of scalars ($D=1$), or the Naimark complement thereof ($D=N-1$, a regular simplex),
the parameters of either it or its Naimark complement satisfy~\eqref{eq.minimal} in the special case where $R_0=1$, namely that $2\leq D_0\leq\frac N2$.
Though such $\ETF(D_0,N)$ seem rare and remain poorly understood, especially in the complex setting,
a number of diverse ways for constructing infinite families of them have now been discovered.
We do not discuss these in detail, but rather refer the reader to~\cite{FickusM16} for a survey and to~\cite{FickusJ19,FickusM21} for more recent developments.

There is a simple technique for converting these (or more generally any) EITFFs into other EITFFs for spaces of larger dimension:
if \smash{$\set{\calU_n^{(i)}}_{n=1}^N$} is a (real and/or equichordal and/or equi-isoclinic) $\TFF(D_i,N,R_i)$ for $\calH_i$ for both $i=1$ and $i=2$ and $\frac{R_1}{D_1}=\frac{R_2}{D_2}$ then \smash{$\set{\calU_n^{(1)}\times\calU_n^{(2)}}_{n=1}^{N}$} is a (real and/or equichordal and/or equi-isoclinic) $\TFF(D_1+D_2,N,R_1+R_2)$ for $\calH_1\oplus\calH_2$.
Perhaps the easiest way to see this is to let $\calH_i=\bbF^{D_i}$,
let \smash{$\bfPhi_n^{(i)}$} be the $D_i\times R_i$ synthesis operator of an ONB for \smash{$\calU_n^{(i)}$},
and let
$\bfPsi=\left[\begin{array}{ccc}\bfPsi_1&\dotsb&\bfPsi_N\end{array}\right]$ where each $\bfPsi_n$ is the $(D_1+D_2)\times(R_1+R_2)$ matrix defined as
\begin{equation*}
\bfPsi_n=\left[\begin{array}{cc}
\bfPhi_n^{(1)}&\bfzero\\
\bfzero&\bfPhi_n^{(2)}
\end{array}\right],\
\text{ and so }\
\bfPsi\bfPsi^*=\left[\begin{array}{cc}
\frac{NR_1}{D_1}\bfI&\bfzero\\
\bfzero&\frac{NR_2}{D_2}\bfI
\end{array}\right],\
\bfPsi_{n_1}^*\bfPsi_{n_2}^{}
=\left[\begin{array}{cc}
\bigbracket{\bfPhi_{n_1}^{(1)}}^*\bfPhi_{n_2}^{(1)}&\bfzero\\
\bfzero&\bigbracket{\bfPhi_{n_1}^{(2)}}^*\bfPhi_{n_2}^{(2)}
\end{array}\right].
\end{equation*}
Here, having $\frac{R_1}{D_1}=\frac{R_2}{D_2}$ ensures that
\smash{$\set{\calU_n^{(1)}\times\calU_n^{(2)}}_{n=1}^{N}$} is a TFF,
and moreover, when $\set{\calU_n^{(i)}}_{n=1}^N$ is an $\EITFF(D_i,N,R_i)$ for $i\in\set{1,2}$,
that the singular value \smash{$[\tfrac{NR_i-D_i}{D_i(N-1)}]^{\frac12}$} of \smash{$\bigbracket{\bfPhi_{n_1}^{(i)}}^*\bfPhi_{n_2}^{(i)}$} is independent of $i$.
Iteratively applying this idea, starting with two copies of a (real and/or equichordal and/or equi-isoclinic) $\TFF(D,N,R)$ yields a ``tensor-type" (real and/or equichordal and/or equi-isoclinic) $\TFF(KD,N,KR)$ for any positive integer $K$.
See~\cite{CasazzaFMWZ11,CalderbankTX15}, for example, for recent generalizations and instances of this classical trick~\cite{LemmensS73b}.
Here, since $\frac1{D^2}f(D,N,R)=N(\frac RD)-1-N(\tfrac RD)^2$,
and since having \smash{$\frac{R_1}{D_1}=\frac{R_2}{D_2}$} implies \smash{$\frac{R_1}{D_1}=\frac{R_2}{D_2}=\frac{R_1+R_2}{D_1+D_2}$},
the parameters of such TFFs satisfy
\begin{equation*}
\tfrac{f(D_1,N,R_1)}{D_1^2}
=\tfrac{f(D_2,N,R_2)}{D_2^2}
=\tfrac{f(D_1+D_2,N,R_1+R_2)}{(D_1+D_2)^2}.
\end{equation*}
In particular, this construction preserves the sign of $f$.
It also preserves minimality since \eqref{eq.minimal} equates to having
$\frac2N\leq\frac{R_0}{D_0}\leq\frac12$.

As a variation of this idea,
Hoggar~\cite{Hoggar77} uses the injective ring homomorphism
$a+\rmi b\mapsto[\begin{smallmatrix}a&-b\\b&\phantom{-}a\end{smallmatrix}]$
(from the field $\bbC$ into the ring of $2\times 2$ real matrices)
to convert complex equi-isoclinic subspaces into real ones.
In fact, one can verify that applying this mapping to every entry of the synthesis operator of a complex (equichordal or equi-isoclinic) $\TFF(D,N,R)$ yields a real (equichordal or equi-isoclinic) $\TFF(2D,N,2R)$.
Similarly, applying the injective ring homomorphisms
\begin{equation*}
a+\rmi b+\rmj c+\rmk d
=(a+\rmi b)+\rmj(c-\rmi d)
\mapsto[\begin{smallmatrix}
a+\rmi b&          -c-\rmi d\\
c-\rmi d&\phantom{-}a-\rmi b
\end{smallmatrix}]
\mapsto\left[\begin{smallmatrix}
\phantom{-}a&          -b&          -c&\phantom{-}d\\
\phantom{-}b&\phantom{-}a&          -d&          -c\\
\phantom{-}c&\phantom{-}d&\phantom{-}a&\phantom{-}b\\
          -d&\phantom{-}c&          -b&\phantom{-}a
\end{smallmatrix}\right],
\end{equation*}
(from the quaternions $\bbH$ into the ring of $2\times 2$ complex and $4\times 4$ real matrices, respectively) to every entry of a quaternionic (equichordal or equi-isoclinic) $\TFF(D,N,R)$ yields a(n equichordal or equi-isoclinic) complex $\TFF(2D,N,2R)$ or real $\TFF(4D,N,4R)$, respectively~\cite{Hoggar77,Waldron20}.
Remarkably, using numerical techniques that do not generalize to the real and complex settings, \cite{CohnKM16} gives computer-assisted proofs of the existence of quaternionic $\ETF(D,N)$ with
\begin{equation}
\label{eq.quaternionic ETF parameters}
\begin{gathered}
D=2,\ N\in\set{5,6},
\qquad
D=3,\ N\in[6,13]\cup\set{15},
\qquad
D=4,\ N\in[8,21],\\
D=5,\ N\in[10,27],
\qquad
D=6,\ N\in[12,34].
\end{gathered}
\end{equation}
In~\cite{CohnKM16}, ETFs and Naimark complements are referred to as \textit{tight simplices} and \textit{Gale duals}, respectively.
Here, we have included some $(D,N)$ pairs that correspond to the Naimark complements of others explicitly given in~\cite{CohnKM16}, specifically $(2,5)$ and $(2,6)$,
and have similarly omitted others.
Explicit formulations of quaternionic $\ETF(2,5)$ and $\ETF(2,6)$ are given in~\cite{EtTaoui20},
and the fact that quaternionic ETFs have Naimark complements was independently rediscovered in~\cite{Waldron20}.

We note that a (real and/or equichordal and/or equi-isoclinic) $\TFF(D,N,R)$ is incapable of being produced by any of the aforementioned techniques if $D$ and $R$ are coprime.
We are only aware of two published techniques that are capable of producing $\EITFF(D,N,R)$ with such parameters.
One involves a modification of Hoggar's $\bbC$-to-$\bbR$ trick:
applying the mapping
$a+\rmi b\mapsto[\begin{smallmatrix}a&\phantom{-}b\\b&-a\end{smallmatrix}]$ to the entries of a complex symmetric conference matrix of size $N$ yields something which, when suitably scaled and summed with the identity,
becomes the ``fusion Gram" matrix of a real $\EITFF(N,N,2)$~\cite{EtTaoui18,BlokhuisBE18}.
Such a conference matrix can be obtained from a real symmetric conference matrix of size $N+1$~\cite{BlokhuisBE18},
which itself equates to a real $\ETF(\frac12(N+1),N+1)$.
These exist,
for example, whenever $N=Q$ is a prime power such that $Q\equiv 1\bmod 4$~\cite{FickusM16}.
The other technique obtains a real $\EITFF(D(2D+1),D(2D-1),3)$ from a quaternionic $\ETF(D,2D^2-D)$ by exploiting the theory of projective $2$-designs~\cite{IversonKM21}.
To date, these requisite quaternionic ETFs have only been discovered when $D\in\set{2,3}$~\cite{CohnKM16,EtTaoui20,Waldron20},
yielding a real $\EITFF(10,6,3)$ (whose Naimark-complementary $\EITFF(8,6,3)$ has minimal parameters) and a real $\EITFF(21,15,3)$.
(Clearly $21$ and $3$ are not coprime,
but nevertheless no real $\EITFF(21,15,3)$ could arise from the aforementioned techniques since $3$ is odd and no real $\ETF(7,15)$ exists~\cite{FickusM16}.)
We summarize these previously known facts as follows:

\begin{theorem}
\label{thm.EITFF existence}
Let $(D_0,N,R_0)\in\bbZ^3$ satisfy~\eqref{eq.minimal}.
An $\EITFF(D_0,N,R_0)$ exists if:
\begin{enumerate}
\renewcommand{\labelenumi}{(\roman{enumi})}
\item
$R_0=1$ and an $\ETF(D_0,N)$ exists~\cite{FickusM16}.
If this ETF is real then so is this EITFF.\smallskip

\item
$(D_0,N,R_0)=(D_1+D_2,N,R_1+R_2)$ when an $\EITFF(D_1,N,R_1)$ and an $\EITFF(D_2,N,R_2)$  with $\frac{R_1}{D_1}=\frac{R_2}{D_2}$ exist.
In particular,
if an $\EITFF(D,N,R)$ exists then an $\EITFF(KD,N,KR)$ exists for any integer $K>0$.
The resulting EITFFs are real if the requisite ones are as well.\smallskip

\item
\cite{Hoggar77,Waldron20}
a complex $\EITFF(\frac{D_0}2,N,\frac{R_0}2)$,
quaternionic $\EITFF(\frac{D_0}4,N,\frac{R_0}4)$,
or a quaternionic $\EITFF(\frac{D_0}2,N,\frac{R_0}2)$ exists.
In the first two cases, the resulting EITFFs are real.\smallskip

For context, see~\eqref{eq.quaternionic ETF parameters} for the parameters of some quaternionic $\EITFF(D,N,1)$ from~\cite{CohnKM16}.\smallskip

\item
\cite{EtTaoui18,BlokhuisBE18}
$(D_0,N,R_0)=(N,N,2)$ and there exists a complex symmetric conference matrix of size $N$,
such as when there exists a real symmetric conference matrix of size $N+1$,
such as when $N\equiv 1\bmod 4$ is a prime power.\smallskip

\item
\cite{IversonKM21}
$(D_0,N,R_0)\in\set{(8,6,3),(21,15,3)}$.
These EITFFs are real.
\end{enumerate}
\end{theorem}

\subsection{Known equichordal tight fusion frames}

By Theorem~\ref{thm.summary} every $\TFF(D,N,R)$ with $f(D,N,R)<0$ is necessarily equichordal,
and moreover if such a TFF exists then a real $\ECTFF(D,N,R)$ exists.
This same result implies that every (real) $\ECTFF(D,N,R)$ with $f(D,N,R)\geq0$  is obtained via iterated alternating Naimark and spatial complements of some (real) $\ECTFF(D_0,N,R_0)$ whose parameters satisfy~\eqref{eq.minimal}.
As such, we now compile a list of known ECTFFs whose parameters are minimal in this sense.
Some such ECTFFs are equi-isoclinic; these were already considered in Theorem~\ref{thm.EITFF existence}.
Moreover, some of the same techniques that preserve tightness, equi-isoclinicity and minimality also preserve equichordality: if a (real) $\ECTFF(D_1,N,R_1)$ and $\ECTFF(D_2,N,R_2)$ with $\frac{R_1}{D_1}=\frac{R_2}{D_2}$ exist then so does a (real) $\ECTFF(D_1+D_2,N,R_1+R_2)$;
if a complex or quaternionic $\ECTFF(D,N,R)$ exists then so does a real or complex $\ECTFF(2D,N,2R)$, respectively~\cite{Waldron20}.

One of the earliest known constructions of ECTFFs is also one of the most impressive, as it achieves equality in Gerzon's bound~\eqref{eq.Gerzon's bound}: \cite{CalderbankHRSS99} gives a real $\ECTFF(P,\frac{P(P+1)}{2},\frac{P-1}{2})$ for any prime $P$ for which there exists a Hadamard matrix of size $\frac{P+1}{2}$.
Its parameters are automatically minimal.

Zauner~\cite{Zauner99} constructs real ECTFFs from \textit{balanced incomplete block designs} (BIBDs).
For integers $V>K\geq 2$ and $\Lambda>0$,
a corresponding $\BIBD(V,K,\Lambda)$ is a $V$-element vertex set $\calV$
along with a multiset $\calB$ of $K$-element subsets of $\calV$ (called \textit{blocks}) with the property that any distinct $v_1,v_2\in\calV$ are contained in exactly $\Lambda$ blocks.
Here, counting arguments reveal that every vertex is contained in exactly \smash{$R=\tfrac{\Lambda(V-1)}{K-1}$} blocks,
and that there are exactly \smash{$B=\tfrac{VR}{K}=\tfrac{\Lambda V(V-1)}{K(K-1)}$} blocks total.
A BIBD's $\calB\times\calV$ incidence matrix $\bfX$ satisfies $\bfX^\rmT\bfX=(R-\Lambda)\bfI+\Lambda\bfJ$ where $\bfJ$ is an all-ones matrix.
The fact that $V=\rank(\bfX^\rmT\bfX)=\rank(\bfX)\leq B$ is called \textit{Fisher's inequality}.
For any $\BIBD(V,K,\Lambda)$, Zauner~\cite{Zauner99} notes that $\set{\calU_v}_{v\in\calV}$, $\calU_v:=\Span\set{\bfdelta_b}_{b\in\set{b'\in\calB: v\in v'}}$ defines a real $\ECTFF(B,V,R)$.
These ECTFFs are never equi-isoclinic, as the principal angles between any two subspaces are $0$ and $\frac\pi2$ with multiplicities $\Lambda>0$ and $R-\Lambda>0$, respectively.
The parameters of any such ECTFF automatically satisfy a part of~\eqref{eq.minimal}: since $K\geq 2$,
\begin{equation*}
2D=2B\leq BK=VR=NR,
\quad
\text{i.e.,}
\quad
D\leq NR-D.
\end{equation*}
Meanwhile, they satisfy $R\leq D-R$ if and only if $2R\leq B$,
or equivalently, $2K\leq V$.
Since the complements of the blocks of any BIBD form another BIBD,
we can assume this to be the case without loss of generality.
(In fact, complementary BIBDs yield spatial-complementary ECTFFs via Zauner's construction.)
That is, if an $\ECTFF(D,N,R)$ arises from Zauner's construction then either $(D,N,R)$ or $\sigma(D,N,R)=(D,N,D-R)$ is minimal,
and so $f(D,N,R)\geq0$.
The BIBD existence literature is vast; see~\cite{MathonR07,AbelG07} for overviews containing many examples.

Next, as detailed in the previous section,
King~\cite{King16} constructs an $\ECTFF(D,H,\frac GH)$ from any semiregular
$\DDS(\frac GH,H,D,\Lambda_1,\Lambda_2)$.
(Recall that the final two parameters of a semiregular DDS are superfluous,
being given by~\smash{$\Lambda_1=\tfrac{D(DH-G)}{G(H-1)}$} and \smash{$\Lambda_2=\frac{D^2}{G}$}.)
For example, applying this to known infinite families of semiregular DDSs discussed in Theorem~2.3.6, Corollary~2.3.8 and Result~2.3.9 of~\cite{Pott95} yields $\ECTFF(D,N,R)$ of the following three types, respectively, for any prime power $Q$ and integers $1\leq I\leq J$:
\begin{enumerate}
\renewcommand{\labelenumi}{(\roman{enumi})}
\item
\smash{$(D,N,R)
=(Q^{I-1}\tfrac{Q^{2J-2I}(Q^I-1)}{Q-1},Q^I,\tfrac{Q^{2J-2I}(Q^I-1)}{Q-1})$},\smallskip
\item
\smash{$(D,N,R)=(LQ^{J-1}\tfrac{Q^J-1}{Q-1},Q^J,\tfrac{Q^J-1}{Q-1})$} provided an $L$-element difference set for $\bbF_Q$ exists,\smallskip
\item
\smash{$(D,N,R)=(2(3^{2J}-3^J),3^{2J},4)$}.
\end{enumerate}
One may also obtain a semiregular DDS by summing a semiregular RDS for $\calG$ with a difference set for $\calH$~\cite{Ionin00}.
We do not report them separately here since the resulting ECTFFs have ``ETF-tensor-ONB" parameters.
As reported in the second version of~\cite{King16},
they are in fact EITFFs that arise by taking the tensor products of the vectors of a harmonic ETF with those of some harmonic mutually unbiased bases that arise in the manner of~\cite{GodsilR09}.

As noted in Theorem~\ref{thm.ECTFF from DF},
every $\ECTFF(D,N,R)$ that arises from a semiregular DDS is an example of one that more generally arises from a DF for $\calH$,
and so has $(D,N,R)=(D,H,\frac GH)$.
By Corollary~\ref{cor.ECTFFs from known DFs},
we moreover have $(D,N,R)=(KR,V,R)$ where $K$ and \smash{$R=\frac{\Lambda(V-1)}{K(K-1)}$} are the cardinality and numbers of subsets that comprise that DF, respectively,
and $V=H$ is the order of $\calH$.
Provided we exclude trivial DFs that consist of singleton sets, these parameters $(D,N,R)$ automatically satisfy a part of~\eqref{eq.minimal}: since $K\geq 2$,
we have $2R\leq KR=D$ and so $R\leq D-R$.
They thus satisfy~\eqref{eq.minimal} in total if and only if $2D=2KR$ is at most $NR=HR$,
namely if and only if $2K\leq H$.
Since the complements of the sets of a DF themselves form a DF,
we can assume this to be the case without loss of generality.
(In fact, complementary DFs yield Naimark complementary ECTFFs via Corollary~\ref{cor.ECTFFs from known DFs}.)

Next, whenever a (real) $\ETF(D,M)$ can be partitioned into regular simplices for their span, then their spans form a (real) $\ECTFF(D,N,R)$ where \smash{$R=[\frac{D(M-1)}{M-D}]^{\frac12}$} and $N=\frac{M}{R+1}$~\cite{FickusJKM18}.
Some of the resulting parameters are explainable by other methods.
For example,
applying this fact to either a Steiner or McFarland harmonic ETF simply recovers an ECTFF that arises from $\BIBD(V,K,1)$ via Zauner's method~\cite{Zauner99}.
Applying it to the Naimark complements of harmonic ETFs arising from certain Singer and twin-prime-power difference sets yields EITFFs with ``ETF-tensor-ONB" parameters~\cite{FickusS20}.
It also yields a non-equi-isoclinic $\ECTFF(\frac12(Q+1)^2,Q,Q+1)$ whenever $Q$ and $Q+2$ are prime powers with $Q\equiv 1\bmod 4$~\cite{FickusS20}.
However, an ECTFF with these parameters arises more generally for any odd prime power $Q$ by applying Corollary~\ref{cor.ECTFFs from known DFs} to a $\DF(Q,\frac{Q+1}{2},\frac{(Q+1)^2}{2})$ of~\cite{Wilson72}.
That said, some of the ECTFFs of~\cite{FickusJKM18} remain novel even in light of recent discoveries:
it gives an $\ECTFF(Q^2-Q+1,Q^2-Q+1,Q)$ for any prime power $Q$ that is moreover real when $Q$ is odd.
For context, Zauner's method produces a real $\ECTFF(Q^2-Q+1,Q^2-Q+1,Q)$ whenever $Q-1$ is a prime power, or perhaps more generally, whenever there exists a projective plane of order $Q-1$.
For the $\ECTFF(D,N,R)$ of~\cite{FickusJKM18}, it seems nontrivial to determine the minimal point $(D_0,N,R_0)$ of $\Orb(D,N,R)$ in general, as the method of obtaining it from $(D,N,R)$ seems to depend on the underlying $\ETF(D,M)$.
As such, we just compute it in the aforementioned special case:
$(Q^2-Q+1,Q^2-Q+1,Q)$ is already minimal provided $Q\geq 3$,
and when $Q=2$, $\Orb(3,3,2)$ contains no minimal point.

Continuing, the theory of \textit{paired difference sets} constructed from quadratic forms over $\bbF_2$ yields a real
$\ECTFF(2^{J-1}(2^J+\varepsilon),2^{2J},\frac13(2^{2J}-1))$  for any integer $J\geq 2$~\cite{FickusIJK20} and $\varepsilon\in\set{1,-1}$.
Taking its spatial complement yields one with minimal parameters
$(2^{J-1}(2^J+\varepsilon),2^{2J},\frac13[2^{J-1}(2^J+3\varepsilon)+1])$.
One of these is an ETF, namely when $J=2$ and $\varepsilon=-1$.

A number of explicit constructions of individual real $\ECTFF(D,N,R)$ are given in~\cite{ConwayHS96,CohnKM16}.
Some of these are subsumed by other constructions:
applying Hoggar's trick to known complex $\ETF(2,4)$, $\ETF(3,9)$ and $\ETF(4,8)$ yields real $\EITFF(4,4,2)$, $\EITFF(6,9,2)$ and $\EITFF(8,8,2)$, cf.~\cite{ConwayHS96};
the $\ECTFF(7,28,3)$ of~\cite{ConwayHS96} generalizes to the aforementioned construction of~\cite{CalderbankHRSS99};
the $\ECTFF(D,N,R)$ of~\cite{CohnKM16} with $(D,N,R)$ being $(5,4,2)$, $(7,4,3)$ or $(8,4,3)$ have $f(D,N,R)<0$ and so arise in the Naimark-spatial way from trivial TFFs;
applying Zauner's method to a $\BIBD(4,2,1)$ (an affine plane of order $2$) yields an $\ECTFF(6,4,3)$, cf.~\cite{CohnKM16}.
Others seem novel even in hindsight,
such as the real $\ECTFF(D,N,R)$ of~\cite{ConwayHS96} with $(D,N,R)$ being $(4,5,2)$, $(4,6,2)$, $(4,10,2)$ or $(8,28,2)$,
and those of~\cite{CohnKM16} with $(D,N,R)$ being $(4,7,2)$ or $(4,8,2)$.
Remarkably~\cite{CohnKM16} also gives computer-assisted proofs, rooted in numerical methods, of the existence of numerous real $\ECTFF(D,N,R)$ with small parameters, including whenever
\begin{align}
\nonumber
&D = 4,\ R = 2,\ N\in[4,6],
&&D = 5,\ R = 2,\ N\in[5,11],
&&D = 6,\ R = 2,\ N\in[6,14],\\
\label{eq.numerical ECTFF}
&D = 6,\ R = 3,\ N\in[5,16],
&&D = 7,\ R = 2,\ N\in[7,17],
&&D = 7,\ R = 3,\ N\in[5,22],\\
\nonumber
&D = 8,\ R = 2,\ N\in[8,21],
&&D = 8,\ R = 3,\ N\in[6,28],
&&D = 8,\ R = 4,\ N\in[5,30].
\end{align}
By our reckoning, around two-thirds of these real ECTFFs are not currently explainable by other known methods.
It is also notable that~(6.3) of~\cite{CohnKM16} only differs from~\eqref{eq.invariant} by a function of $N$.

We summarize these known constructions of ECTFFs with minimal parameters as follows:

\begin{theorem}
\label{thm.ECTFF existence}
Let $(D_0,N,R_0)\in\bbZ^3$ satisfy~\eqref{eq.minimal}.
An $\ECTFF(D_0,N,R_0)$ exists if:
\begin{enumerate}
\renewcommand{\labelenumi}{(\roman{enumi})}
\item
an $\EITFF(D_0,N,R_0)$ exists: see Theorem~\ref{thm.EITFF existence}.
If this EITFF is real then so is this ECTFF.\smallskip

\item
$(D_0,N,R_0)=(D_1+D_2,N,R_1+R_2)$ when an $\ECTFF(D_1,N,R_1)$ and an $\ECTFF(D_2,N,R_2)$  with $\frac{R_1}{D_1}=\frac{R_2}{D_2}$ exist.
In particular,
if an $\ECTFF(D,N,R)$ exists then an $\ECTFF(KD,N,KR)$ exists for any integer $K>0$.
The resulting ECTFFs are real if the requisite ones are as well.\smallskip

\item
\cite{Hoggar77,Waldron20}
a complex $\ECTFF(\frac{D_0}2,N,\frac{R_0}2)$,
quaternionic $\ECTFF(\frac{D_0}4,N,\frac{R_0}4)$,
or a quaternionic $\ECTFF(\frac{D_0}2,N,\frac{R_0}2)$ exists.
In the first two cases, the resulting ECTFFs are real.\smallskip

\item
\cite{CalderbankHRSS99} $(D_0,N,R_0)=(P,\frac{P(P+1)}{2},\frac{P-1}{2})$ where $P$ is prime and a Hadamard matrix of size $\frac{P+1}{2}$ exists.
These ECTFFs are real.\smallskip

\item
\cite{Zauner99}
$(D_0,N,R_0)=(B,V,R)$ for some $\BIBD(V,K,\Lambda)$ where
$B=\tfrac{\Lambda V(V-1)}{K(K-1)}$ and $R=\tfrac{\Lambda(V-1)}{K-1}$.
These ECTFFs are real.\smallskip

\item
(Corollary~\ref{cor.ECTFFs from known DFs})
$(D_0,N,R_0)=(\tfrac{\Lambda(V-1)}{K-1}, V, \tfrac{\Lambda(V-1)}{K(K-1)})$ and a $\DF(V,K,\Lambda)$ exists.

As a subcase, any semiregular
$\DDS(\frac GH,H,D,\Lambda_1,\Lambda_2)$ yields an $\ECTFF(D,H,\frac{G}{H})$~\cite{King16}.\smallskip

\item
\cite{FickusJKM18}
$(D_0,N,R_0)=(Q^2-Q+1,Q^2-Q+1,Q)$ provided $Q\geq 3$ is a prime power.
These ECTFFs are real whenever $Q$ is odd.\smallskip

\item
\cite{FickusIJK20}
$(D_0,N,R_0)=(2^{J-1}(2^J+\varepsilon),2^{2J},\frac13[2^{J-1}(2^J+3\varepsilon)+1])$ for some integer $J\geq 2$ and $\varepsilon\in\set{1,-1}$.
These ECTFFs are real.\smallskip

\item
\cite{ConwayHS96,CohnKM16}
$(D_0,N,R_0)$ is $(4,5,2)$, $(4,6,2)$, $(4,10,2)$, $(8,28,2)$, $(4,7,2)$, $(4,8,2)$ or, more generally, when $(D_0,N,R_0)=(D,N,R)$ is given in~\eqref{eq.numerical ECTFF}.
These ECTFFs are real.
\end{enumerate}
\end{theorem}

We now use this summary to certify the novelty of some of the ECTFFs that arise from DFs via Corollary~\ref{cor.ECTFFs from known DFs}.
(We caution that if $\set{\calD_r}_{r\in\calR}$ is any $\DF(V,K,\Lambda)$ for $\calV$ then $\set{v+\calD_r}_{v\in\calV,\,r\in\calR}$ is the block set of a $\BIBD(V,K,\Lambda)$.
Via~\cite{Zauner99}, it thus gives rise to a real $\ECTFF(D',N',R')$ whose parameters \smash{$D'=\frac{\Lambda V(V-1)}{K(K-1)}$}, $N'=V$ and $R'=\frac{\Lambda(V-1)}{K-1}$ differ from those of the (complex) ECTFF produced in Corollary~\ref{cor.ECTFFs from known DFs} from the same DF.)
This process is best explained by example:
since $Q=19$, $K=3$ and $R=3$ satisfy $\frac{RK(K-1)}{Q-1},\frac{2RK}{Q-1}\in\bbZ$,
Corollary~\ref{cor.ECTFFs from known DFs} yields a complex $\ECTFF(9,19,3)$.
(It arises from the $\DF(19,3,1)$ produced by Theorem~7 of~\cite{Wilson72}.)
Here, note $(D_0,N,R_0)=(9,19,3)$ satisfies~\eqref{eq.minimal}.
(This automatically implies $f(D_0,N,R_0)\geq 0$.
In fact, since $N=19>4$, it automatically implies $f(D_0,N,R_0)>0$.
Here, $f(9,19,3)=261$.)
By the theory of Section~3, specifically Theorem~\ref{thm.f>0},
if such an ECTFF arose in the Naimark-spatial way from any previously known $\ECTFF(D,N,R)$ then $(D_0,N,R_0)=(9,19,3)$ is necessarily the unique minimal point of its orbit $\Orb(D,N,R)$, and so would appear in Theorem~\ref{thm.ECTFF existence} someplace beyond (vi).
As we now detail, this is not the case.

Here, note that if $(9,19,3)=(D_1+D_2,19,R_1+R_2)$ for some positive integers $D_1\leq D_2$, $R_1\leq R_2$ with $\frac{R_1}{D_1}=\frac{R_2}{D_2}$ then necessarily $R_1=1$ and
$\frac{R_1}{D_1}=\frac{R_1+R_2}{D_1+D_2}=\frac{3}{9}=\frac13$, and so $D_1=3$.
However, no $\ECTFF(3,19,1)$ exists since these parameters violate Gerzon's bound~\eqref{eq.Gerzon's bound}.
As such, our $\ECTFF(9,19,3)$ cannot be obtained from ones with smaller parameters via (ii).
Similarly, since $3$ is not even, this ECTFF cannot be obtained via Hoggar's tricks.
(We caution that Corollary~\ref{cor.ECTFFs from known DFs} also yields an $\ECTFF(6,13,2)$ from the $\DF(13,3,1)$ that we used as an example in the previous section, but we cannot certify its novelty since Hoggar's method converts a known quaternionic $\ETF(3,13)$ of~\cite{CohnKM16} into a complex $\EITFF(6,13,2)$.)
Since $(9,19,3)$ is also not consistent with any of the other cases of Theorem~\ref{thm.EITFF existence}, we moreover conclude that no $\EITFF(9,19,3)$ is known.
Altogether, we see no $\ECTFF(9,19,3)$ arises in the manner of (i), (ii) or (iii) of~Theorem~\ref{thm.ECTFF from DF}.
Continuing, no $\ECTFF(D,N,R)$ produced by Corollary~\ref{cor.ECTFFs from known DFs} could also arise via (iv) since $D=KR$ is never prime.
No $\ECTFF(9,13,3)$ arises via (v) since, by Fisher's inequality, no BIBD with $B=9$ and $V=19$ exists.
No $\ECTFF(9,19,3)$ arises from a semiregular $\DDS(3,19,9,\Lambda_1,\Lambda_2)$ since, as detailed in Section~4, such a DDS could only exist if    \smash{$\Lambda_2=\frac{D^2}{G}=\frac{9^2}{3(19)}=\frac{27}{19}$} was an integer.
Moreover, $(D_0,N,R_0)=(9,19,3)$ is clearly not of any form given in (vii), (viii) or (ix).

Thus, the $\ECTFF(9,19,3)$ arising from Corollary~\ref{cor.ECTFFs from known DFs} is certifiably novel.
Whenever $(D_0,N,R_0)$ satisfies~\eqref{eq.minimal} with $N>4$,
and a newly discovered $\ECTFF(D_0,N,R_0)$ is certified to be novel in this manner,
Theorem~\ref{thm.f>0} (along with the fact that two orbits under a group action are either identical or disjoint) implies that a new $\ECTFF(D,N,R)$ exists for every $(D,N,R)$ in the infinite set $\Orb(D_0,N,R_0)$.
That is,
every new such ECTFF actually yields an infinite family of ECTFFs (via iterated alternating Naimark and spatial complements), every member of which is new.
Applying this idea with $(D_0,N,R_0)=(9,19,3)$ for example yields an infinite family of new ECTFFs with the following parameters:
\begin{equation*}
\dotsb
\underset{\rms}{\leftrightarrow}
(105,19,6)
\underset{\rmN}{\leftrightarrow}
(9,19,6)
\underset{\rms}{\leftrightarrow}
\mathbf{(9,19,3)}
\underset{\rmN}{\leftrightarrow}
(48,19,3)
\underset{\rms}{\leftrightarrow}
(48,19,45)
\dotsb.
\end{equation*}

We emphasize that not every ECTFF produced by Corollary~\ref{cor.ECTFFs from known DFs} is novel: as we have already noted, infinite families of semiregular DDSs are known,
meaning the subcase of (vi) is nontrivial.
Moreover, some DFs, such as those consisting of all $I$-dimensional subspaces of a finite vector space $\calV$ of dimension $J$,
yield ECTFFs whose parameters match some of those obtained in~(v).
We leave a deeper investigation of such phenomena for future research.
That said, Corollary~\ref{cor.ECTFFs from known DFs} seems to produce an infinite number of certifiably novel ECTFFs with minimal parameters.
For example, to generalize the aforementioned example with $(D_0,N,R_0)=(9,19,3)$,
let $Q\geq 19$ be any prime power such that $Q\equiv 7\bmod 12$.
(By Dirichlet's theorem on arithmetic progressions, an infinite number of such $Q$ exist.)
We apply Corollary~\ref{cor.ECTFFs from known DFs} with \smash{$K=\frac{Q-1}{6}$} and $R=3$:
since \smash{$\tfrac{RK(K-1)}{Q-1}=\tfrac{Q-7}{12}\in\bbZ$} and $\frac{2RK}{Q-1}=1\in\bbZ$,
it provides an \smash{$\ECTFF(\frac{Q-1}{2},Q,3)$} from a $\DF(Q,\frac{Q-1}{6},\frac{Q-7}{12})$ of~\cite{Wilson72}.
Here, the arguments used above in the $Q=19$ case generalize,
certifying that this ECTFF is novel.
In particular, such an ECTFF cannot arise via (ii) since no $\ETF(D,N)$ with $D=\frac{N-1}{6}>1$ is apparently known~\cite{FickusM16},
and it cannot arise via the subcase of (vi) since the fact that $Q$ and $Q-1$ are coprime implies \smash{$\Lambda_2=\frac{D^2}{G}=\frac{(Q-1)^2}{12Q}\notin\bbZ$}.
From the perspective of constructing certifiably new ECTFFs,
perhaps the greatest weakness of Corollary~\ref{cor.ECTFFs from known DFs} is that every $\ECTFF(D,N,R)$ it produces has $R$ dividing $D$,
which leaves the door open to possible alternative constructions that exploit (ii).
Moving forward, this encourages the search for $\EITFF(D,N,R)$ where $D$ and $R$ are coprime.

We conclude this paper with a result that completely characterizes the existence of (real and/or equichordal and/or equi-isoclinic) $\TFF(2R,4,R)$.
When combined with Theorems~\ref{thm.f=0}, \ref{thm.summary} and~\ref{thm.TFF existence},
this fully settles the existence problem for all (real and/or equichordal and/or equi-isoclinic) $\TFF(D,N,R)$ with $f(D,N,R)\leq 0$,
including all those with $N\in\set{2,3,4}$.
Future work on this problem should thus focus on cases where $f(D,N,R)>0$,
or more precisely by Theorem~\ref{thm.f>0},
$(D_0,N,R_0)$ that satisfy~\eqref{eq.minimal} with $N>4$.

\begin{theorem}
\label{thm.(2R,4,R)}
For any positive integer $R$,
an $\EITFF(2R,4,R)$ and a real $\ECTFF(2R,4,R)$ exist.
Moreover, a real $\EITFF(2R,4,R)$ exists if and only if $R$ is even.
\end{theorem}

\begin{proof}
Hoggar's method converts the well-known complex $\ETF(2,4)$~\cite{FickusM16} into a real $\EITFF(4,4,2)$.
``Tensoring" it yields a real $\EITFF(2R,4,R)$ for any even positive integer $R$.
Meanwhile, ``tensoring" the $\ETF(2,4)$ directly yields a complex $\EITFF(2R,4,R)$ for any positive integer $R$.
Next, applying Zauner's construction to the well-known $\BIBD(4,2,1)$ (the affine plane of order $2$) yields a real $\ECTFF(6,4,3)$.
Combining an arbitrary number of copies of these real $\EITFF(4,4,2)$ and $\ECTFF(6,4,3)$ in the manner of Theorem~\ref{thm.ECTFF existence}.(ii) yields a real $\ECTFF(2R,4,R)$ for any positive integer $R$.
(For a more sophisticated use of this last idea, see the proof of Theorem~8 of~\cite{Wilson72}.)

To conclude, it thus suffices to prove a real $\EITFF(2R,4,R)$ only exists if $R$ is even.
Here, let $\set{\calU_n}_{n=1}^{4}$ be an $\EITFF(2R,4,R)$ for $\bbR^{2R}$.
Without loss of generality, the synthesis operator of the concatenation of the $R$-vector ONBs for its four subspaces is a $2R\times 4R$ matrix of the form
\begin{equation*}
\bfPhi
=\left[\begin{array}{cccc}
\bfPhi_1&\bfPhi_2&\bfPhi_3&\bfPhi_4
\end{array}\right]
=\left[\begin{array}{cccc}
\bfPhi_{1,1}&\bfPhi_{1,2}&\bfPhi_{1,3}&\bfPhi_{1,4}\\
\bfPhi_{2,1}&\bfPhi_{2,2}&\bfPhi_{2,3}&\bfPhi_{3,4}
\end{array}\right]
\end{equation*}
where each $\bfPhi_n=[\begin{smallmatrix}\bfPhi_{1,n}\\\bfPhi_{2,n}\end{smallmatrix}]$ is a $2R\times R$ matrix whose columns form an orthonormal basis for $\calU_n$,
and each $\bfPhi_{m,n}$ is an $R\times R$ matrix.
Since $\bfPhi_1^*\bfPhi_1^{}=\bfI$,
we may apply a $2R\times 2R$ orthogonal matrix on the left of $\bfPhi$ to assume without loss of generality that $\bfPhi_{1,1}=\bfI$ and $\bfPhi_{1,2}=\bfzero$.
Moreover, since $\set{\calU_n}_{n=1}^{N}$ is an $\EITFF(2R,4,R)$,
every cross-Gram matrix $\bfPhi_{n_1}^*\bfPhi_{n_2}^{}$ is an orthogonal $R\times R$ matrix scaled by a factor of
\smash{$[\tfrac{NR-D}{D(N-1)}]^{\frac12}=\frac1{\sqrt{3}}$}.
In particular, for each $n=2,3,4$,
\begin{equation*}
\sqrt{3}\bfPhi_1^*\bfPhi_n^{}
=\sqrt{3}(\bfPhi_{1,1}^*\bfPhi_{1,n}^{}+\bfPhi_{2,2}^*\bfPhi_{2,n}^{})
=\sqrt{3}(\bfI^*\bfPhi_{1,n}^{}+\bfzero^*\bfPhi_{2,n}^{})
=\sqrt{3}\bfPhi_{1,n}
\end{equation*}
is an orthogonal matrix.
For each $n=2,3,4$, this in turn implies that
\begin{equation*}
\bfI
=\bfPhi_n^*\bfPhi_n^{}
=\bfPhi_{1,n}^*\bfPhi_{1,n}^{}+\bfPhi_{2,n}^*\bfPhi_{2,n}^{}
=\tfrac13\bfI+\bfPhi_{2,n}^*\bfPhi_{2,n}^{}
\end{equation*}
and so that $\sqrt{\frac 32}\bfPhi_{2,n}$ an orthogonal matrix.
That is, without loss of generality, $\bfPhi$ is of the form
\begin{equation*}
\bfPhi
=\left[\begin{array}{cccc}
\bfPhi_1&\bfPhi_2&\bfPhi_3&\bfPhi_4
\end{array}\right]
=\left[\begin{array}{crrr}
\bfI&\frac1{\sqrt{3}}\bfU_{1,2}&\frac1{\sqrt{3}}\bfU_{1,3}&\frac1{\sqrt{3}}\bfU_{1,4}\smallskip\\
\bfzero&\sqrt{\frac23}\bfU_{2,2}&\sqrt{\frac23}\bfU_{2,3}&\sqrt{\frac23}\bfU_{2,4}
\end{array}\right]
\end{equation*}
where each $\bfU_{m,n}$ is an orthogonal matrix.
Next, for each $n=2,3,4$, multiplying
$\bfPhi_n$
on the right by $\bfU_{1,n}^*$ corresponds to simply choosing an alternative ONB for $\calU_n$.
Thus, without loss of generality,
\begin{equation*}
\bfPhi
=\left[\begin{array}{cccc}
\bfPhi_1&\bfPhi_2&\bfPhi_3&\bfPhi_4
\end{array}\right]
=\left[\begin{array}{crrr}
\bfI&\frac1{\sqrt{3}}\bfI&\frac1{\sqrt{3}}\bfI&\frac1{\sqrt{3}}\bfI\smallskip\\
\bfzero&\sqrt{\frac23}\bfV_2&\sqrt{\frac23}\bfV_3&\sqrt{\frac23}\bfV_4
\end{array}\right]
\end{equation*}
where $\bfV_2$, $\bfV_3$ and $\bfV_4$ are orthogonal matrices.
Continuing, multiplying $\bfPhi$ on the left by the orthogonal matrix $[\begin{smallmatrix}\bfI&\bfzero\\\bfzero&\bfV_2^*\end{smallmatrix}]$,
we may further assume without loss of generality that
\begin{equation*}
\bfPhi
=\left[\begin{array}{cccc}
\bfPhi_1&\bfPhi_2&\bfPhi_3&\bfPhi_4
\end{array}\right]
=\left[\begin{array}{crrr}
\bfI&\frac1{\sqrt{3}}\bfI&\frac1{\sqrt{3}}\bfI&\frac1{\sqrt{3}}\bfI\smallskip\\
\bfzero&\sqrt{\frac23}\bfI&\sqrt{\frac23}\bfW_3&\sqrt{\frac23}\bfW_4
\end{array}\right]
\end{equation*}
where $\bfW_3$ and $\bfW_4$ are orthogonal matrices.
Since $\set{\calU_n}$ is a $\TFF(2R,4,R)$ we moreover have that
\begin{equation*}
\left[\begin{array}{cc}
2\bfI&\bfzero\\
\bfzero&2\bfI\end{array}\right]
=\bfPhi\bfPhi^*
=\left[\begin{array}{cc}
2\bfI&\frac{\sqrt{2}}{3}(\bfI+\bfW_2^*+\bfW_3^*)\\
\frac{\sqrt{2}}{3}(\bfI+\bfW_2+\bfW_3)&2\bfI
\end{array}\right]
\end{equation*}
namely that $\bfW_3=-(\bfI+\bfW_2)$.
Thus,
\begin{equation}
\label{eq.EITFF(2R,4,R)}
\bfPhi
=\left[\begin{array}{cccc}
\bfPhi_1&\bfPhi_2&\bfPhi_3&\bfPhi_4
\end{array}\right]
=\left[\begin{array}{crrr}
\bfI&\frac1{\sqrt{3}}\bfI&\frac1{\sqrt{3}}\bfI&\frac1{\sqrt{3}}\bfI\smallskip\\
\bfzero&\sqrt{\frac23}\bfI&\sqrt{\frac23}\bfU&-\sqrt{\frac23}(\bfI+\bfU)
\end{array}\right]
\end{equation}
where $\bfU$ and $-(\bfI+\bfU)$ are orthogonal matrices.
Continuing, since $\set{\calU_n}_{n=1}^{N}$ is equi-isoclinic,
each of the following cross-Gram matrices is necessarily an orthogonal matrix scaled by a factor of $\frac1{\sqrt{3}}$:
\begin{align*}
\bfPhi_2^*\bfPhi_3^{}
&=\tfrac13\bfI+\tfrac 23\bfU
=\tfrac13(\bfI+2\bfU),\\
\bfPhi_2^*\bfPhi_4^{}
&=\tfrac13\bfI-\tfrac23(\bfI+\bfU)
=-\tfrac13\bfI-\tfrac 23\bfU
=-\tfrac13(\bfI+2\bfU),\\
\bfPhi_3^*\bfPhi_4^{}
&=\tfrac13\bfI-\tfrac23\bfU^*(\bfI+\bfU)
=-\tfrac13\bfI-\tfrac23\bfU^*
=-\tfrac13(\bfI+2\bfU)^*.
\end{align*}
In summary, $\bfU$, $-(\bfI+\bfU)$ and \smash{$\frac1{\sqrt{3}}(\bfI+2\bfU)$} are orthogonal matrices.
Any (complex) eigenvalue $\lambda$ of $\bfU$ thus satisfies
\begin{equation*}
1=\abs{\lambda}^2,
\quad
1
=\abs{1+\lambda}^2
=1+2\real(\lambda)+\abs{\lambda}^2,
\quad
1
=\abs{\tfrac1{\sqrt{3}}(2+\lambda)}^2
=\tfrac13[4+4\real(\lambda)+\abs{\lambda}^2],
\end{equation*}
and so is necessarily either $\frac12(-1+\sqrt{3}\rmi)$ or its conjugate.
Letting $M$ denote the multiplicity of $\frac12(-1+\sqrt{3}\rmi)$ as an eigenvalue of $\bfU$, we thus have
\begin{equation*}
\Tr(\bfU)
=M\tfrac12(-1+\sqrt{3}\rmi)+(R-M)\tfrac12(-1-\sqrt{3}\rmi)
=-\tfrac 12R+\tfrac{\sqrt{3}}2(2M-R)\rmi.
\end{equation*}
At the same time, $\bfU$ is a real matrix and so $\Tr(\bfU)\in\bbR$,
implying $R=2M$ is even, as claimed.
\end{proof}

We remark that the above proof technique is a generalization of the naive yet effective method for deriving an $\ETF(2,4)$,
and does not seem to easily generalize to situations where either $N>4$ or $D\neq 2R$.
In fact, to our knowledge, Theorem~\ref{thm.(2R,4,R)} is the only result which disproves the existence of a real $\EITFF(D_0,N,R_0)$ in a situation where $(D_0,N,R_0)$ satisfies both~\eqref{eq.minimal} and~\eqref{eq.Gerzon's bound}, namely where $0<4R_0\leq 2D_0\leq NR_0$ and $N\leq \frac12D_0(D_0+1)$.

\section*{Acknowledgments}
 The authors thank Profs.\ Peter~G.~Casazza and Dustin~G.~Mixon for their thoughtful comments.
The views expressed in this article are those of the authors and do not reflect the official policy or position of the United States Air Force, Department of Defense, or the U.S.~Government.

\end{document}